\documentclass[12pt]{amsart}
\usepackage {amscd}
\usepackage[german,english]{babel}
\usepackage {amsfonts}
\usepackage[all]{xy}
\usepackage{amssymb,amsmath, amsthm,latexsym,verbatim}
\usepackage{stmaryrd}
\usepackage{a4wide}
\usepackage[hypertex]{hyperref}
\usepackage{enumitem}

\renewcommand{\Re}{\operatorname{Re}}
\newcommand{\C}{\mathbb{C}}
\newcommand{\R}{\mathbb{R}}
\newcommand{\Z}{\mathbb{Z}}
\newcommand{\Q}{\mathbb{Q}}

\newcommand{\N}{\mathbb{N}}

\newcommand{\s}{\operatorname{s}}

\newcommand{\dist}{\operatorname{dist}}
\newcommand{\rk}{\operatorname{rk}}

\newcommand{\kL}{\mathfrak{k}}
\newcommand{\mL}{\mathfrak{m}}
\newcommand{\aL}{\mathfrak{a}}
\newcommand{\gL}{\mathfrak{g}}
\newcommand{\pL}{\mathfrak{p}}
\newcommand{\nL}{\mathfrak{n}}

\newcommand{\bL}{\mathfrak{b}}

\newcommand{\Gl}{\operatorname{GL}}
\newcommand{\pr}{\operatorname{pr}}
\newcommand{\Rep}{\operatorname{Rep}}
\newcommand{\Id}{\operatorname{Id}}

\newcommand{\Spin}{\operatorname{Spin}}

\newcommand{\reg}{\operatorname{reg}}

\newcommand{\Ad}{\operatorname{Ad}}
\newcommand{\Sl}{\operatorname{SL}}

\newcommand{\SO}{\operatorname{SO}}
\newcommand{\Iim}{\operatorname{Im}}
\newcommand{\Mat}{\operatorname{Mat}}
\newcommand{\inj}{\operatorname{inj}}

\newcommand{\modo}{\operatorname{mod}}
\newcommand{\spec}{\operatorname{spec}}
\newcommand{\CC}{\operatorname{C}}

\newcommand{\Symm}{\operatorname{Sym}}

\newcommand{\vol}{\operatorname{vol}}
\newcommand{\Tr}{\operatorname{Tr}}

\newcommand{\tr}{\operatorname{tr}}
\newcommand{\End}{\operatorname{End}}
\newcommand{\Real}{\operatorname{Re}}
\newcommand{\un}{\operatorname{un}}
\newcommand{\sym}{\operatorname{sym}}
\newtheorem {thrm}{Theorem}[section]
\newtheorem {prop}[thrm] {Proposition}
\newtheorem {lem}[thrm] {Lemma}
\newtheorem {kor}[thrm]{Corollary}
\theoremstyle{definition}
\newtheorem{defn}[thrm] {Definition}
\theoremstyle{remark}
\newtheorem {bmrk}[thrm] {Remark}

\begin{document}
\title[]{The analytic torsion and its asymptotic behaviour for sequences 
of hyperbolic manifolds of finite volume}

\date{\today}

\author{Werner M\"uller}
\address{Universit\"at Bonn\\
Mathematisches Institut\\
Endenicher Allee 60\\
D -- 53115 Bonn, Germany}
\email{mueller@math.uni-bonn.de}

\author{Jonathan Pfaff}
\address{Universit\"at Bonn\\
Mathematisches Institut\\
Endenicher Alle 60\\
D -- 53115 Bonn, Germany}
\email{pfaff@math.uni-bonn.de}


\begin{abstract}
In this paper we study the regularized analytic torsion of finite volume 
hyperbolic manifolds. We consider sequences of coverings $X_i$ of a fixed 
hyperbolic orbifold $X_0$. Our main result is that for
certain sequences of coverings and strongly acyclic flat bundles, the analytic 
torsion divided by the index
of the covering, converges to the $L^2$-torsion. Our results apply to certain 
sequences of arithmetic groups, in particular to sequences of principal 
congruence subgroups of $\SO^0(d,1)(\Z)$ and 
to sequences of principal congruence subgroups or Hecke subgroups of Bianchi 
groups. 
\end{abstract}

\maketitle

\section{Introduction}
\setcounter{equation}{0}
The aim of this paper is to extend the results of Bergeron and Venkatesh
\cite{BV}
on the asymptotic equality of analytic and $L^2$-torsion for strongly 
acyclic representations from the compact to the finite volume case. 

Therefore, we shall first recall the results of Bergeron and Venkatesh 
about the compact case. 
Let $G$ be a semisimple Lie group of non-compact type. Let $K$ be a maximal
compact subgroup of $G$ and let $\widetilde X=G/K$  be the associated Riemannian
symmetric space endowed with a $G$-invariant metric. Let $\Gamma\subset G$ be 
a co-compact discrete subgroup.
For simplicity we assume that $\Gamma$ is torsion free. Let $X:=\Gamma\backslash
\widetilde X$. Then $X$ is a compact locally symmetric manifold of non-positive
curvature. Let $\tau$ be an irreducible  finite dimensional complex
representation of $G$. Let $E_\tau\to X$ be the flat vector bundle associated
to the restriction of $\tau$ to $\Gamma$. By \cite{MtM}, $E_\tau$ can be
equipped
with a canonical Hermitian fibre metric, called admissible, which is unique up 
to scaling. Let $\Delta_P(\tau)$ be the Laplace operator on $E_\tau$-valued
$p$-forms with respect to the metric on $X$ and in $E_\tau$. Let 
$\zeta_p(s;\tau)$ be the zeta function of $\Delta_p(\tau)$ (see \cite{Sh}).
Then the analytic torsion $T_X(\tau)\in\R^+$ is defined by
\begin{equation}\label{analtors1}
T_X(\tau):=\exp\left(\frac{1}{2}\sum_{p=1}^d (-1)^pp\frac{d}{ds}\zeta_p(s;\tau)
\big|_{s=0}\right).
\end{equation}
On the other hand there is the $L^2$-torsion $T^{(2)}_X(\tau)$ (see \cite{Lo}).
Since the heat kernels on $\widetilde X$ are $G$-invariant, one has
\begin{equation}\label{l2-tor}
\log T^{(2)}_X(\tau)=\vol(X) t^{(2)}_{\widetilde X}(\tau),
\end{equation}
where $t^{(2)}_{\widetilde X}(\tau)$ is a constant that depends only on 
$\widetilde X$ and $\tau$. 
It is an interesting problem to see if the $L^2$-torsion can be approximated
by the torsion of finite coverings $X_i\to X$. This problem has been studied
by Bergeron and Venkatesh \cite{BV} under a certain non-degeneracy condition
on $\tau$. Representations which satisfy this condition are called 
{\it strongly acyclic}. One of the main results of \cite{BV} is as follows. 
Let $X_i\to X$, $i\in\N$, be a sequence of finite coverings of $X$. 
Let $\tau$ be strongly acyclic. Let $\inj(X_i)$ denote the injectivety radius
of $X_i$ and assume that $\inj(X_i)\to\infty$ as $i\to\infty$. Then
by \cite[Theorem 4.5]{BV} one has
\begin{equation}\label{limittor}
\lim_{i\to\infty}\frac{\log T_{X_i}(\tau)}{\vol(X_i)}=t^{(2)}_{\widetilde
X}(\tau).
\end{equation}
If $\rk_\C(G)-\rk_\C(K)=1$, one can show that $t^{(2)}_{\widetilde X}(\tau)
\neq0$.
Using the equality of analytic torsion and Reidemeister torsion \cite{Mu2},
Bergeron and Venkatesh \cite{BV} used this result to study the growth of
torsion in the cohomology of cocompact arithmetic groups. Furthermore, recently
P. Scholze \cite{Sch} has shown the existence of Galois representations 
associated with mod $p$ cohomology of locally symmetric spaces for $\Gl_n$ over
a totally real or CM field. This makes it desirable 
to extend these results in various directions. Especially, one would like to 
extend \eqref{limittor} to the finite volume case. However, due to the 
presence of the continuous spectrum of the Laplace operators in the 
non-compact case, one encounters serious technical difficulties in attempting
to generalize  \eqref{limittor} to the finite volume case. In \cite{Ra1} J. 
Raimbault has dealt with finite volume hyperbolic 3-manifolds. In \cite{Ra2}
he applied this to study the growth of torsion in the cohomology for certain
sequences of congruence subgroups of Bianchi groups. His 
result generalized the exponential growth of torsion, obtained in \cite{Pf2} 
for local systems induced from the even symmetric powers of the 
standard representation of $\Sl_2(\C)$, to all strongly acyclic local systems 
and furthermore they implied that the limit of the normalized torsion size 
exists. 
The main
purpose of the present paper is to extend \eqref{limittor} to 
hyperbolic manifolds of finite volume and arbitrary dimension. 
\newline

So from now on we let $G=\Spin(d,1)$, $K=\Spin(d)$ or  $G=\SO^0(d,1)$ and
$K=\SO(d)$ for $d>1$. Then $K$ is a 
maximal compact  subgroup of $G$. Let $\widetilde X=G/K$. Choose an invariant 
Riemannian metric
on $\widetilde X$. If the metric is suitably normalized, $\widetilde X$ is 
isometric to the $d$-dimensional hyperbolic space $\mathbb{H}^d$. Let
$\Gamma\subset
G$ be a torsion free lattice, i.e., $\Gamma$ is a discrete, torsion free 
subgroup with $\vol(\Gamma\backslash G)<\infty$. Let
$X=\Gamma\backslash\widetilde X$. Then
$X$ is an oriented $d$-dimensional hyperbolic manifold of finite volume. 
Let $\tau$ be an irreducible finite dimensional complex representation of $G$
and
let $E_\tau\to X$ be the flat vector bundle associated to $\tau$ as above,
endowed with an admissible Hermitian fibre metric. The first problem is to
define the analytic torsion if $X$ is non-compact, which is the case we are 
interested in. Then the Laplace operator $\Delta_p(\tau)$ 
has a non-empty continuous spectrum and hence, the zeta function 
$\zeta_p(s;\tau)$ can not be defined in the usual way. It requires an
additional regularization. We use the method introduced in \cite{MP2}.
One uses an appropriate height function 
to truncate $X$ at sufficiently high level $Y>Y_0$ to get a 
compact submanifold $X(Y)\subset X$ with boundary $\partial X(Y)$. Let
$K^{p,\tau}(t,x,y)$
be the kernel of the heat operator $\exp(-t\Delta_p(\tau))$. Then it follows 
that there exists $\alpha(t)\in\R$ such that 
$\int_{X(Y)} \tr K^{p,\tau}(x,x,t)\,dx -
\alpha(t)\log Y$ has a limit as $Y\to\infty$. Then we put
\begin{equation}\label{regtrace1}
\Tr_{\reg}\left(e^{-t\Delta_p(\tau)}\right):=\lim_{Y\to\infty}\left(
\int_{X(Y)}\tr K^{p,\tau}(t,x,x)\,dx-\alpha(t)\log Y\right).
\end{equation}

As pointed out in \cite[Remark 5.4]{MP2}, the regularized trace
is not uniquely defined. It depends on the choice of truncation parameters on
the 
manifold $X$. However, if a locally symmetric space 
$X_0=\Gamma_0\backslash\widetilde{X}$ of finite volume is given and if
truncation parameters 
on $X_0$ are fixed, then every locally symmetric manifold 
$X$ which is a finite covering of $X_0$ is canonically equipped with 
truncation parameters: One simply pulls back the truncation on $X_0$ to a
truncation on $X$ via 
the covering map. This will be explained in detail in section \ref{sectr} of the
present paper. 

We remark that we do not assume that the group
$\Gamma_0$ is 
torsion-free. In fact, the typical example for $\Gamma_0$ in the arithmetic case
will be $\Gamma_0=\SO^0(d,1)(\Z)$ 
or $\Gamma_0=\Sl_2(\mathcal{O}_D)$, where $\mathcal{O}_D$ is the ring of 
integers of an imaginary quadratic number field $\mathbb{Q}(\sqrt{-D})$, 
$D\in\mathbb{N}$ being square-free. Then 
$\Gamma$ will denote, for example, a principal congruence subgroup. However, we
assume 
that $\Gamma$ is not only a torsion-free lattice but also that $\Gamma$
satisfies the following condition: For each $\Gamma$-cuspidal parabolic 
subgroup $P^\prime$ of $G$ one has
\begin{equation}\label{assumGamma}
\Gamma\cap P^\prime=\Gamma\cap N_{P^\prime},
\end{equation}
where $N_{P^\prime}$ denotes the nilpotent radical of $P^\prime$. 
This condition holds naturally, for example, for all principal congruence 
subgroups of sufficiently high level.

Let $\theta$ be the
Cartan involution of $G$ with respect to our choice of $K$. Let $\tau_\theta=
\tau\circ\theta$. If $\tau\not\cong\tau_\theta$, it can be shown that 
$\Tr_{\reg}\left(e^{-t\Delta_p(\tau)}\right)$ is exponentially decreasing as
$t\to\
\infty$ and admits an asymptotic expansion as $t\to 0$. Therefore, the 
regularized zeta function $\zeta_p(s;\tau)$ of $\Delta_p(\tau)$ can be defined 
as in the compact case by
\begin{equation}\label{regzeta}
\zeta_p(s;\tau):=\frac{1}{\Gamma(s)}\int_0^\infty 
\Tr_{\reg}\left(e^{-t\Delta_p(\tau)}\right) t^{s-1}\;dt.
\end{equation}
The integral converges absolutely and uniformly on compact subsets of the
half-plane $\Real(s)>d/2$ and admits a meromorphic extension to the whole 
complex plane. The zeta function is regular at $s=0$. So in analogy with the
compact case, the analytic torsion $T_X(\tau)\in\R^+$ can be defined by the
same formula \eqref{analtors1}. 

In even dimensions, $T_X(\tau)$ is rather 
trivial (see \cite{MP2}). 
So we assume that $d=2n+1$, $n\in\N$. To formulate
our main result, we need to introduce some notation. We let $\Gamma_0$ be a 
fixed lattice in $G$ and we let $X_0:=\Gamma_0\backslash\widetilde{X}$. We let 
$\Gamma_i$, $i\in\mathbb{N}$ be a sequence of finite index torsion-free
subgroups of
$\Gamma_0$. 
Then following Raimbault \cite{Ra1}, in definition
\ref{Defcuspun} we define the condition on the sequence $\Gamma_i$ to be
cusp-uniform. This condition is, roughly spoken, a condition on 
the shape of the 2n-tori which form the cross-sections of the cusps of the
manifolds $X_i:=\Gamma_i\backslash\widetilde{X}$. For 
more details, we refer to section \ref{geoms}. We let $\ell(\Gamma_i)$ be the 
length of the shortest closed geodesic on $X_i$. We assume that truncation
parameters 
on the orbifold $X_0$ are fixed and for each $i$ and $\tau$ with
$\tau\neq\tau_\theta$ we define the analytic torsion 
with respect to the induced truncation parameters on $X_i$ as above. 
Then our main
result can be stated as the following theorem. 

\begin{thrm}\label{Mainthrm}
Let $\Gamma_0$ be a lattice in $G$. Let $\Gamma_i$, $i\in\mathbb{N}$ be a
sequence of 
finite-index subgroups of $\Gamma_0$ which is cusp-uniform. Assume that for
$i\geq 1$ the group 
$\Gamma_i$ is 
torsion free and satisfies \eqref{assumGamma}. Let  
$\mathfrak{P}_{\Gamma_i}=\{P_{i,j},\:
j=1,\dots,\kappa(\Gamma_i)\}$ be a set of representatives of
$\Gamma_i$-conjugacy 
classes of $\Gamma_i$-cuspidal parabolic subgroups of $G$ and let $N_{P_{i,j}}$
denote 
the nilpotent radical of $P_{i,j}$. Assume that 
$\lim_{i\to\infty}\ell(\Gamma_i)=\infty$ and that 

\begin{align}\label{condseq}
\lim_{i\to\infty}\frac{1}{[\Gamma_0:\Gamma_i]}\bigl(\kappa(\Gamma_i)+\sum_{j=1}^
{\kappa(\Gamma_i)}\log[
\Gamma_0\cap N_{P_{i,j}}:\Gamma_i\cap N_{P_{i,j}}]\bigr)=0.
\end{align}
Then for $X_i:=\Gamma_i\backslash\widetilde{X}$ and every $\tau$ with
$\tau\neq\tau_\theta$ one has
\[
\lim_{i\to\infty}\frac{\log T_{X_i}(\tau)}{[\Gamma_0:\Gamma_i]}=
t^{(2)}_{\widetilde X}(\tau)\vol(X_0).
\]
\end{thrm}

We remark that the condition \eqref{condseq} is independent of the choice of
$\mathfrak{P}_{\Gamma_i}$. 
Furthermore, one immediately sees that it is satisfied, for
example, if
\begin{align}\label{condnew}
\lim_{i\to\infty}\frac{\kappa(\Gamma_i)\log[\Gamma_0:\Gamma_i]}{[
\Gamma_0:\Gamma_i]}=0.
\end{align}

For hyperbolic $3$-manifolds, Theorem \ref{Mainthrm} was proved by J. Raimbault
\cite{Ra1} under additional assumptions on the intertwining operators. We 
emphasize that we don't need this assumption.
 
For sequences of cusp uniform normal subgroups $\Gamma_i$ of $\Gamma_0$
which 
exhaust $\Gamma_0$, the assumption \eqref{condseq} is easily verified 
and we have the following theorem for the case of normal subgroups.

\begin{thrm}\label{Mainthrm2}
Let $\Gamma_0$ be a lattice in $G$ and let $\Gamma_i$, $i\in\mathbb{N}$, be a
sequence 
of finite-index normal subgroups which is cusp uniform and such that 
each $\Gamma_i$, $i\geq 1$, is torsion-free and satisfies \eqref{assumGamma}. If
$\lim_{i\to\infty}[\Gamma_0:\Gamma_i]=\infty$ and if each
$\gamma_0\in\Gamma_0-\{1\}$ 
only belongs to finitely many $\Gamma_i$, then for each $\tau$ with
$\tau\neq\tau_\theta$ one has  
\begin{align}\label{eqtheo}
\lim_{i\to\infty}\frac{\log T_{X_i}(\tau)}{[\Gamma:\Gamma_i]}=
t^{(2)}_{\widetilde X}(\tau)\vol(X_0).
\end{align}
In particular, if under the same assumptions $\Gamma_i$ is a tower of normal
subgroups,
i.e. $\Gamma_{i+1}\subset
\Gamma_i$ for each $i$ and 
$\cap_i\Gamma_i=\{1\}$, then \eqref{eqtheo} holds.
\end{thrm}

We shall now give applications of our main results to the case of arithmetic 
groups. Firstly let $\Gamma_0:=\SO^0(d,1)(\Z)$. Then $\Gamma_0$ is a lattice
in
$\SO^0(d,1)$. For 
$q\in\mathbb{N}$ let $\Gamma(q)$ be the principal congruence subgroup of
level $q$ (see section \ref{secSO}). Using a result of Deitmar and
Hoffmann \cite{DH}, it follows that the family of principal congruence 
subgroups is cusp uniform (see Lemma \ref{Cuspun2}). Thus, Theorem
\ref{Mainthrm2} implies the following corollary. 

\begin{kor}\label{Korcong1}
For any finite-dimensional irreducible representation $\tau$ of
$\SO^0(d,1)$
with 
$\tau\neq\tau_\theta$ the principal congruence subgroups $\Gamma(q)$, $q\geq 3
$, of
$\Gamma_0:=\SO^0(d,1)(\Z)$ satisfy 
\begin{align*}
\lim_{q\to\infty}\frac{\log T_{X_q}(\tau)}{[\Gamma:\Gamma(q)]}=
t^{(2)}_{\widetilde X}(\tau)\vol(X_0),
\end{align*}
where $X_q:=\Gamma(q)\backslash\mathbb{H}^d$ and
$X_0:=\Gamma_0\backslash\mathbb{H}^d$. 
\end{kor}

Secondly, we give some specific applications in the 3-dimensional case. There is
a natural isomorphism
$\Spin(3,1)\cong\Sl_2(\C)$. If $\rho$ is the
standard-representation of $\Sl_2(\C)$ on $\C^2$, then the finite-dimensional 
irreducible representations of $\Sl_2(\C)$ are given as $\Symm^m\rho\otimes
\Symm^n\overline{\rho}$, $m,n\in\mathbb{N}$. Here
$\Symm^k$ denotes the $k$-th symmetric power and $\overline{\rho}$ denotes the
complex-conjugate representation 
of $\rho$.  One has $(\Symm^m\rho\otimes
\Symm^n\overline{\rho})_\theta=\Symm^n\rho\otimes \Symm^m\overline{\rho}$. For
$D\in\mathbb{N}$ square-free let $\mathcal{O}_D$ be 
the ring of integers of the imaginary quadratic number field $\Q(\sqrt{-D})$
and let $\Gamma(D):=\Sl_2(\mathcal{O}_D)$. Then $\Gamma(D)$ is a lattice 
in $\Sl_2(\C)$. If $\aL$ is a non-zero ideal in $\mathcal{O}_D$, let
$\Gamma(\aL)$ be the 
associated principal congruence subgroup of level $\aL$ (see section
\ref{secSL2}). Then Theorem 
\ref{Mainthrm2} implies the following corollary.

\begin{kor}\label{KorCongr2}
If $\aL_i$ is a sequence
of non-zero ideals in $\mathcal{O}_D$ such that each $N(\aL_i)$ is sufficiently
large and such that 
$\lim_{i\to\infty}N(\aL_i)=\infty$, then for any representation
$\tau=\Symm^n\rho\otimes\Symm^m\bar{\rho}$ 
with $m\neq n$ and for $X_D:=\Gamma(D)\backslash\mathbb{H}^3$ and
$X_i:=\Gamma(\aL_i)\backslash\mathbb{H}^3$ one has 
\begin{align}
\lim_{i\to\infty}\frac{\log T_{X_i}(\tau)}{[\Gamma(D):\Gamma(\aL_i)]}=
t^{(2)}_{\widetilde X}(\tau)\vol(X_D).
\end{align}

\end{kor}

Finally, due to their arithmetic significance, in the 3-dimensional 
case we also want to treat Hecke subgroups of the Bianchi groups. These
groups do not fall directly in the framework of our two main  theorems, since 
their systole does not necessarily tend to infinity if their index in the 
Bianchi groups does. However, a slight modification 
of the proof of our main results will also give the corresponding statement 
for these groups. More precisely, for a non-zero ideal $\aL$ of $\mathcal{O}_D$
let $\Gamma_0(\aL)$ be the 
corresponding Hecke subgroup. Actually, since these 
groups are not torsion-free, we have to take a fixed torsion-free subgroup 
$\Gamma'$ of $\Gamma(D)$ of finite index which satisfies assumption
\eqref{assumGamma}, for 
example a principal congruence subgroup of sufficiently high level, and
consider 
the intersections $\Gamma_0'(\aL):=\Gamma_0(\aL)\cap\Gamma'$. Then we have the
following theorem: 

\begin{thrm}\label{ThrmCongr3}
If $\aL_i$ is a sequence
of non-zero ideals in $\mathcal{O}_D$ such that 
$\lim_{i\to\infty}N(\aL_i)=\infty$, then for any representation
$\tau=\Symm^n\rho\otimes\Symm^m\bar{\rho}$ 
with $m\neq n$ and for $X_D:=\Gamma(D)\backslash\mathbb{H}^3$, 
$X_i':=\Gamma_0'(\aL_i)\backslash\mathbb{H}^3$ one has 
\begin{align}
\lim_{i\to\infty}\frac{\log T_{X_i'}(\tau)}{[\Gamma(D):\Gamma_0'(\aL_i)]}=
t^{(2)}_{\widetilde X}(\tau)\vol(X_D).
\end{align}
\end{thrm}

We shall now outline our method to prove our main results. Let $d=2n+1$. We
assume that the representation $\tau$ is not invariant under the Cartan
involution. To
indicate the dependence of the heat operator, the regularized trace and 
other quantities on the covering $X_i$, we use the subscript $X_i$. Let
\begin{equation}\label{tor-heat}
K_{X_i}(t,\tau):= \frac{1}{2}\sum_{p=1}^d (-1)^p p \Tr_{\reg;X_i}
\left(e^{-t\Delta_{X_i,p}(\tau)}\right).
\end{equation}
As observed above, $K_{X_i}(t,\tau)$ is exponentially decreasing as $t\to\infty$
and admits an asymptotic expansion as $t\to0$. Thus the analytic torsion 
$T_{X_i}(\tau)\in\R^+$ can be defined by
\begin{equation}
\log T_{X_i}(\tau)=\frac{d}{ds}\left(\frac{1}{\Gamma(s)}\int_0^\infty
K_{X_i}(t,\tau) t^{s-1}\;dt\right)\bigg|_{s=0}.
\end{equation}
The integral converges for $\Re(s)>d/2$ and its value at $s=0$ is defined by
analytic continuation. For $T>0$ write
\begin{equation}\label{torsplit}
\log T_{X_i}(\tau)=\frac{d}{ds}\left(\frac{1}{\Gamma(s)}\int_0^T
K_{X_i}(t,\tau) t^{s-1}\;dt\right)\bigg|_{s=0}
+\int_T^\infty K_{X_i}(t,\tau) t^{-1}\;dt.
\end{equation} 
Now we study the behaviour as $i\to\infty$ of the terms on the right hand side.
We start with the second term. Our assumption about $\tau$ implies that the
spectrum of the Laplacians $\Delta_{X_i,p}$, $i\in\N$, have a uniform positive
lower bound. Using the definition \eqref{regtrace} of the regularized trace,
it follows that there exist constants $C_i,c>0$ such that for $t\ge 10$ we
have
\[
|K_{X_i}(t,\tau)|\le C_i e^{-ct}
\]
The problem is to estimate $C_i$. In Proposition \ref{estimregtr}, we will show
that there exists a constant $C$ such that for each $i$ and each $t\geq 10$ 
one has an estimation 
\begin{align}\label{estlongt}
|\Tr_{\reg;X_i}\left(e^{-t\Delta_{X_i,p}(\tau)}\right)|\le C e^{-ct}
\left(\Tr_{\reg;X_i}\left(e^{-\Delta_{X_i,p}(\tau)}\right)+\vol(X_i)\right)
\end{align} 
for each $p=1,\dots,d$. 
This estimate is easy to prove in the compact case and one does not 
need the term $\vol(X_i)$ here. More precisely, if 
$X_i$ is compact and if 
$\lambda_1(i)\le\lambda_2(i)\le\cdots$ 
are the eigenvalues of $\Delta_{X_i,p}(\tau)$, counted with multiplicity, 
then for $t\ge 2$ we have
\[
\Tr\left(e^{-t\Delta_{X_i,p}(\tau)}\right)=\sum_{j=1}^\infty e^{-t\lambda_j(i)}
\le 
e^{-t\lambda_1(i)/2}
\sum_{j=1}^\infty e^{-\lambda_j(i)}=e^{-t\lambda_1(i)/2}
\Tr\left(e^{-\Delta_{X_i,p}(\tau)}\right), 
\]
and the assumption on $\tau$ implies that there is $c>0$ such that $\lambda_1(i)
\ge c$ for all $i\in\N$. 

In the non-compact case, the proof of equation \eqref{estlongt} is more
difficult since 
one also has to deal with the contribution of the continuous spectrum to the
regularized trace, which
is given by the logarithmic derivative of certain intertwining operators. 
The key ingredient of our approach to treat the terms involving the intertwining
operators
is the factorization of the determinant of the intertwining operators, which 
we will study carefully under coverings in section \ref{secfact}. Our main 
result is Theorem \ref{ThrmC}. 

To estimate
$\Tr_{\reg;X_i}\left(e^{-\Delta_{X_i,p}}\right)$
we use that the regularized trace of the heat operator, up to a minor term,
is equal to the spectral side of the  Selberg trace formula applied to the
heat operator (see \cite{MP2}). Then we apply the Selberg trace formula
to express the regularized trace through the geometric side of the trace 
formula. More precisely, let $\widetilde E_\tau$ be the homogeneous vector
bundle over
$\widetilde X=G/K$ associated to $\tau|_K$ and let 
$\widetilde{\Delta}_p(\tau)$ be the Laplacian on
$\widetilde E_\tau$-valued $p$-forms on $\widetilde X$. The heat operator
$e^{-t\widetilde{\Delta}_p(\tau)}$ is a convolution operator with kernel
$H_t^{\nu_p(\tau)}\colon G\to \End(\Lambda^p\pL^*\otimes V_\tau)$. Let
$h_t^{\nu_p(\tau)}(g)=\tr H_t^{\nu_p(\tau)}(g)$, $g\in G$. Then by the trace
formula
we get
\begin{equation}\label{gstr}
\Tr_{\reg;X_i}\left(e^{-t\Delta_{X_i,p}(\tau)}\right)=
I_{X_1}(h_t^{\tau,p})+H_{X_1}
(h_t^{\tau,p})+T'_{X_1}(h_t^{\tau,p})+S_{X_1}(h_t^{\tau,p}),
\end{equation}
where $I_{X_i}$, $H_{X_i}$, $T'_{X_i}$, and $S_{X_i}$ are distributions on $G$ 
associated to the identity, the hyperbolic and the parabolic conjugacy classes
of $\Gamma_i$, respectively. The distributions are described in section 
\ref{geoms}. For example, the identity contribution is given by
\[
I_{X_i}(h_t^{\tau,p})=\vol(X_i)h_t^{\tau,p}(1).
\]
Now we put $t=1$ and estimate each term on the right hand side of
\eqref{gstr}. In this way we can conclude that there exist $C,c>0$ such that
for $t\ge 10$ and all $i\in\N$ we have 
\[
|K_{X_i}(t,\tau)|\le C(\vol(X_i)+\kappa(X_i)+\alpha(X_i))e^{-ct},
\]
where $\alpha(X_i)$ is defined in terms of the lattices associated to the
cross sections of the cusps of $X_i$ (see \eqref{Defalpha}). Using the 
assumptions of Theorem 
\ref{Mainthrm}, we finally get that there exist $C,c>0$ such that
\begin{equation}\label{largetime0}
\frac{1}{\vol(X_i)}\bigg|\int_T^\infty K_{X_i}(t,\tau)t^{-1}\;dt\bigg|\le 
C e^{-cT}
\end{equation}
for all $i\in\N$.  

To deal with the first term on the right hand side of \eqref{torsplit}, put
\begin{equation}
k_t^\tau:=\frac{1}{2}\sum_{p=1}^d (-1)^p p h_t^{\tau,p}.
\end{equation}
Then by \eqref{tor-heat} and \eqref{gstr} we get 
\begin{equation}\label{gmstr1}
K_{X_i}(t,\tau)=I_{X_1}(k_t^\tau)+H_{X_1}
(k_t^\tau)+T'_{X_1}(k_t^\tau)+S_{X_1}(k_t^\tau).
\end{equation}
Now we take the partial Mellin transform of each term on the right hand side,
take its derivative at $s=0$, and study its behaviour as $i\to\infty$. For
the contribution of the identity we get $\vol(X_i)(t_{\widetilde X}^{(2)}(\tau)
+O(e^{-cT}))$.  Using the assumptions of Theorem \ref{Mainthrm}, it follows
that the other terms, divided by  $[\Gamma_0:\Gamma_i]$, converge to $0$.
Thus we get
\begin{equation}\label{limitmell}
\lim_{i\to\infty}\frac{1}{[\Gamma_0:\Gamma_i]}\frac{d}{ds}\left(\frac{1}{
\Gamma(s)}
\int_0^T K_{X_i}(t,\tau) t^{s-1}\;dt\right)\bigg|_{s=0}=\vol(X_0)
(t^{(2)}_{\widetilde X}(\tau)+O(e^{-cT})).
\end{equation}
Combining \eqref{limitmell}, \eqref{torsplit} and \eqref{largetime0}, and
using that $T>0$ is arbitrary, Theorem \ref{Mainthrm} follows.

Theorem \ref{Mainthrm2} is a simple consequence of Theorem \ref{Mainthrm}.
For the corollaries we only need to verify that the assumptions of the main
theorems are satisfied. \newline

The paper is organized as follows. In section \ref{secprel} we fix some 
notation and collect some basic facts. In section \ref{seceis} we recall 
some facts about Eisenstein series and intertwining operators. Section 
\ref{secfact} deals with the factorization of the determinant of the 
$C$-matrix. The main result is Theorem \ref{ThrmC}. In section \ref{secbol}
we consider Bochner-Laplace operators and establish some properties of
their spectrum. In section \ref{sectr} we introduce the regularized trace of the
heat operator using the truncated heat kernel and express it in terms of  
spectral data of the corresponding Laplace operator. Section \ref{secexpdec}
deals with the estimation of the regularized trace of the heat operator for 
large time. The bound obtained in Proposition \ref{estimregtr} involves
the regularized trace of the heat operator at time $t=1$. In section \ref{geoms}
we use the geometric side of the trace formula to study this term in detail.
Of particular importance are the constants obtained from the
contribution of the parabolic conjugacy classes which we need to  estimate 
uniformly
with respect to the covering. In section \ref{secmainres} we prove our
main theorems. In the final sections \ref{secSO} and \ref{secSL2} we apply
our results to derive the corollaries.

\bigskip
{\bf Acknowledgement.}
We would like to thank Tobias Finis for several very helpful explanations 
concerning the Hecke subgroups of the Bianchi groups. 
In particular, Proposition \ref{Props} and its proof are due to Tobias Finis.

\section{Preliminaries}\label{secprel}
\setcounter{equation}{0}
We let $d=2n+1$, $n\in\N$ and we let either $G=\SO^0(d,1)$, $K=\SO(d)$ or
$G=\Spin(d,1)$,
$K=\Spin(d)$. Then $K$ is a maximal compact 
subgroup of $G$ and if the quotient $\widetilde{X}:=G/K$ is equipped with the 
$G$-invariant metric defined by \eqref{metr}, then
$\widetilde{X}$ is 
isometric to the $d$-dimensional hyperbolic space. Let $G=NAK$
be the Iwasawa decomposition of $G$ as in \cite[section 2]{MP2} and let $M$ be
the centralizer of $A$ in
$K$. Let $\gL$, $\nL$, $\aL$, $\kL$, $\mL$ denote the Lie algebras
of $G$, $N$, $A$ $K$ and $M$. 
Fix a Cartan subalgebra $\mathfrak{b}$ of $\mathfrak{m}$. 
Then
\begin{align*}
\mathfrak{h}:=\mathfrak{a}\oplus\mathfrak{b}
\end{align*}
is a Cartan subalgebra of $\mathfrak{g}$. We can identify
$\mathfrak{g}_\C\cong\mathfrak{so}(d+1,\C)$. Let $e_1\in\aL^*$ be the
positive restricted root defining $\mathfrak{n}$.
Then we fix $e_2,\dots,e_{n+1}\in
i\mathfrak{b}^*$ such that 
the positive roots $\Delta^+(\mathfrak{g}_\C,\mathfrak{h}_\C)$ are chosen as in
\cite[page 684-685]{Knapp2}
for the root system $D_{n+1}$. We let
$\Delta^+(\mathfrak{g}_\C,\mathfrak{a}_\C)$ be
the set of roots of $\Delta^+(\mathfrak{g}_\C,\mathfrak{h}_\C)$ which do not
vanish on $\aL_\C$. The positive roots
$\Delta^+(\mathfrak{m}_\C,\mathfrak{b}_\C)$
are chosen such that they are restrictions of elements from
$\Delta^+(\mathfrak{g}_\C,\mathfrak{h}_\C)$.
For $j=1,\dots,n+1$ let
\begin{equation}\label{rho}
\rho_{j}:=n+1-j.
\end{equation}
Then the half-sums of positive roots $\rho_G$ and $\rho_M$, respectively, are
given by
\begin{align}\label{Definition von rho(G)}
\rho_{G}:=\frac{1}{2}\sum_{\alpha\in\Delta^{+}(\mathfrak{g}_{\mathbb{C}},
\mathfrak{h}_\mathbb{C})}\alpha=\sum_{j=1}^{n+1}\rho_{j}e_{j};\quad
\rho_{M}:=\frac{1}{2}\sum_{\alpha\in\Delta^{+}(\mathfrak{m}_{\mathbb{C}},
\mathfrak{b}_{\mathbb{C}})}\alpha=\sum_{j=2}^{n+1}\rho_{j}e_{j}.
\end{align}
Put
\begin{align}\label{metr}
 \left<X,Y\right>_{\theta}:=-\frac{1}{2(d-1)}B(X,\theta(Y)),\quad X,Y\in
\mathfrak{g}.
\end{align}
Let ${{\mathbb{Z}}\left[\frac{1}{2}\right]}^{j}$ be the set of all 
$(k_{1},\dots,k_{j})\in\mathbb{Q}^{j}$ such that either all $k_{i}$ are 
integers or all $k_{i}$ are half integers. 
Let $\Rep(G)$ denote the set of finite dimensional irreducible
representations 
$\tau$ of $G$. These  
are parametrized by their highest weights
\begin{equation}\label{Darstellungen von G}
\Lambda(\tau)=k_{1}(\tau)e_{1}+\dots+k_{n+1}(\tau)e_{n+1};\:\:
k_{1}(\tau)\geq 
k_{2}(\tau)\geq\dots\geq k_{n}(\tau)\geq \left|k_{n+1}(\tau)\right|,
\end{equation}
where $(k_{1}(\tau),\dots, k_{n+1}(\tau))$ belongs to
${{\mathbb{Z}}\left[\frac{1}{2}\right]}^{n+1}$ if 
$G=\Spin(d,1)$ and to ${{\mathbb{Z}}}^{n+1}$ if $G=\SO^0(d,1)$. 
Moreover, the finite dimensional irreducible representations $\nu\in\hat{K}$ of
$K$ are
parametrized by their highest weights
\begin{equation}\label{Darstellungen von K}
\Lambda(\nu)=k_{2}(\nu)e_{2}+\dots+k_{n+1}(\nu)e_{n+1};\:\:
k_{2}(\nu)\geq 
k_{3}(\nu)\geq\dots\geq k_{n}(\nu)\geq k_{n+1}(\nu)\geq 0,
\end{equation} 
where $(k_{2}(\nu),\dots, k_{n+1}(\nu))$ belongs to
${{\mathbb{Z}}\left[\frac{1}{2}\right]}^{n}$ if 
$G=\Spin(d,1)$ and to ${{\mathbb{Z}}}^{n}$ if $G=\SO^0(d,1)$.
Finally, the  finite dimensional irreducible representations 
$\sigma\in\hat{M}$ of $M$ 
are parametrized by their highest weights
\begin{equation}\label{Darstellungen von M}
\Lambda(\sigma)=k_{2}(\sigma)e_{2}+\dots+k_{n+1}(\sigma)e_{n+1};\:\:
k_{2}(\sigma)\geq 
k_{3}(\sigma)\geq\dots\geq k_{n}(\sigma)\geq \left|k_{n+1}(\sigma)\right|,
\end{equation}
where $(k_{2}(\sigma),\dots, k_{n+1}(\sigma))$ belongs to
${{\mathbb{Z}}\left[\frac{1}{2}\right]}^{n}$, if 
$G=\Spin(d,1)$, and to ${{\mathbb{Z}}}^{n}$, if $G=\SO^0(d,1)$.
For $\nu\in\hat{K}$ and  $\sigma\in \hat{M}$ we denote by $[\nu:\sigma]$ the
multiplicity of $\sigma$ in the restriction of $\nu$ to $M$. 

Let $\Omega$, $\Omega_K$ and $\Omega_M$ be the Casimir elements of $G$, $K$
and $M$, respectively, with respect to the inner
product \eqref{metr}. Then by a standard computation one has
\begin{align}\label{Casimir}
\Omega=H_1^2-2nH_1+\Omega_M\quad\mod \nL U(\gL).
\end{align}

Let $M'$ be the normalizer of $A$ in $K$ and let $W(A)=M'/M$ be the 
restricted Weyl-group. It has order two and it acts on the finite-dimensional 
representations of $M$ as follows. Let $w_0\in W(A)$ be the non-trivial 
element and let $m_0\in M^\prime$ be a representative of $w_0$. Given 
$\sigma\in\hat M$, the representation $w_0\sigma\in \hat M$ is defined by
\[
w_0\sigma(m)=\sigma(m_0mm_0^{-1}),\quad m\in M.
\] 
Let
$\Lambda(\sigma)=k_{2}(\sigma)e_{2}+\dots+k_{n+1}(\sigma)e_{n+1}$ be the 
highest weight
of $\sigma$ as in \eqref{Darstellungen von M}. Then the highest weight 
$\Lambda(w_0\sigma)$ of $w_0\sigma$ is given by
\begin{equation}\label{wsigma}
\Lambda(w_0\sigma)=k_{2}(\sigma)e_{2}+\dots+k_{n}(\sigma)e_{n}
-k_{n+1}(\sigma)e_{n+1}.
\end{equation}

Let $P:=NAM$. We equip $\aL$ with the norm induced from the restriction of the
normalized Killing form on $\gL$. Let $H_1\in\aL$ be
the unique vector which is of norm one and such that the positive restricted
root,
implicit in the choice of $N$, is positive on $H_1$. Let
$\exp:\aL\to
A$ be
the exponential map.
Every $a\in A$ can be written 
as $a=\exp{\log{a}}$, where $\log{a}\in\mathfrak{a}$ is unique.
For $t\in\mathbb{R}$, we let 
$a(t):=\exp{(tH_{1})}$. If $g\in G$, we define $n(g)\in N$, $
H(g)\in \R$ and $\kappa(g)\in K$ by 
\begin{align*}
g=n(g)a(H(g))\kappa(g). 
\end{align*}

Now let $P'$ be any parabolic subgroup. Then 
there exists a $k_{P'}\in K$ such that
$P'=N_{P'}A_{P'}M_{P'}$ with $N_{P'}=k_{P'}Nk_{P'}^{-1}$,
$A_{P'}=k_{P'}Ak_{P'}^{-1}$, $M_{P'}=k_{P'}Mk_{P'}^{-1}$. We choose a set of
$k_{P'}$'s, which will be fixed from now on. Let
$k_{P}=1$.
We let $a_{P'}(t):=k_{P'}a(t)k_{P'}^{-1}$. If $g\in G$, we define
$n_{P'}(g)\in N_{P'}$, $H_{P'}(g)\in
\mathbb{R}$ and $\kappa_{P'}(g)\in K$ by 
\begin{align}\label{eqg}
g=n_{P'}(g)a_{P'}(H_{P'}(g))\kappa_{P'}(g) 
\end{align}
and we define an identification $\iota_{P'}$ of
$\left(0,\infty\right)$ with
$A_{P'}$ by $\iota_{P'}(t):=a_{P'}(\log(t))$. 
For $Y>0$, let $A^{0}_{P'}\left[Y\right]:=\iota_{P'}(Y,\infty)$ and
$A_{P'}\left[Y\right]:=\iota_{P'}[Y,\infty)$. For $g\in G$ as in \eqref{eqg} we
let $y_{P'}(g):=e^{H_{P'}(g)}$.

\smallskip
Let $\Gamma$ be a discrete subgroup of $G$ such that
$\vol(\Gamma\backslash G)<\infty$. We do not assume at the moment 
that $\Gamma$ is torsion-free. Let 
$X:=\Gamma\backslash\widetilde{X}$. Let $\pr_X:G\to X$ be the projection. 
A parabolic subgroup $P'$ of $G$ is called a $\Gamma$-cuspidal parabolic
subgroup if $\Gamma\cap N_{P'}$ is a lattice in $N_{P'}$. 
Let $\mathfrak{P}_\Gamma=\{P_1,\dots,P_{\kappa(\Gamma)}\}$ be a set of 
representatives of $\Gamma$-conjugacy-classes of $\Gamma$-cuspidal parabolic 
subgroups of $G$. Then for each $P'\in\mathfrak{P}_{\Gamma}$ one has 
\begin{align}\label{eqPara}
\Gamma\cap P'=\Gamma\cap(M_{P'}N_{P'}). 
\end{align}
The number
\begin{equation}\label{cusps}
\kappa(X):=\kappa(\Gamma)=\#\mathfrak{P}_{\Gamma}
\end{equation}
is finite and equals the number of cusps of $X$. More precisely, for each
$P_i\in\mathfrak{P}_\Gamma$ there exists a $Y_{P_i}>0$ and there exists 
a compact connected subset $C=C(Y_{P_1},\dots,Y_{P_{\kappa(\Gamma)}})$
of $G$  such that in the sense of a disjoint union one has
\begin{align}\label{FBI}
G=\Gamma\cdot
C\sqcup\bigsqcup_{i=1}^{
\kappa(X)}\Gamma\cdot
N_{P_i}A^{0}_{P_i}\left[Y_{P_i}\right]K
\end{align}
and such that 
\begin{align}\label{FBII}
\gamma\cdot N_{P_i}A^0_{P_i}\left[Y_{P_i}\right]K\cap
N_{P_i}A_{P_i}^{0}\left[Y_{P_i}\right]K\neq
\emptyset\Leftrightarrow\gamma\in \Gamma\cap P_i. 
\end{align}
For each $P_i\in\mathfrak{P}_\Gamma$ let 
\begin{equation}\label{infimum}
Y^0_{P_i}(\Gamma)=\inf\{Y_{P_i}\colon  Y_{P_i}\in\R^+\;\;\textup{satisfies}\;\;
\eqref{FBII}\}.
\end{equation}
Moreover, we define the height-function $y_{\Gamma,P_i}$ on $X$ by 
\begin{equation}\label{heightf}
y_{\Gamma,P_i}(x):=\sup\{y_{P_i}(g)\colon g\in G,\; \pr_X(g)=x\}. 
\end{equation}
By \eqref{FBI} and \eqref{FBII} the supremum is finite.
For $Y_1,\dots,Y_{\kappa(X)}\in (0,\infty)$ we let 
\begin{equation}\label{truncmfd}
X(P_1,\dots,P_{\kappa(X)};Y_1,\dots,Y_{\kappa(X)}):=\{x\in X\colon
y_{\Gamma, P_i}(x)\leq Y_i,\:i=1,\dots,\kappa(X)\}.
\end{equation}
If $Y\in (0,\infty)$, we write $X_{\mathfrak{P}_\Gamma}(Y)$ or
$X(P_1,\dots,P_{\kappa(X)};Y)$ for
$X(P_1,\dots,P_{\kappa(X)};Y,\dots,Y)$, i.e. 
\begin{align}\label{XY}
X_{\mathfrak{P}_\Gamma}(Y):=X(P_1,\dots,P_{\kappa(X)};Y) :=\{x\in X\colon
y_{\Gamma,P_i}(x)\leq
Y,\:i=1,\dots,\kappa(X)\}.
\end{align}

For later purposes we now recall the interpretation of the semisimple elements
in terms of closed geodesics. For further details we refer, for example, to 
\cite[section 3]{Pf}.
We let $\Gamma_{\s}$ denote the semisimple elements of $\Gamma$ which are not
$G$-conjugate 
to an element of $K$. By $\CC(\Gamma)_{\s}$ we denote the set of 
$\Gamma$-conjugacy classes of elements of $\Gamma_{\s}$. Then for each
$\gamma\in\Gamma_{\s}$ 
there exists a unique geodesic
$\widetilde{c}_{\gamma}$ in 
$\widetilde{X}$ which is stabilized by $\gamma$. If one lets
\begin{align}\label{Ellgamma}
\ell(\gamma)=\inf_{x\in\widetilde{X}}d(x,\gamma x),
\end{align}
then $\ell(\gamma)>0$ and the 
infimum is attained exactly by the points 
in $\widetilde{X}$ lying on $\widetilde{c}_\gamma$. Let $\mathcal{C}(X)$ denote 
the set of closed geodesics of $X$. For $\gamma\in \Gamma_{\s}$ let
$c_\gamma$ be the projection of the segment of $\widetilde{c}_\gamma$ from $x_0$
to $\gamma x_0$, $x_0$ a point on $\widetilde{c}_\gamma$, to $X$.
Then one can show that $c_{\gamma}$ depends only on the 
$\Gamma$-conjugacy class of $\gamma$ and that the assignment $\gamma\mapsto
c_{\gamma}$ induces 
a bijection between $\CC(\Gamma)_{\s}$ and
$\mathcal{C}(X)$. For $c\in\mathcal{C}(X)$ let $\ell(c)$ denote its length.
Then 
there exists a constant $C_X$ such that for each $R$ one can estimate
\begin{align}\label{Estgeod}
\#\{c\in\mathcal{C}(X)\colon \ell (c)\leq R\}\leq C_Xe^{2nR}.
\end{align}
In particular, if one sets
\begin{align}\label{shortleng}
\ell(\Gamma):=\ell(X):=\inf\{\ell(c)\colon c\in\mathcal{C}(X)\},
\end{align}
then $\ell(\Gamma)>0$.

Measures are normalized as follows.
We normalize the Haar-measure on $K$ such that $K$ has volume 1. 
We fix an isometric identification of  $\mathbb{R}^{2n}$ with
$\mathfrak{n}$ 
with respect to the inner product $\left<\cdot,\cdot\right>_{\theta}$. We give 
$\mathfrak{n}$ the measure, induced from the Lebesgue measure under this 
identification. Moreover, we identify $\mathfrak{n}$ and $N$ by the 
exponential map and we will denote by $dn$ the Haar measure on $N$,
induced from the measure on $\mathfrak{n}$ under this identification. We 
normalize the Haar measure on $G$ by setting
\begin{align}\label{Haarmass auf G}
\int_{G}{f(g)dg}
&=\int_{N}\int_{\mathbb{R}}\int_{K}{e^{-2nt}f(na(t)k)dkdtdn}.
\end{align}

If $P'$ is a parabolic subgroup of $G$, the measures on $N_{P'}$ and $A_{P'}$
will be the measures induced from $N$ and $A$
via the 
conjugation with $k_{P'}$. Let $f$ be integrable over $\Gamma\backslash G$. Then
identifying
$f$ with a measurable function on $G$ it follows from \eqref{Haarmass auf G},
\eqref{FBI} and \eqref{FBII} that for every $Y\geq
Y_0$ one has
\begin{align}\label{intQuot}
\int_{\Gamma\backslash
G}f(x)dx=\int_{C(Y)}f(g)dg+\sum_{i=1}^{\kappa(\Gamma)}\int_{\Gamma\cap
N_{P_i}\backslash N_{P_i}}\int_{\log Y}^\infty\int_K
e^{-2nt}f(n_{P_i}a_{P_i}(t)k)dn_{P_i}dtdk
\end{align}

For $\sigma\in\hat{M}$ and $\lambda\in\C$ let $\pi_{\sigma,\lambda}$ be the 
principal series representation of $G$ parametrized as in \cite[section
2.7]{MP2}. In 
particular, the representations $\pi_{\sigma,\lambda}$ are unitary iff
$\lambda\in\R$. We denote by $\Theta_{\sigma,\lambda}$ the global character of
$\pi_{\sigma,\lambda}$. 
For $\sigma\in\hat{M}$ with highest weight $\Lambda(\sigma)$ as in
\eqref{Darstellungen von M} let $\sigma(\Omega_M)$ 
be the Casimir eigenvalue of $\sigma$ and let
\begin{align}\label{csigma}
c(\sigma):=\sigma(\Omega_M)-n^2=\sum_{j=2}^{n+1}(k_j(\sigma)+\rho_j)^2-
\sum_{j=1}^{n+1}\rho_j^2,
\end{align}
where the second equality follows from a standard computation. 

\section{Eisenstein series}\label{seceis}
\setcounter{equation}{0}
In this section we recall the definition and some basic properties of the
Eisenstein series. Let $\Gamma$ be a discrete subgroup of $G$ such that
$\vol(\Gamma\backslash G)$ is 
finite. Furthermore, for convenience we assume in this section that $\Gamma$ is
torsion-free and that 
for each $\Gamma$-cuspidal parabolic subgroup $P'$ of $G$ one has 
\begin{align}\label{asGamma}
\Gamma\cap P'=\Gamma\cap N_{P'}.
\end{align}
Let $\mathfrak{P}_\Gamma$ be a fixed set of representatives of
$\Gamma$-conjugacy
classes of $\Gamma$-cuspidal parabolic 
subgroups of $G$. Let $P_i\in\mathfrak{P}_\Gamma$. For $\sigma\in\hat{M}$ we
define a representation 
$\sigma_{P_i}$ of $M_{P_i}$ by
\begin{equation}\label{repsigma}
\sigma_{P_i}(m_{P_i}):=\sigma(k_{P_i}^{-1}m_{P_i}k_{P_i}),\quad m_{P_i}\in
M_{P_i}.
\end{equation} 
Now let $\nu\in\hat{K}$ and $\sigma_{P}\in\hat{M}$ such that
$\left[\nu:\sigma\right]\neq 0$. 
Then we let $\mathcal{E}_{P_i}(\sigma,\nu)$ be the set of all
continuous functions $\Phi$ on $G$ which are left-invariant under
$N_{P_i}A_{P_i}$ such
that for all
$x\in G$ the function
$m\mapsto \Phi_{P_i}(mx)$ belongs to $L^2(M_{P_i},\sigma_{P_i})$, the
$\sigma_{P_i}$-isotypical component of the
right regular representation of $M_{P_i}$, and such that 
for all $x\in G$ the function $k\mapsto \Phi_{P_i}(xk)$ belongs to the
$\nu$-isotypical component of the right regular representation of $K$. 
The space $\mathcal{E}_{P_i}(\sigma,\nu)$ is finite dimensional and 
in fact one has
\begin{align}\label{dim1}
\dim(\mathcal{E}_{P_i}(\sigma,\nu))=\dim(\sigma)\dim(\nu).
\end{align}

We define an inner product on $\mathcal{E}_{P_i}(\sigma,\nu)$ as follows.
Any element of $\mathcal{E}_{P_i}(\sigma,\nu)$ can be identified
canonically with a function on $K$. For
$\Phi,\Psi\in\mathcal{E}_{P_i}(\sigma,\nu)$ put
\begin{align}\label{InprodaufE}
\left<\Phi,\Psi\right>:=\vol(\Gamma\cap N_{P_i}\backslash
N_{P_i})\int_K\Phi(k)\bar{\Psi}(k)dk.
\end{align}
Define the Hilbert space $\mathcal{E}_{P_i}(\sigma)$ by
\begin{align*}
\mathcal{E}_{P_i}(\sigma):=\bigoplus_{\substack{\nu\in\hat{K}\\\left[
\nu:\sigma\right]\neq 0}}\mathcal{E}_{P_i}(\sigma,\nu).
\end{align*}

For $\Phi_{P_i}\in\mathcal{E}_{P_i}(\sigma,\nu)$ and $\lambda\in\C$ let
\begin{align}\label{PhiP}
\Phi_{P_i,\lambda}(g):=e^{(\lambda+n)(H_{P_i}(x))}\Phi_{P_i}(g).
\end{align}
Let $x\in\Gamma\backslash G$, $x=\Gamma g$. Then the Eisenstein series
$E(\Phi_{P_i}:\lambda:x)$ is defined by 
\begin{align}\label{Def ER}
E(\Phi_{P_i}:\lambda:x):=\sum_{\gamma\in(\Gamma\cap N_{P_i})\backslash\Gamma}
\Phi_{P_i,\lambda}(\gamma g).
\end{align}
On $\Gamma\backslash G\times\{\lambda\in\mathfrak\C\colon
\Real(\lambda)>n\}$ the series \eqref{Def ER} is absolutely and locally
uniformly
convergent. As a function of $\lambda$, it has a meromorphic continuation to
$\C$ with only finitely many poles in the strip $0<\Real(\lambda)\leq n$
which are located on $(0,n]$ and it has no poles on the line
$\Real(\lambda)=0$. By \eqref{Casimir}, for $\sigma\in\hat{M}$ with
$[\nu:\sigma]\neq 0$ and
$\Phi_{P_i}\in\mathcal{E}(\sigma,\nu)$ one has 
\begin{align}\label{DeltaPhi}
\Omega\Phi_{P_i,\lambda}= (\lambda^2+c(\sigma))\Phi_{P_i,\lambda},
\end{align}
where $c(\sigma)$ is as in \eqref{csigma}. Since $\Omega$ is $G$-invariant it
follows that 
\begin{align}\label{EqE}
\Omega E(\Phi_{P_i}:\lambda:x)=(\lambda^2+c(\sigma))E(\Phi_{P_i}:\lambda:x). 
\end{align}
Let 
\[
\boldsymbol{\mathcal{E}}(\sigma,\nu):=\bigoplus_{P_i\in\mathfrak{P}_\Gamma
}\mathcal{E}_{P_i}(\sigma,\nu);\quad
\boldsymbol{\mathcal{E}}(\sigma):=\bigoplus_{P_i\in\mathfrak{P}_\Gamma
}\mathcal{E}_{P_i}(\sigma).
\]
By \eqref{dim1} one has
\begin{align}\label{dim2}
\dim\boldsymbol{\mathcal{E}}(\sigma,\nu
)=\kappa(\Gamma)\dim(\sigma)\dim(\nu).
\end{align}

Let $P_i,P_j\in\mathfrak{P}_\Gamma$ and let $\sigma\in\hat{M}$. 
For $\Phi_{P_i}\in\mathcal{E}_{P_i}(\sigma,\nu)$, $i=1,2$, and $g\in G$ let
\begin{align*}
E_{P_j}(\Phi_{P_i}:g:\lambda):=\frac{1}{\vol\left(\Gamma\cap N_{P_j}\backslash
N_{P_j}\right)}\int_{\Gamma\cap N_{P_j}\backslash
N_{P_j}}{E(\Phi_{P_i}:ng:\lambda)dn}
\end{align*}
be the constant term of $E(\Phi_{P_i}:-:\lambda)$ along $P_j$. 
Then there exists a meromorphic function
\begin{align*}
C_{P_i|P_j}(\sigma:\nu:\lambda):\mathcal{E}_{P_i}(\sigma,\nu
)\longrightarrow\mathcal{E}_{P_j}(w_0\sigma,\nu),
\end{align*}
such that for
$P_i,P_j\in\mathfrak{P}_\Gamma$ one has
\begin{align}\label{C1}
E_{P_j}(\Phi_{P_i}:g:\lambda)=\delta_{i,j}\Phi_{P_i,\lambda}(g)+(C_{P_i|P_j}
(\sigma:\nu:\lambda)\Phi_{P_i})_{
-\lambda}(g).
\end{align}
Now we let
\begin{align*}
C_{P_i|P_j}(\sigma_{P_i},\lambda):=\bigoplus_{\substack{\nu\in\hat{K}\\
[\nu:\sigma]\neq 0}}C_{P_i|P_j}(\sigma,\nu,\lambda),
\end{align*}
where $\sigma_{P_i}$ is defined by \eqref{repsigma}. 
Furthermore, let 
\begin{align*}
\mathbf{C}(\sigma,\lambda):\boldsymbol{\mathcal{E}}(\sigma
) \to\boldsymbol{\mathcal{E}}(w_0\sigma
);\quad
\mathbf{C}(\sigma,\nu,\lambda):\boldsymbol{\mathcal{E}}(\sigma,\nu)
\to\boldsymbol{\mathcal{E}}(w_0\sigma,\nu)
\end{align*}
be the maps built from the maps $C_{P_i|P_j}(\sigma,\lambda)$, resp. 
$C_{P_i|P_j}(\sigma,\nu,\lambda)$. Then one has
\begin{align}\label{FG}
\mathbf{C}(w_0\sigma,\lambda)\mathbf{C}(\sigma,-\lambda)=\Id;\quad\mathbf{C}
(\sigma,\lambda)^{*}=\mathbf{C}
(w_0\sigma,\bar{\lambda}).
\end{align}

Let $\sigma\in\hat{M}$ and $\nu\in\hat{K}$. If  
$\sigma=w_0\sigma$, let
$\overline{\mathcal{E}}_{P_i}(\sigma,\nu):=\mathcal{E}_{P_i}(\sigma,
\nu)$, 
$\boldsymbol{\overline{\mathcal{E}}}(\sigma,\nu):=\boldsymbol{
\mathcal{E}}(\sigma,\nu)$,
$\mathbf{\overline{C}}(\sigma:\nu:s):=\mathbf{C}(\sigma:\nu:s)$. If 
$\sigma\neq w_0\sigma$, let
$\overline{\mathcal{E}}_{P_i}(\sigma,\nu):=\mathcal{E}_{P_i}(\sigma,
\nu)\oplus \mathcal{E}_{P_i}(w_0\sigma,\nu)$
$\boldsymbol{\overline{\mathcal{E}}}(\sigma,\nu):=\boldsymbol{
\mathcal{E}} (\sigma,\nu)\oplus
\boldsymbol{\mathcal{E}}(\boldsymbol{w_0\sigma},\nu)$
and
\begin{equation}\label{intbar}
\mathbf{\overline{C}}(\sigma,\nu,s):\
\:\boldsymbol{\overline{\mathcal{E}}}
(\sigma,\nu)\rightarrow
\boldsymbol{\overline{\mathcal{E}}}(\sigma,\nu);\quad
{\mathbf{\overline{C}}}(\sigma,\nu,s):=\begin{pmatrix}0&\mathbf{C}
(w_0\sigma,\nu,s)\\ \mathbf{C}(
\sigma,\nu,s)&0 \end{pmatrix}.
\end{equation}

Let $R_\sigma$ (resp. $R_{w_0\sigma}$) denote the right regular
representation of  $K$ on $\boldsymbol{\mathcal{E}}(\sigma)$ (resp.
$\boldsymbol{\mathcal{E}}(w_0\sigma)$). Then 
$\mathbf{C}(\sigma,s)$ is an intertwining operator between $R_\sigma$ and
$R_{w_0\sigma}$. Thus if $\nu$ 
is a finite-dimensional representation of $K$ on $V_\nu$, we can define  
$\widetilde{\boldsymbol{C}}(\sigma,\nu,s)$ as the restriction of
$(\boldsymbol{C}(\sigma,s)\otimes\Id)$ to 
a map from $(\boldsymbol{\mathcal{E}}
(\sigma)\otimes V_\nu)^K$ to $(\boldsymbol{\mathcal{E}}(w_0\sigma)\otimes
V_\nu)^K$.
For later purpose we need the following Lemma.

\begin{lem}\label{LemC}
In the sense of meromorphic functions one has 
\begin{align*}
\Tr\left(\widetilde{\boldsymbol{C}}
(\sigma,\nu,s)^{-1}\frac{
d}{ds}\widetilde{\boldsymbol{C}}(\sigma,\nu,
s)\right)
=\frac{1}{\dim(\nu)}
\Tr\left(\boldsymbol{C}
(\sigma,\nu,s)^{-1}\frac{
d}{ds}\boldsymbol{C}(\sigma,\nu,s)\right)
\end{align*}
for each $\sigma\in\hat{M}$, $\nu\in\hat{K}$ with $[\nu:\sigma]\neq 0$. 
\end{lem}
\begin{proof}
Let $P_1$ be the projection form $\boldsymbol{\mathcal{E}}(\sigma)$ to
$\boldsymbol{\mathcal{E}}(\sigma,\nu)$
and let $P_2$ be the projection from $(\boldsymbol{\mathcal{E}}
(\sigma)\otimes V_\nu)$ to $(\boldsymbol{\mathcal{E}}
(\sigma)\otimes V_\nu)^K$. Then using that $\check{\nu}\cong\nu$ we have
\begin{align*}
P_{1}=\dim(\nu)\int_K\chi_\nu(k)R_\sigma(k);\quad
P_{2}=\int_KR_\sigma(k)\otimes\nu(k)dk,
\end{align*}
where $\chi_\nu$ is the character of $\nu$. 
Thus one has 
\begin{align*}
&\Tr\left(\widetilde{\boldsymbol{C}}
(\sigma,\nu,s)^{-1}\frac{
d}{ds}\widetilde{\boldsymbol{C}}(\sigma,\nu,
s)\right)=\Tr\left(
\boldsymbol{C}(\sigma,s)^{-1}\frac{d}{ds}\boldsymbol{C}(\sigma,s)\otimes\Id\circ
P_2\right)\\
=&\Tr\left(
\int_K\boldsymbol{C}(\sigma,s)^{-1}\frac{d}{ds}\boldsymbol{C}(\sigma,s)\circ
R_\sigma(k)\otimes\nu(k)dk\right)\\
=&\Tr\left(\int_K
\boldsymbol{C}(\sigma,s)^{-1}\frac{d}{ds}\boldsymbol{C}(\sigma,s)\circ
\chi_\nu(k)R_\sigma(k)dk\right)\\
=&\frac{1}{\dim(\nu)}\Tr\left(
\boldsymbol{C}(\sigma,s)^{-1}\frac{d}{ds}\boldsymbol{C}(\sigma,s)\circ
P_1 \right)=\frac{1}{\dim(\nu)}
\Tr\left(\boldsymbol{C}
(\sigma,\nu,s)^{-1}\frac{
d}{ds}\boldsymbol{C}(\sigma,\nu,s)\right),
\end{align*}
which concludes the proof of the proposition.
\end{proof}

\section{Factorization of the C-matrix}\label{secfact}
\setcounter{equation}{0}
We let $\Gamma$ be a discrete subgroup of $G$ satisfying \eqref{asGamma} and we
keep the
notations of the previous section. By the results of M\"uller, in particular
\cite[equation (6.8)]{Mu1}, the
determinant of the matrix ${\mathbf{\overline{C}}}(\sigma,\nu,\lambda)$ 
factorizes into a product of an exponential factor and an infinite Weierstrass
product involving its zeroes and poles. For the case of 
a hyperbolic surface, this factorization was first established by 
Selberg  (see\cite[page 656]{Se}). 

While the poles and zeroes of the $C$-matrices are 
easily seen to be independent of the choice of $\mathfrak{P}_\Gamma$, the
exponential factor 
depends on $\mathfrak{P}_\Gamma$ or, equivalently, on the choice of truncation
parameters. This fact will become particularly crucial if one lets the manifold
$X$ vary. 
In \cite{Mu1}, the manifold $X$ and the set $\mathfrak{P}_\Gamma$ were fixed.
Therefore, for the purposes of the present article we have to go through 
the arguments of the paper \cite{Mu1} which led to
equation (6.8) in this paper and to keep track of the precise choices of
truncation parameters.

Let $R_\Gamma$ be the right regular representation 
of $G$ on $L^2(\Gamma\backslash G)$. If $\nu$ is a finite dimensional
representation of $K$, let 
$L^2(\Gamma\backslash G)_\nu$ denote the $\nu$-isotypical component of the
restriction of $R_\Gamma$ 
to $K$. Let $C_c^\infty(\Gamma\backslash G)_\nu:=C_c^\infty(\Gamma\backslash
G)\cap L^2(\Gamma\backslash G)_\nu$. Then 
it is easy to see that $C_c^\infty(\Gamma\backslash G)_\nu$ is dense in
$L^2(\Gamma\backslash G)_\nu$. 

Now let $\Delta_\nu$ be the differential operator in 
$C^\infty(\Gamma\backslash G)_\nu$, which is induced by $-R_\Gamma(\Omega)$.
If we regard it as an operator in $L^2(\Gamma\backslash G)_\nu$ with domain
$C_c^\infty(\Gamma\backslash G)_\nu$, it is symmetric, essentially
selfadjoint and satisfies $\Delta_\nu\geq -\nu(\Omega_K)$, where
$\nu(\Omega_K)\in\R^+$ is the Casimir eigenvalue of $\nu$. This follows 
easily from the considerations in the next section \ref{secbol}. The closure of
$\Delta_\nu$ will be denoted by $\overline{\Delta}_\nu$. One has
\begin{align}\label{Spec}
\sigma(\overline{\Delta_\nu})\subset (-\nu(\Omega_K),\infty).
\end{align}

We fix a smooth function $\phi$ on $\R$ with values in $[0,1]$ such that 
$\phi(t)=0$ for $t\leq 0$ and $\phi(t)=1$ for $t\geq 1$. 
If $P_i\in\mathfrak{P}_\Gamma$, then for $Y\in (0,\infty)$ we let
\[
\psi_{P_i,Y}(n_{P_i}a_{P_i}(t)k):=\phi(t-\log{Y}),\quad
n_{P_i}\in N_{P_i},\: t\in \R.
\]

Now let $Y_{P_i}\in (0,\infty)$, $i=1,\dots,\kappa(\Gamma)$, such that 
$Y_{P_i}\geq Y_{P_i}^0(\Gamma)$, where $Y_{P_i}^0(\Gamma)$ is defined by
\eqref{infimum}. For 
$\Phi_{P_i}\in\mathcal{E}(\sigma_{P_i},\nu)$ 
we define a function $\theta(\Phi_{P_i}:Y_{P_i}:\lambda:x)$ on $\Gamma\backslash
G$ by 
\begin{align}\label{deftheta}
\theta(\Phi_{P_i}:Y_{P_i}:\lambda:x):=\sum_{\gamma\in\Gamma\cap
N_{P_i}\backslash\Gamma}
\psi_{ P_i,
Y_{P_i}} (\gamma
g)\Phi_{P_i,\lambda}(\gamma g); \quad x=\Gamma g.
\end{align}
By \eqref{FBII} at most one summand in this sum can be non-zero . We let
\[
H(\Phi_{P_i}:Y_{P_i}:\lambda:x):=(\Delta_\nu+c(\sigma_{P_i}
)+\lambda^2)\theta(\Phi_{P_i}:Y_{P_i}:\lambda:x).
\]
Then by \eqref{DeltaPhi} one has 
$H(\Phi_{P_i}:Y_{P_i}:\lambda:x)\in
C_c^\infty(\Gamma\backslash G)_\nu$. Moreover, the
Eisenstein series can be characterized by the following Proposition, which for 
$\dim X=2$ is due to Colin de Verdi\`{e}re \cite{CV}.

\begin{prop}\label{PropER}
For $P_i\in\mathfrak{P}_\Gamma$, $Y_{P_i}\geq Y_{P_i}^0(\Gamma)$ and
$\lambda\in\C$ with 
$\lambda^2+c(\sigma)\notin (-\infty,\nu(\Omega_K))$ and $\Real(\lambda)>0$ one 
has
\[
E(\Phi_{P_i}:\lambda:x)=\theta(\Phi_{P_i}:Y_{P_i}:\lambda:x)-(\overline{\Delta}
_\nu+\lambda^2+c(\sigma))^{-1}(H(\Phi_{P_i}:Y_{P_i}:\lambda:x)).
\]
\end{prop}
\begin{proof}
This was proved in general in \cite[Proposition 4.7]{Mu1}. For the convenience 
of the reader we recall the proof. Denote the right hand side by $\widetilde 
E(\Phi_{P_i}:\lambda:x)$. By definition it satisfies $(\Delta_\nu+\lambda^2
+c(\sigma))\widetilde E(\Phi_{P_i}:\lambda:x)=0$. By \eqref{EqE}, 
$E(\Phi_{P_i}:\lambda:x)$ satisfies the same differential equation. By
\cite[Lemma 4.5]{Mu1},
$E(\Phi_{P_i}:\lambda)-\theta(\Phi_{P_i}:Y_{P_i}:\lambda)$
is square integrable for $\Re(\lambda)>n$. Hence, $u:=E((\Phi_{P_i}:\lambda)-
\widetilde E(\Phi_{P_i}:\lambda)$ is square integrable for $\Re(\lambda)>n$ 
and satisfies $(\Delta_\nu+\lambda^2+c(\sigma))u=0$. Since 
$\Delta_\nu$ is essentially self-adjoint, it follows that 
$E((\Phi_{P_i}:\lambda)=\widetilde E(\Phi_{P_i}:\lambda)$ for $\Re(\lambda)>n$.
The proposition follows by the uniqueness the analytic continuation. 
\end{proof}

\begin{lem}\label{Lemres}
There exists a constant $C_1$ which is independent of $\Gamma$ and
$\mathfrak{P}_\Gamma$ such that for all $\lambda\in\C$ with
$\lambda^2+c(\sigma)\notin (-\infty,\nu(\Omega_K))$ and $\Real(\lambda)>0$, all
$Y_{P_i}\geq Y_{P_i}^0(\Gamma)$, and all
$\Phi_{P_i}\in\mathcal{E}_{P_i}(\sigma,\nu)$, $P_i\in\mathfrak{P}_\Gamma$,
one has
\[
\left\|H(\Phi_{P_i}:Y_{P_i}:\lambda:x)\right\|_{L^2(\Gamma\backslash G)}\leq
C_1e^{\Real(\lambda)(\log Y_{P_i}+2)}
\left\|\Phi_{P_i}\right\|_{\mathcal{E}_{P_i}(\sigma,\nu)}.
\]
\end{lem}
\begin{proof}
There exists a unique $\Phi_P\in\mathcal{E}_P(\sigma,\nu)$ such that
$\Phi_{P_i,\lambda}(g)=\Phi_{P,\lambda}(\kappa_{P_i}^{-1}g\kappa_{P_i})$. 
Since $\Delta_\nu$ commutes
with the 
right-action of $G$ on $\Gamma\backslash G$, it follows from \eqref{intQuot}
that
\begin{align*}
&\int_{\Gamma\backslash G}|H(\Phi_{P_i}:Y_{P_i}:\lambda:x)|^2dx\\
=&\vol(\Gamma\cap
N_{P_i}\backslash N_{P_i})\int_{\log{Y_{P_i}}}^{\log{Y_{P_i}}+1}\int_K
{e^{-2nt}}|(\Delta_\nu+c(\sigma)+\lambda^2)\psi_{P_i,Y_{P_i}}(a_{P_i}(t))\Phi_{
P_i,\lambda}(a_{P_i}(t)k)|^2\;dkdt\\
=&\vol(\Gamma\cap N_{P_i}\backslash
N_{P_i})\int_{\log{Y_{P_i}}}^{\log{Y_{P_i}}+1}\int_K
{e^{-2nt}}|(\Delta_\nu+c(\sigma)+\lambda^2)\psi_{P,Y_{P_i}}(a(t))\Phi_{
P,\lambda}(a(t)k)|^2\;dkdt.
\end{align*}
Now using \eqref{Casimir} and \eqref{EqE} one obtains
\begin{align*}
&(\Delta_\nu+c(\sigma)+\lambda^2)(\psi_{P,Y_{P_i}}(a(t))\Phi_{
P,\lambda}(a(t)k))\\=&-e^{(\lambda+n)t}\Phi_P(k)(\phi''(t-\log
Y_{P_i})+2\lambda\phi'(t-\log Y_{P_i})
).
\end{align*}
This proves the proposition.
\end{proof}

\begin{kor}\label{korres}
There exists a constant $C_1$ which is independent of $\Gamma$ and
$\mathfrak{P}_\Gamma$ such that for all $\lambda\in\C$ with
$\Real(\lambda^2)+c(\sigma)\geq \nu(\Omega_K)+1$ and $\Real(\lambda)>0$, 
all $Y_{P_i}\geq Y_{P_i}^0(\Gamma)$ and all
$\Phi_{P_i}\in\mathcal{E}_{P_i}(\sigma,\nu)$, $P_i\in\mathfrak{P}_\Gamma$,
one has
\[
\|(\overline{\Delta}_\nu+\lambda^2+c(\sigma))^{-1}H(\Phi_{P_i}:Y_{P_i}
:\lambda:x)\|_{
L^2(\Gamma\backslash G)}\leq C_1 e^{\Real(\lambda)(\log{Y_{P_i}}+2)}
\left\|\Phi_{P_i}\right\|_{\mathcal{E}_{P_i}(\sigma,\nu)}
\]
\end{kor}
\begin{proof}
By \cite[V,\S 3.8]{Kato} one can estimate the operator norm of the resolvent by
\[
\|(\overline{\Delta}_\nu+\lambda^2+c(\sigma))^{-1}\|\leq\frac{1}{
\dist(-\lambda^2-c(\sigma),
\spec(\overline{\Delta}_\nu))},
\]
where the estimate holds without any constant. 
Applying the previous Lemma and \eqref{Spec}, the corollary follows.
\end{proof}

In the following proposition we estimate the coefficients 
of the $C$-matrix. 
\begin{prop}
There exists a constant $C_2$, which is independent of $\Gamma$ and
$\mathfrak{P}_\Gamma$ 
such that for all $P_i,P_j\in\mathfrak{P}_\Gamma$, all 
$Y_{P_i}, Y_{P_j}\in (0,\infty)$ with $Y_{P_i}\geq Y^0_{P_i}(\Gamma)$,
$Y_{P_j}\geq Y^0_{P_j}(\Gamma)$, all
$\Phi_{P_i}\in\overline{\mathcal{E}}_{P_i}(\sigma,\nu)$,
$\Phi_{P_j}\in\overline{\mathcal{E}}_{P_j}(\sigma,\nu)$ and all $\lambda\in\C$
with
$\Real(\lambda^2)+c(\sigma)\geq \nu(\Omega_K)+1$ and $\Real(\lambda)>0$,  
one has 
\[
|\left<C_{P_i|P_j}(\sigma,\nu,\lambda)(\Phi_{P_i}),\Phi_{P_j}\right>_{\overline{
\mathcal{E
}}_{P_j}(\sigma,\nu)}| \leq
C_2e^{\Real(\lambda)(\log{Y_{P_i}}+\log{Y_{P_j}}+4)}\|\Phi_{P_i}\|_{\overline{
\mathcal{E
}}_{P_i}(\sigma,\nu)}\cdot
\|\Phi_{P_j}\|_{\overline{
\mathcal{E
}}_{P_j}(\sigma,\nu)}.
\]
\end{prop}

\begin{proof}
By the definition \eqref{C1} of the constant term it follows that for each 
$t\in\R$ and each $k\in K$ one has
\begin{align*}
C_{{P_i|P_j}}(\sigma,\nu,\lambda)(\Phi_{P_i})(k)=&e^{(\lambda-n)t}(C_{{P_i|P_j}}
(\sigma,\nu,\lambda)(\Phi_{P_i}))_{-\lambda}(a_{P_j}(t)k)\\
=&e^{(\lambda-n)t}\bigl(E_{P_j}(\Phi_{P_i}:a_{P_j}(t)k:\lambda)-\delta_{i,j}
\Phi_{P_i,\lambda}(a_{P_j}(t)k)\bigr).
\end{align*}
Moreover, by \eqref{FBI} and \eqref{FBII}, for  $t\geq \log
Y_{P_j}+1$ one has
\begin{align*}
\theta(\Phi_{P_i}:Y_{P_i}:\lambda:a_{P_j}(t)k)=\delta_{i,j}\Phi_{
P_i,\lambda}(a_{P_j}(t)k).
\end{align*}
Thus by Proposition \ref{PropER} for $t\geq \log{Y_{P_j}}+1$ one has
\begin{align*}
&E_{P_j}(\Phi_{P_i}:a_{P_j}(t)k:\lambda)-\delta_{i,j}\Phi_{P_i,
\lambda}(a_{P_j}(t)k)\\=&-\frac{1
}{ \vol\left(\Gamma\cap N_{P_j}\backslash
N_{P_j}\right)}\int_{\Gamma\cap N_{P_j}\backslash
N_{P_j}}{(\overline{\Delta}_\nu+c(\sigma)+\lambda^2)^{-1}
(H(\Phi_{P_i}:Y_{P_i}:\lambda:n_{P_j}a_{P_j}(t)k))\;dn_{P_j}}.
\end{align*}
Combining these equations, it follows that for each $t\geq \log Y_{P_j}+1$ one 
has
\begin{align}\label{eqc}
&\left<C_{P_i|P_j}(\sigma,\nu,\lambda)(\Phi_{P_i}),\Phi_{P_j}\right>_{\overline{
\mathcal{E
}}_{P_j}(\sigma,\nu)}\nonumber\\=&\vol(\Gamma\cap
N_{P_j}\backslash
N_{P_j})\int_{K}\overline{\Phi_{P_j}}
(k)C_{P_i|P_j}(\sigma,\nu,\lambda)(\Phi_{P_i})(k)dk
=-e^{(\lambda-n)t}\nonumber\\ &\times\int_K
\overline{\Phi_{P_j}}(k)\int_{\Gamma\cap
N_{P_j}\backslash N_{P_j}}(\overline{\Delta}_\nu+c(\sigma)+\lambda^2)^{-1}
(H(\Phi_{P_i}:Y_{P_i}:\lambda:n_{P_j}a_{P_j}(t)k))\;dn_{P_j}dk.
\end{align}
Now we define a function $\widetilde{f}_{P_j,\lambda}$ on $G$ by 
\[
\widetilde{f}_{P_j,\lambda}(n_{P_j}a_{P_j}(t)k)=e^{(n+\lambda)t}\chi_{[\log
Y_{P_j},\:
\log Y_{P_j}+1]}(t)\Phi_{P_j}(k), 
\]
where 
$\chi_{[\log Y_{P_j},\:\log Y_{P_j}+1]}(t)$ denotes the characteristic function
of the
interval $[\log Y_{P_j},\log Y_{P_j}+1]$. Then we 
define a function $f_{P_j,\lambda}$ on $\Gamma\backslash G$ by 
\begin{align*}
f_{P_j,\lambda}(x)=\sum_{\gamma\in\Gamma\cap P_j\backslash
\Gamma}\widetilde{f}_{P_j,\lambda}(\gamma g),\quad x=\Gamma g.
\end{align*}
By \eqref{FBII}, at most one summand in this sum can be nonzero. 
Integrating equation \eqref{eqc} over $t$ in the interval $[\log Y_{P_j},\:\log
Y_{P_j}+1]$ and using \eqref{intQuot},
we obtain 
\begin{align*}
&\big|\left<C_{P_i|P_j}(\sigma,\nu,\lambda)(\Phi_{P_i}),\Phi_{P_j}\right>_{
\overline{
\mathcal{E
}}_{P_j}(\sigma,\nu)}\big|\\
=&\big|\left<(\overline{\Delta}_\nu+c(\sigma)+\lambda^2)^{-1}(
H(\Phi_{P_i}:Y_{P_i}:\lambda)),f_{P_j,\lambda}\right>_{L^2(\Gamma\backslash
G)}\big|.
\end{align*}
Now observe that
\begin{align*}
\|f_{P_j,\lambda}\|_{L^2(\Gamma\backslash G)}\leq
e^{\Real(\lambda)(\log
Y_{P_j}+1)}\|\Phi_{P_j}\|_{\overline{
\mathcal{E
}}_{P_j}(\sigma,\nu)}.
\end{align*}
Applying Corollary \ref{korres}, the Proposition follows.
\end{proof}

Summarizing our results, we obtain the following
refinement of \cite[Lemma 6.1]{Mu1}.

\begin{kor}
Let $\bar{d}(\sigma,\nu):=\dim\overline{\mathcal{E}}_P(\sigma,\nu)$.
For each $P_i\in\mathfrak{P}_\Gamma$ let $Y_{P_i}\geq Y_{P_i}^0(\Gamma)$ be
given.
Put
\[
q_1:=\prod_{i=1}^{\kappa(\Gamma)}e^{2(\log Y_{P_i}+2)\bar{d}(\sigma,\nu)}.
\]
There exists a constant $C>0$ which is independent of $\Gamma$,
$\mathfrak{P}_\Gamma$, and $Y_{P_i}$, $i=1,...,\kappa(\Gamma)$, such that  
for all $\lambda\in\C$ satisfying $\Real(\lambda^2)+c(\sigma)\geq 
\nu(\Omega_K)+1$ and $\Real(\lambda)>0$, one has
\[
|\det(\overline{\mathbf{C}}(\sigma,\nu,\lambda))|\leq C q_1^{\Real(\lambda)}.
\]
\end{kor}
\begin{proof}
If one chooses for each $i=1,\dots,\kappa(\Gamma)$ an orthonormal 
base of $\mathcal{E}_{P_i}(\sigma,\nu)$ resp.
$\mathcal{E}_{P_i}(w_0\sigma,\nu)$ and applies 
the preceding proposition, the corollary follows immediately from 
the Leibniz formula for the determinant. 
\end{proof}

Applying the previous Corollary we can restate the 
factorization of the $C$-matrix, \cite[equation 6.8]{Mu1} with an expression for
the 
exponential factor in terms of the truncation parameters that 
will be sufficient for our later considerations.
\begin{thrm}\label{ThrmC}
Let $\sigma_j$, $j=1,\dots,l$ denote the poles of
$\det(\overline{\mathbf{C}}(\sigma,\nu,\lambda))$ in the interval $(0,n]$
and let $\eta$ run through the poles of
$\det(\overline{\mathbf{C}}(\sigma,\nu,\lambda))$ in the half-plane
$\Real(\lambda)\leq 0$, both counted 
with multiplicity.
Then one has 
\[
\det(\overline{\mathbf{C}}(\sigma,\nu,\lambda))=\det(\overline{\mathbf{C}}
(\sigma,\nu,0))q^\lambda\prod_{j=1}^{l}\frac{\lambda+\sigma_j}{
\lambda-\sigma_j}\prod_{\eta}\frac{\lambda+\bar{\eta}}{\lambda-\eta}.
\]
Moreover, if for each $P_i\in\mathfrak{P}_\Gamma$ a $Y_{P_i}\in (0,\infty)$ with
$Y_{P_i}\geq Y_{P_i}^0(\Gamma)$ is given, then $q$ can be written as
\begin{align}\label{Repq}
q=e^{a}\prod_{i=1}^{\kappa(\Gamma)}e^{2(\log Y_{P_i}+2)\bar{d}(\sigma,\nu)},
\end{align}
where $a\in\R$, $a\leq 0$.
\end{thrm}

\begin{proof}
Using the previous Corollary instead of \cite[Lemma 6.1]{Mu1}, one can proceed
exactly as in \cite[section 6]{Mu1} to 
obtain the Theorem.
\end{proof}

\section{Twisted Laplace operators}
\label{secbol}
\setcounter{equation}{0}

Let $\nu$ be a finite dimensional unitary representation of $K$ on 
$(V_{\nu},\left<\cdot,\cdot\right>_{\nu})$. Let
\begin{align*}
\tilde{E}_{\nu}:=G\times_{\nu}V_{\nu}
\end{align*}
be the associated homogeneous vector bundle over $\tilde{X}$. Then 
$\left<\cdot,\cdot\right>_{\nu}$ induces a $G$-invariant metric 
$\tilde{B}_{\nu}$ on $\tilde{E}_{\nu}$. 
Let 
\begin{align*}
E_{\nu}:=\Gamma\backslash(G\times_{\nu}V_{\nu})
\end{align*}
be the associated locally homogeneous bundle over $X$. Since 
$\tilde{B}_{\nu}$ is $G$-invariant, it can be
pushed down to a fiber metric $B_{\nu}$ on 
$E_{\nu}$. Let
\begin{align}\label{globsect}
C^{\infty}(G,\nu):=\{f:G\rightarrow V_{\nu}\colon f\in C^\infty,\;
f(gk)=\nu(k^{-1})f(g),\,\,\forall g\in G, \,\forall k\in K\}.
\end{align}
Let
\begin{align}\label{globsect1}
C^{\infty}(\Gamma\backslash G,\nu):=\left\{f\in C^{\infty}(G,\nu)\colon 
f(\gamma g)=f(g),\;\forall g\in G, \;\forall \gamma\in\Gamma\right\}.
\end{align}
Let $C^{\infty}(X,E_{\nu})$ denote the space of smooth sections of $E_{\nu}$. 
Then there is a canonical isomorphism
\begin{align*}
A:C^{\infty}(X,E_{\nu})\cong C^{\infty}(\Gamma\backslash G,\nu)
\end{align*}
(see \cite[p. 4]{Mi1}).
There is also a corresponding isometry for the space $L^{2}(X,E_{\nu})$ of 
$L^{2}$-sections of $E_{\nu}$. 

Let $\tau$ be an irreducible finite dimensional representation of $G$ on
$V_{\tau}$. Let $E_{\tau}$ be the flat
vector bundle associated to the
restriction of $\tau$ to $\Gamma$. Let $\widetilde E_\tau\to \widetilde X$ be 
the homogeneous vector bundle associated to $\tau|_K$. Then by \cite{MtM}
there is canonical isomorphism 
\[
E_\tau\cong\Gamma\backslash\widetilde E_\tau.
\]
By \cite{MtM}, there exists an inner product $\left<\cdot,\cdot\right>$ on
$V_{\tau}$ such that
\begin{enumerate}
\item $\left<\tau(Y)u,v\right>=-\left<u,\tau(Y)v\right>$ for all
$Y\in\mathfrak{k}$, $u,v\in V_{\tau}$
\item $\left<\tau(Y)u,v\right>=\left<u,\tau(Y)v\right>$ for all
$Y\in\mathfrak{p}$, $u,v\in V_{\tau}$.
\end{enumerate}
Such an inner product is called admissible. It is unique up to scaling. Fix an
admissible inner product. Since $\tau|_{K}$ is unitary with respect to this
inner product, it induces a fiber metric on $\widetilde E_\tau$, and hence
on $E_\tau$. This fiber metric will also  be called admissible. Let 
$\Lambda^{p}(X,E_{\tau})$ be the space of $E_\tau$-valued $p$-forms.
This is the space of smooth sections of the vector bundle
$\Lambda^{p}(E_{\tau}):=
\Lambda^pT^*X\otimes E_\tau$. Let
\begin{equation}
d_{p}(\tau)\colon \Lambda^{p}(X,E_{\tau})\to \Lambda^{p+1}(X,E_{\tau})
\end{equation}
be the exterior derivative and let
\begin{equation}\label{laplace}
\Delta_p(\tau)=d_p(\tau)^* d_p(\tau)+d_{p-1}(\tau)d_{p-1}(\tau)^*
\end{equation}
be the Laplace operator on $E_\tau$-valued $p$-forms.
This operator can be expressed in the
locally homogeneous setting as follows. Let $\nu_{p}(\tau)$ be the
representation of $K$ defined by
\begin{equation}\label{nutau}
\nu_{p}(\tau):=\Lambda^{p}\Ad^{*}\otimes\tau:\:K\rightarrow\Gl(\Lambda^{p}
\mathfrak{p}^{*}\otimes V_{\tau}).
\end{equation}
There is a canonical isomorphism
\begin{align}\label{p Formen als homogenes Buendel}
\Lambda^{p}(E_{\tau})\cong\Gamma\backslash(G\times_{\nu_{p}(\tau)}(\Lambda^{p}
\mathfrak{p}^{*}\otimes V_{\tau})),
\end{align}
which induces an isomorphism
\begin{align}\label{sections}
\Lambda^{p}(X,E_{\tau})\cong C^{\infty}(\Gamma\backslash G,\nu_{p}(\tau)).
\end{align}
There is a corresponding isometry of the $L^{2}$-spaces. 
Let $\tau(\Omega)$ be the Casimir eigenvalue of $\tau$. With respect to the
isomorphism \eqref{sections} on has
\begin{align}\label{kuga}
\Delta_{p}(\tau)=-R_\Gamma(\Omega)+\tau(\Omega)\Id
\end{align}
(see \cite[(6.9)]{MtM}).
Next we want to show that the discrete spectrum of the operators
$\Delta_p(\tau)$ is greater or equal than $1/4$ for each 
$p$ and each $\tau\in\Rep(G)$ satisfying $\tau\neq\tau_\theta$. This was
already
stated in \cite[Lemma 7.3]{MP2}. 
However, as it was kindly brought to our attention by Martin Olbrich, the
parametrization of the complementary series used in the proof of
that Lemma was incorrect. Therefore we shall now correct the part of the
argument leading 
to the proof of \cite[Lemma 7.3]{MP2} which involved the complementary series.
We let $\hat{G}_{\un}$ denote the unitary dual of $G$.
\begin{lem}
Let $\tau\in\Rep(G)$ such that $\tau\neq\tau_{\theta}$. Let
$\pi\in\hat{G}_{\un}$
belong to the complementary series. Let
$p\in\{0,\dots,d\}$.  Then if $[\pi:\nu_p(\tau)]\neq 0$ one 
has $-\pi(\Omega)+\tau(\Omega)\geq 1$. 
\end{lem}
\begin{proof}
Let $\tau$ be a finite-dimensional irreducible representation of $G$ of highest
weight
$\Lambda(\tau)=\tau_1e_1+\dots+\tau_{n+1}e_{n+1}$ as in 
\eqref{Darstellungen von G} and assume that $\tau\neq\tau_{\theta}$. Let
$p\in\{0,\dots,d\}$ and let 
$\sigma\in\hat{M}$ such that $[\nu_p(\tau):\sigma]\neq 0$. Assume that
$\sigma=w_0\sigma$. Let
$\Lambda(\sigma)=k_2(\sigma)e_2+\dots+k_{n+1}(\sigma)e_{n+1}$ be 
the highest weight of $\sigma$ as in \eqref{Darstellungen von M}. 
It was shown in the proof of \cite[Lemma 7.1]{MP2} that 
$\tau_{j-1}+1\geq |k_j(\sigma)|$
for every $j\in\{2,\dots,n+1\}$. Let $c(\sigma)$ be as in \eqref{csigma} and let
$l\in\{1,\dots,n\}$ be
minimal with the property that $k_{l+1}(\sigma)=0$. Using
$\rho_{j-1}=\rho_{j}+1$ and \cite[equation 2.20]{MP1}, it follows that one
can estimate
\begin{align}\label{eqc2}
c(\sigma)=\sum_{j=2}^{l}(k_j(\sigma)+\rho_j)^2-\sum_{j=1}^{n+1}\rho_j^2\leq
&\sum_{j=2}^{l}(\tau_{j-1}+\rho_{j-1})^2-\sum_{j=1}^{n+1}
\rho_j^2\nonumber\\ =&\tau(\Omega)-\sum_{j=l}^{n+1}(\tau_j+\rho_j)^2.
\end{align}
We parametrize 
the principal series representations as above. Then if $\pi$ belongs to 
the complementary series, by \cite[Proposition 49,
Proposition 53]{Knapp Stein} and our parametrization there exists a
$\sigma\in\hat{M}$, 
$\sigma=w_0\sigma$ and a $\lambda\in(0,n-l+1)$, where $l$ is minimal with the
property that $k_{l+1}(\sigma)=0$, such that $\pi_{\sigma,i\lambda}$ is
unitarizable with unitarization
$\pi$. We write
$\pi=\pi^c_{\sigma,i\lambda}$. If $[\pi^c_{\sigma,i\lambda}:\nu_p(\tau)]\neq 0$,
then by Frobenius 
reciprocity \cite[page 208]{Knapp} one has $[\nu_p(\tau):\sigma]\neq 0$.
Thus, since 
$\sigma=w_0\sigma$, it follows easily from the branching laws for restrictions
of 
representations from $G$ to $K$ and from $K$ to $M$, \cite{Goodman}[Theorem
8.1.3, Theorem 8.1.4]  that all $k_j(\tau)$ 
defined as in \eqref{Darstellungen von G} are integral. 
By \cite[Corollary 2.4]{MP1} one has 
\begin{align}
-\pi^c_{\sigma,i\lambda}(\Omega)+\tau(\Omega)=-\lambda^2-c(\sigma)+\tau(\Omega) 
\end{align}
and if we apply equation \eqref{eqc2} and the condition $|\tau_{n+1}|\geq 1$, it
follows that 
\begin{align*}
-\pi^c_{\sigma,i\lambda}(\Omega)+\tau(\Omega)\geq\sum_{j=l}^{n+1}
(\tau_j+\rho_j)^2-(n-l+1)^2= \sum_{j=l}^{n+1}(\tau_j+\rho_j)^2-\rho_l^2\geq
\tau_{n+1}^2\geq 1 
\end{align*}
and the Lemma is proved. 
\end{proof}

\begin{kor}\label{Korspec}
Let $\tau\in\Rep(G)$, $\tau\neq\tau_\theta$. For $p\in\{0,\dots,d\}$ let
$\lambda_0$ be an 
eigenvalue of $\Delta_p(\tau)$. Then one has $\lambda_0\geq\frac{1}{4}$. 
\end{kor}
\begin{proof}
Using the preceding Lemma, one can proceed exactly as in the proof of
\cite[Lemma 7.3]{MP2} to establish the corollary. 
\end{proof}

\section{The regularized trace under coverings}\label{sectr}
\setcounter{equation}{0}
Let $X=\Gamma\backslash\mathbb{H}^d$ be a finite-volume hyperbolic manifold.
For $\tau$ a finite-dimensional irreducible representation of $G$ let
$e^{-t\Delta_p(\tau)}$ be the heat operator associated to the Laplace
operator \eqref{laplace} acting on the locally homogeneous vector-bundle
$E_\tau$ over $X$.  To begin with recall the definition 
of the regularized trace of the heat operators $e^{-t\Delta_p(\tau)}$ introduced
in
\cite{MP2}. Let 
\begin{align*}
K_X^{\tau,p}(t;x,y)\in C^\infty(X_1\times X_1,E_\tau \boxtimes E_\tau^*)
\end{align*} 
be the kernel of $e^{-t\Delta_p(\tau)}$. If a set
$\mathfrak{P}_{\Gamma}$ of representatives of
$\Gamma$-cuspidal parabolic subgroups of $X$ is fixed, then according to
\eqref{XY}, one obtains compact 
smooth manifolds $X_{\mathfrak{P}_\Gamma}(Y)$ with boundary which exhaust $X$.
Using the Maass-Selberg relations, one can show that 
there is an asymptotic expansion
\begin{align}\label{asexp}
\int_{X_{\mathfrak{P}_\Gamma}(Y)}\Tr K_X^{\tau,p}(t;x,x)dx=\alpha_{-1}(t)\log
Y+\alpha_0(t)+o(1),
\end{align}
as $Y\to\infty$. Now recall that  on a compact manifold the trace
of the heat operator is given by the integral of the pointwise trace of the
heat kernel. Based on this observation one defines the regularized trace 
$\Tr_{\reg}(e^{-t\Delta_p(\tau)})$ as the constant term of the asymptotic
expansion 
\eqref{asexp}. However, this definition depends on the choice of the
set $\mathfrak{P}_{\Gamma}$ of representatives of $\Gamma$-cuspidal 
parabolic subgroups of $G$ or equivalently on
the choice of a truncation parameter on the manifold $X$, see
\cite[Remark 5.4]{MP2}. Therefore, this definition is not suitable if one wants
to study the regularized trace for families of hyperbolic manifolds.

To overcome this problem, we remark that if  $\pi\colon X_1\to X_0$ is a finite
covering of $X_0$ and if truncation parameters on the manifold $X_0$
are given, then there is a canonical way to truncate the
manifold $X_1$, putting $X_1(Y):= \pi^{-1}(X_0(Y))$. Thus one only has to fix
truncation parameters for the manifold $X_0$ or equivalently a set 
$\mathfrak{P}_{\Gamma_0}$ of representatives of $\Gamma_0$-cuspidal
parabolic subgroups of $G$. To make this approach rigorous, we first
need to discuss some facts about height functions.

Let $\Gamma_0$ be a discrete subgroup of $G$ of finite
covolume. We emphasize that we do not  assume that $\Gamma_0$ is torsion-free. 
Let
$\mathfrak{P}_{\Gamma_0}:=\{P_{0,1},\dots,P_{0,\kappa(X_0)}\}$ be a fixed
set of 
$\Gamma_0$-cuspidal parabolic subgroups of $G$. Each $P_{0,l}$,
$l=1,\dots,\kappa(X_0)$, 
has a Langlands decomposition $P_{0,l}=N_{0,l}A_{0,l}M_{0,l}$. If $P'$ is any
$\Gamma_0$-cuspidal parabolic subgroup of $G$, there exists  
$\gamma'\in\Gamma_0$ and a unique $l'\in\{1,\dots,\kappa(\Gamma_0)\}$ 
such that $\gamma'P'\gamma'^{-1}=P_{0,l'}$. Write 
\begin{align}\label{gammprime}
\gamma'=n_{0,l'}\iota_{P_{0,l'}}(t_{P'})k_{0,l'},  
\end{align}
$n_{0,l'}\in N_{P_{0,l'}}$, $t_{P'}\in(0,\infty)$,
$\iota_{P_{0,l'}}(t_{P'})\in A_{P_{0,l(P')}}$ as above, and $k_{0,l'}\in K$.
Since $P_{0,l'}$ equals its normalizer in $G$, the projection of the element
$\gamma'$ 
to $(\Gamma_0\cap P_{0,l'})\backslash \Gamma_0 $ is unique. Moreover, since
$P_{0,l'}$ is $\Gamma_0$-cuspidal, one has 
$\Gamma_0\cap P_{0,l'}=\Gamma_0\cap N_{P_{0,l'}}M_{P_{0,l'}}$. Thus 
$t_{P'}$ depends only on $\mathfrak{P}_{\Gamma_0}$ and $P'$.

Now we let $\Gamma_1\subset\Gamma_0$ be a subgroup of finite
index. Then a parabolic subgroup $P'$ of $G$ is $\Gamma_0$-cuspidal iff it is
$\Gamma_1$-cuspidal. 
We assume for simplicity that 
$\Gamma_1$ satisfies \eqref{asGamma}. Let 
$X_0=\Gamma_0\backslash\widetilde{X}$,
$X_1=\Gamma_1\backslash\widetilde{X}$. Let $\pi:X_1\to X_0$ be the covering map
and 
let $\pr_{X_0}:G\to X_0$ and $\pr_{X_1}:G\to X_1$ be the corresponding 
projections. 
Let $\mathfrak{P}_{\Gamma_1}=\{P_1,\dots,P_{\kappa(X_1)}\}$ be a set of 
representatives of $\Gamma_1$-cuspidal parabolic subgroups.  
Then for each $j\in\{1,\dots,\kappa(X_1)\}$ let 
$l(j)\in\{1,\cdots,\kappa(\Gamma_0)\}$, $\gamma_j \in\Gamma_0$, and 
$t_j:=t_{P_j}$ be as in \eqref{gammprime} with respect to $P_j$. 
Fix $Y(\Gamma_0)\in (0,\infty)$ such that for 
each $P_{0,l}\in\mathfrak{P}_{\Gamma_0}$, $l=1,\dots,\kappa(\Gamma_0)$, one has 
\begin{align}\label{YGamma}
Y(\Gamma_0)\ge Y_{\Gamma_0}^0(P_{0,l}),
\end{align}
where $Y_{\Gamma_0}^0(P_{0,l})$ is defined by \eqref{infimum}.
Then the following Lemma holds.

\begin{lem}\label{Lemheightfnctn}
For each $P_j\in\mathfrak{P}_{\Gamma_1}$ let $Y^0_{P_j}(\Gamma_1)$ be defined by
\eqref{infimum}. Then one has
\begin{align}\label{eqhf}
Y^0_{P_j}(\Gamma_1)\leq t_j^{-1}Y(\Gamma_0).
\end{align}
Let $X_0(Y):=X_0(P_{0,1},\dots,P_{0,\kappa(X_0)},Y)$. Then for $Y$
sufficiently large one has
\begin{align*}
\pi^{-1}(X_0(Y))=X_1(P_1,\dots,P_{\kappa(X_1)};t_1^{-1}Y,\dots,t_{
\kappa(X_1)}^{
-1}
Y).
\end{align*}
\end{lem}
\begin{proof}
Since $P_{0,l(j)}=k_{0,l(j)}P_jk_{0,l(j)}^{-1}$, and since the adjoint action 
by $k_{0,l(j)}$ is an isometry form the Lie-algebra of $A_{P_j}$ to the 
Lie-algebra of $A_{P_{0,l(j)}}$, it follows that for
every $j=1,\dots,\kappa(X_1)$ and every $g\in G$ one has
\begin{align}\label{eqheightfnctn}
y_{P_{0,l(j)}}(\gamma_jgk_{0,l(j)}^{-1})=t_jy_{P_j}(g). 
\end{align}
This implies \eqref{eqhf}. Indeed, if $g\in G$ and $\gamma\in\Gamma_1$ satisfy
$y_{P_j}(g)>t_j^{-1}Y(\Gamma_0)$ and 
$y_{P_j}(\gamma g)>t_j^{-1}Y(\Gamma_0)$, then by \eqref{eqheightfnctn} and the
choice of $Y(\Gamma_0)$ 
one has $\gamma\in\gamma_j^{-1}(\Gamma_0\cap P_{l(j)})\gamma_j=\Gamma_0\cap
P_{j}$. 

To prove the second part of the lemma, let 
$x\in X_1-X_1(P_1,\dots,P_{\kappa(X_1)};t_1^{-1}Y,\dots,t_{\kappa(X_1)}^{-1}Y)$.
By \eqref{truncmfd} there exists $j\in\{1,\cdots,\kappa(\Gamma_1)\}$ such that
$y_{\Gamma_1,P_j}(x)>t_j^{-1}Y$. Then by \eqref{heightf} there exists $g\in G$ 
satisfying $\pr_{X_1}(g)=x$ and $y_{P_j}(g)> t_j^{-1}Y$. Now observe that 
$\pr_{X_0}(\gamma_j g k_{0,l(j)}^{-1})=x$. Using \eqref{eqheightfnctn} and
\eqref{heightf}, it follows that $y_{\Gamma_0,P_{0,l(j)}}(\pi(x))>Y$, i.e., 
$x\in \pi^{-1}(X_0-X_0(Y))$. Thus we have shown that
\begin{align}\label{preim1}
X_1-X_1(P_1,\dots,P_{\kappa(X_1)};t_1^{-1}Y,\dots,t_{
\kappa(X_1)}^{-1}Y)\subseteq \pi^{-1}(X_0-X_0(Y)).
\end{align}
It remains to prove the opposite inclusion. Fix $l\in\{1,\dots,\kappa(X_0)\}$. 
Since $P_{0,l}$ equals its normalizer in $G$, it follows that
\begin{align}\label{preim}
\#\{P_j\in\mathfrak{P}_{\Gamma_1}\colon
\gamma_jP_j\gamma_j^{-1}=P_{0,l}\}=\#\left(\Gamma_1\backslash\Gamma_0
/\Gamma_0\cap P_{0,l}\right)
\end{align}
and the $\gamma_j$ with $\gamma_jP_j\gamma_j^{-1}=P_{0,l}$ form a set of
representatives of equivalence 
classes in the double coset \eqref{preim}. 
For each $\gamma_j$ with $\gamma_jP_j\gamma_j^{-1}=P_{0,l}$ let
$\mu_{i,j}\in\Gamma_1\backslash \Gamma_0$, $i=1,\dots,r(j)$, be such that the
orbit of $\Gamma_1\gamma_j$ under the action of $\Gamma_0\cap P_{0,l}$ is given 
by the $\Gamma_1\mu_{i,j}$, $i=1,\dots,r(j)$. Then
\begin{align}\label{Frml}
[\Gamma_0:\Gamma_1]=\sum_{\substack{j\in\{1,\dots,\kappa(\Gamma_1)\}\\
\gamma_jP_j\gamma_j^{-1}=P_{0,l}}}r(j).
\end{align}
Write $\mu_{i,j}=\gamma_jp_{i,j}$ with $p_{i,j}\in \Gamma_0\cap P_{0,l} $. 
Choose $Y_{P_j}\in(0,\infty)$, $j=1,\dots,\kappa(X_1)$, such 
that \eqref{FBI} and \eqref{FBII} hold for $\Gamma_1$. Let
$Y>\max\{t_j^{-1}Y_{P_j}\colon j=1,\dots,\kappa(X_1)\}$.
Let $x_0\in X_0-X_0(Y)$. Then there exists a
$P_{0,l}\in\mathfrak{P}_{\Gamma_0}$ such that 
$y_{\Gamma_0,P_{0,l}}(x_0)>Y$. Thus there exists $g_0\in G$ such that
$x_0=\pr_{X_0}(g_0)$ and $y_{P_{0,l}}(g_0)>Y$. By \eqref{eqPara}
one has $y_{P_{0,l}}(p_{i,j}g_0)>Y$. We claim that
\begin{align}\label{eqpreim}
\pi^{-1}(x_0)=\left\{\pr_{X_1}(\gamma_j^{-1}p_{i,j}g_0)\colon
\gamma_jP_j\gamma_j^{-1}=P_{0,l}, \colon i=1,\dots, r(j)\right\}.
\end{align}
Obviously,  each 
$\pr_{X_1}(\gamma_j^{-1}p_{i,j}g_0)$ is contained in $\pi^{-1}(x_0)$. On the
other 
hand, assume that
$\pr_{X_1}(\gamma_j^{-1}p_{i,j}g_0)=\pr_{X_1}(\gamma_{j'}^{-1}p_{i',j'}
g_0)=:x_1$, where
$\gamma_jP_j\gamma_j^{-1}=P_{0,l}=\gamma_{j'}P_{j'}\gamma_{j'}^{-1}$.
By \eqref{eqPara} and
\eqref{eqheightfnctn} one obtains 
\begin{align}\label{eqpreim2}
y_{\Gamma_1,P_j}(x_1)>t_j^{-1}Y>Y_{P_j},\quad
y_{\Gamma_1,P_j'}(x_1)>t_{j'}^{-1}Y>Y_{P_{j'}}. 
\end{align}
Applying \eqref{FBI}, \eqref{FBII} one
obtains $j=j'$ and hence $i=i'$.                                                
Thus, since $\#\{\pi^{-1}(x_0)\}=[\Gamma_0:\Gamma_1]$, \eqref{eqpreim} follows
from \eqref{Frml}.
Applying \eqref{eqpreim} and \eqref{eqpreim2} one obtains
\[
\pi^{-1}(X_0-X_0(Y))\subseteq X_1-X_1(P_1,\dots,P_{\kappa(X_1)}
;t_1^{-1}Y,\dots,t_{
\kappa(X_1)}^{
-1}
Y). 
\]
and together with \eqref{preim1} the lemma follows.
\end{proof}

Let $\Delta_{X_1,p}(\tau)$ be the Laplace
operator on $E_\tau$-valued $p$-forms on $X_1$.
Using the preceding Lemma, we can give an invariant definition 
of the regularized trace of $e^{-t\Delta_{X_1,p}(\tau)}$ provided 
the set $\mathfrak{P}_{\Gamma_0}$ is fixed. 
We fix a set $\mathfrak{P}_{\Gamma_1}$ of representatives of 
$\Gamma_1$-cuspidal parabolic subgroups of $G$. 
Then by Lemma \ref{Lemheightfnctn} we have
\begin{align}\label{Eqxone}
\int_{\pi^{-1}(X_0(Y))}\Tr
K_{X_1}^{\tau,p}(t;x,x)dx=\int_{X_1(P_1,\dots,P_{\kappa(X_1)}
;t_1^{-1}Y,\dots,t_{
\kappa(X_1)}^{
-1}
Y)}\Tr K_{X_1}^{\tau,p}(t;x,x)dx.
\end{align}
Now arguing exactly as in \cite[section 5]{MP2} and applying Lemma
\ref{Lemheightfnctn}, we obtain
\begin{equation*}
\begin{split}
&\int_{\pi^{-1}X_0(Y)}\Tr K^{\tau,p}_{X_1}(t;x,x)dx\,dx=
\sum_{\substack{\sigma\in\hat{M}\\
\left[\nu_p(\tau):\sigma\right]\neq 0}}
\sum_{P_j\in\mathfrak{P}_{\Gamma_1}}\frac{
e^{-t(\tau(\Omega)-c(\sigma))}\dim(\sigma)\log{(t_j^{-1}Y)}}{\sqrt{4\pi t}}
+\sum_j e^{-t\lambda_j}\\
&+\sum_{\substack{\sigma\in\hat{M};\sigma=w_0\sigma\\
\left[\nu_p(\tau):\sigma\right]\neq
0}}e^{-t(\tau(\Omega)-c(\sigma))}\frac{\Tr(\widetilde{\boldsymbol{C}}(\sigma,\nu
,0))}{4}\\
&-\frac{1}{4\pi}\sum_{\substack{\sigma\in\hat{M}\\
\left[\nu:\sigma\right]\neq
0}}\int_{\R}e^{-t\left(\lambda^2+\tau(\Omega)-c(\sigma)\right)}
\Tr\left(\widetilde{\boldsymbol{C}}
(\sigma,\nu,-i\lambda)\frac{
d}{dz}\widetilde{\boldsymbol{C}}(\sigma,\nu,i\lambda)\right)\,d\lambda+o(1),
\end{split}
\end{equation*}
as $Y\to\infty$. Here the $\lambda_j$ in 
the first row are the eigenvalues of $\Delta_{X_1,p}(\tau)$, counted with 
multiplicity. 
It follows that the integral on the left-hand side of
\eqref{Eqxone} admits 
an asymptotic expansion in $Y$ as $Y$ goes to infinity. Note that, since 
the factor
factor $\tau(\Omega)$ comes from equation \eqref{kuga}, the last equation 
coincides with \cite[equation 5.7]{MP2} up to the 
occurrence of the $t_j$'s in the first sum. This occurrence is caused by the
different 
choices of truncation parameters. The appearance
of the $t_j$'s is exactly the reason why the above integral is independent of 
the choice of $\mathfrak{P}_{\Gamma_1}$ and depends only on the choice of
$\mathfrak{P}_{\Gamma_0}$. 

We assume from now on that the set $\mathfrak{P}_{\Gamma_0}$ is fixed. By the
above considerations we are let to the following definition of the 
regularized trace of the heat operator for finite coverings of $X_0$.

\begin{defn}\label{Defregtr}
Let $X_1=\Gamma_1\backslash\widetilde{X}$ be a finite 
covering of $X_0$ and assume that $\Gamma_1$ is torsion free 
and satisfies \eqref{asGamma}. Let $\Delta_{X_1,p}(\tau)$ be the Laplace
operator on $E_\tau$-valued $p$-forms on $X_1$. For any choice of a set
 $\mathfrak{P}_{\Gamma_1}$ of representatives of $\Gamma_1$-cuspidal parabolic 
subgroups we put
\begin{equation}\label{regtrace}
\begin{aligned}
 &\Tr_{\reg;X_1}(e^{-t\Delta_{X_1,p}(\tau)}):=-\sum_{\substack{\sigma\in\hat{M}
\\
\left[\nu_p(\tau):\sigma\right]\neq 0}}
\sum_{P_j\in\mathfrak{P}_{\Gamma_1}}\frac{
e^{-t(\tau(\Omega)-c(\sigma))}\dim(\sigma)\log{(t_j)}}{\sqrt{4\pi t}}
\\ &+\sum_j e^{-t\lambda_j}
+\sum_{\substack{\sigma\in\hat{M};\sigma=w_0\sigma\\
\left[\nu_p(\tau):\sigma\right]\neq
0}}e^{-t(\tau(\Omega)-c(\sigma))}\frac{\Tr(\widetilde{\boldsymbol{C}}(\sigma,\nu
,0))}{4}\\
&-\frac{1}{4\pi}\sum_{\substack{\sigma\in\hat{M}\\
\left[\nu_p(\tau):\sigma\right]\neq
0}}e^{-t(\tau(\Omega)-c(\sigma))}\int_{\R}e^{-t\lambda^2}
\Tr\left(\widetilde{\boldsymbol{C}}
(\sigma,\nu_p(\tau),-i\lambda)\frac{
d}{dz}\widetilde{\boldsymbol{C}}(\sigma,\nu,i\lambda)\right)\,d\lambda,
\end{aligned}
\end{equation}
where the notation is as above.
\end{defn}

If one expresses $\Tr_{\reg;X_1}(e^{-t\Delta_{X_1,p}(\tau)})$ using the
geometric side 
of the trace formula, then it 
becomes again transparent that the summands $\log{t_j}$ compensate
the ambiguity caused by the choice of $\mathfrak{P}_{\Gamma_1}$ so that 
$\Tr_{\reg;X_1}(e^{-t\Delta_{X_1,p}(\tau)})$ depends
only on the choice of $\mathfrak{P}_{\Gamma_0}$. For further details we refer 
the reader to  section \ref{geoms}, in particular to equations \eqref{Trfrml} 
and \eqref{Snew}. 

\section{Exponential decay of the regularized trace for large 
time}\label{secexpdec}
\setcounter{equation}{0}
In this section we 
estimate the regularized trace  for large time and with respect to coverings. 
Let $\Gamma_0$ be a lattice in $G$ and  put 
$X_0=\Gamma_0\backslash\widetilde{X}$. 
Let $X_1=\Gamma_1\backslash \widetilde{X}$ be a finite covering of $X_0$ such 
that $\Gamma_1$ is torsion-free and satisfies \eqref{asGamma}. We assume that 
a set $\mathfrak{P}_{\Gamma_0}$ of representatives of $\Gamma_0$-cuspidal 
parabolic subgroups is fixed. We define the regularized trace according
to Definition \ref{Defregtr}. 
To begin with we establish  the following lemma.
\begin{lem}\label{Lemint}
For every $\sigma\in (0,\infty)$ one has
\begin{align*}
\int_{\R}\frac{\sigma}{\sigma^2+\lambda^2}e^{-t\lambda^2}d\lambda=\sqrt{4\pi
t}\;e^{t\sigma^2}\int_{\sigma}^\infty e^{-tu^2}du. 
\end{align*}
\end{lem}

\begin{proof}
Put
\[
f(\sigma):=\int_{\R}\frac{\sigma}{\sigma^2+\lambda^2}e^{-t\lambda^2}
d\lambda=\int_\R e^{-t\sigma^2\lambda^2}\frac{1}{1+\lambda^2}d\lambda.
\] 
Then 
\begin{align*}
f'(\sigma)=&-2t\sigma\int_\R
e^{-t\sigma^2\lambda^2}\frac{\lambda^2}{1+\lambda^2}
d\lambda=-2t\sigma\left(\int_\R
e^{-t\sigma^2\lambda^2}d\lambda-\int_\R
e^{-t\sigma^2\lambda^2}\frac{1}{1+\lambda^2}d\lambda\right)\\
&=-\sqrt{4\pi t}+2t\sigma f(\sigma).
\end{align*}
The general solution of this differential equation on $(0,\infty)$ is given by
\begin{align*}
y(\sigma)=e^{t\sigma^2}\left(\sqrt{4\pi t}\int_{\sigma}^\infty
e^{-tu^2}du+C\right)
\end{align*}
and since $f$ satisfies $\lim_{\sigma\to\infty}f(\sigma)=0$, the Lemma follows.
\end{proof}

The following proposition is our main result concerning the large time 
estimation of the regularized trace of the heat kernel.

\begin{prop}\label{estimregtr}
Let $\tau$ be such that $\tau_\theta\not\cong\tau$. There exist constants 
$C,c>0$ such that for all  finite covers $X_1$  of $X_0$ one has
\[
\left|\Tr_{\reg;X_1}\left(e^{-t\Delta_{X_1,p}(\tau)}\right)\right|\le 
C e^{-ct}(\Tr_{\reg;X_1}(e^{-\Delta_{X_1,p}(\tau)})+\vol(X_1))
\]
for $t\ge 10$.
\end{prop}
\begin{proof}
Let $\overline{\boldsymbol{C}}(\sigma,\nu_p(\tau),s)$ be as in \eqref{intbar}.
For each $\sigma\in\hat{M}$ one has $c(\sigma)=c(w_0\sigma)$. Thus 
by Lemma \ref{LemC} the last line of \ref{regtrace} can be rewritten as
\[
-\frac{1}{4\pi\dim(\nu_p(\tau))}\sum_{\substack{\sigma\in\hat{M}/W(A)\\
\left[\nu_p(\tau):\sigma\right]\neq
0}}e^{-t(\tau(\Omega)-c(\sigma))}\int_{\R}e^{-t\lambda^2}
\Tr\left(\overline{\boldsymbol{C}}
(\sigma,\nu_p(\tau),-i\lambda)\frac{
d}{dz}\overline{\boldsymbol{C}}(\sigma,\nu,i\lambda)\right)\,d\lambda. 
\]
We have
\[
\Tr\left(\overline{\boldsymbol{C}}
(\sigma,\nu_p(\tau),-i\lambda)\frac{
d}{dz}\overline{\boldsymbol{C}}(\sigma,\nu_p(\tau),i\lambda)\right)=\frac{d}{dz}
\log\det\overline{\boldsymbol{C}}(\sigma,\nu_p(\tau),i\lambda).
\]
Let $\sigma_1,\dots,\sigma_l\in(0,n]$, $n=(d-1)/2$, be the poles of 
$\det \overline{\boldsymbol{C}}(\sigma,\nu_p(\tau),s)$ in the 
half-plane $\Real(s)\ge 0$. Poles occur only if $\sigma=w_0\sigma$. Let
$\eta$ run over the poles of $\det
\overline{\boldsymbol{C}}(\sigma,\nu_p(\tau),s)$ in the
half-plane $\Real(s)<0$, both counted with
multiplicity. For $\sigma\in\hat M$, put
\begin{equation}
\overline{\sigma}=\begin{cases}\sigma,& \sigma=w_0\sigma;\\
\sigma\oplus w_0\sigma,& \sigma\neq w_0\sigma.
\end{cases}
\end{equation}
Let $Y(\Gamma_0)$ be as in \eqref{YGamma}. By Lemma 
\ref{Lemheightfnctn} we have $t_j^{-1}Y(\Gamma_0)\ge Y^0_{P_j}(\Gamma_1)$ for
$j=1,\cdots,\kappa(\Gamma_1)$. Using Theorem \ref{ThrmC} and \eqref{dim1} we get
\begin{align*}
&\frac{1}{\dim(\nu_p(\tau))}
\Tr\left(\overline{\boldsymbol{C}}
(\sigma,\nu_p(\tau),-i\lambda)\frac{
d}{dz}\overline{\boldsymbol{C}}(\sigma,\nu_p(\tau),
i\lambda)\right)\\ =&2\dim(\overline{
\sigma})\left(-\sum_{j=1}^{\kappa(\Gamma_1)}\log{t_j}
+(Y(\Gamma_0)+2)\kappa(\Gamma_1)\right)
\\&+a(\sigma,\nu)+\frac{1}{\dim(\nu_p(\tau))}\left(\sum_{j=1}^l\frac{2\sigma_j}{
\lambda^2+\sigma_j^2}+
\sum_{\eta}\frac{2\Real(\eta)}{(\lambda-\Iim(\eta))^2+\Real(\eta)^2}\right),
\end{align*}
where $a(\sigma,\nu)\in\R$, $a(\sigma,\nu)\leq 0$. 
Let $\sigma_{pp}(\Delta_{X_1,p}(\tau))$ denote the pure point spectrum
of $\Delta_{X_1,p}(\tau)$. Then $\sigma_{pp}(\Delta_{X_1,p}(\tau))$ is the union
of the cuspidal spectrum
$\sigma_{cusp}(\Delta_{X_1,p}(\tau))$
and the residual spectrum $\sigma_{res}(\Delta_{X_1,p}(\tau))$. For a given
eigenvalue  $\lambda\in
\sigma_{pp}(\Delta_{X_1,p}(\tau))$, let $m(\lambda)$ denote its multiplicity.
Put
\begin{align*}
I_1(t,\nu_p(\tau)):=\sum_{\lambda\in\sigma_{cusp}(\Delta_{X_1,p}(\tau))}
m(\lambda)e^{
-t\lambda},
\end{align*}

\begin{equation*}
\begin{split}
I_2(t,\nu_p(\tau)):=&\sum_{\lambda\in\sigma_{res}(\Delta_{X_1,p}(\tau))}
m(\lambda)e^{-t\lambda}\\
&-\frac{1}{2\pi\dim\nu_p(\tau)}\sum_{\substack{\sigma\in\hat
{M};\sigma=w_0\sigma\\ [\nu_p(\tau):\sigma]\neq
0}}\sum_{j=1}^le^{-t(\tau(\Omega-c(\sigma)))}\int_\R
e^{-t\lambda^2}\frac{\sigma_j}{\lambda^2+\sigma_j^2}d\lambda,
\end{split}
\end{equation*}

\begin{equation*}
\begin{split}
I_3(t,\nu_p(\tau)):=-\sum_{\substack{\sigma\in\hat{M}/W(A)\\
[\nu_p(\tau):\sigma]\neq
0}}&e^{-t(\tau(\Omega)-c(\sigma))}\biggl(\frac{a(\sigma,\nu)}{4\sqrt{\pi
t}}+\frac{1}{\sqrt{4\pi t}}\kappa(\Gamma_1)\dim(\bar{\sigma}
)(Y(\Gamma_0)+2)\\
&+\frac{1}{2\pi\dim(\nu_p(\tau))}\int_\R e^{-t\lambda^2}
\sum_{\eta}\frac{\Real(\eta)}{\Real(\eta)^2+(\lambda-\Iim(\eta))^2}\;d\lambda
\biggr),
\end{split}
\end{equation*}
and
\begin{align*}
I_4(t,\nu_p(\tau)):=\sum_{\substack{\sigma\in\hat{M};\sigma=w_0\sigma\\
\left[\nu_p(\tau):\sigma\right]\neq
0}}e^{-t(\tau(\Omega)-c(\sigma))}\frac{\Tr(\widetilde{\boldsymbol{C}}(\sigma,\nu
,0))}{4}.
\end{align*}
Then it follows from \eqref{regtrace} that we have 
\begin{align}\label{eqI0}
\Tr_{\reg;X_1}(e^{-t\Delta_{X_1,p}(\tau)})=I_1(t,\nu_p(\tau))+I_2(t,
\nu_p(\tau))+I_3(t,
\nu_p(\tau))+I_4(t,\nu_p(\tau)).
\end{align}

To estimate $I_1(t,\nu_p(\tau))$ we apply Corollary \ref{Korspec}. It follows
that for $t\geq 2$ we have
\begin{align}\label{eqI1}
|I_1(t,\nu_p(\tau))|\leq e^{-\frac{t}{8}}I_1(1,\nu_p(\tau)).
\end{align}
To deal with $I_2(t,\nu_p(\tau))$ observe that to each
$\lambda_j\in\sigma_{res}(\Delta_p(\tau))$ there correspond 
a $\sigma\in\hat{M}$ satisfying $\sigma=w_0\sigma$ and 
$[\nu_p(\tau):\sigma]\neq 0$, and a pole $\sigma_j$ of 
$\det C(\sigma,\nu_p(\tau),s)$ in $(0,n]$ such that
\begin{align}
\lambda_{j}=-\sigma_j^2+\tau(\Omega)-c(\sigma).
\end{align}
Moreover, the multiplicity of $\sigma_j$ divided by $\dim(\nu_p(\tau))$ equals 
the multiplicity of the eigenvalue $\lambda_j$. 
Let $\mu_j$ be the sequence of the $\sigma_j$'s, where the multiplicity 
of each $\mu_j$ is the multiplicity of $\sigma_j$ divided by
$\dim(\nu_p(\tau))$. Put
\[
h_{\mu_j}(t):=1-\frac{\sqrt{t}}{
\sqrt {\pi}}\int_{ \mu_j}^\infty e^{-tu^2}du=1-\frac{1}{
\sqrt{\pi}}\int_{
\sqrt{t}\mu_j}^\infty e^{-u^2}du.
\]
Using Lemma \ref{Lemint}, we get
\begin{align*}
I_2(t,\nu_p(\tau))=&\sum_{\substack{\sigma\in\hat
{M};\sigma=w_0\sigma\\ [\nu_p(\tau):\sigma]\neq
0}}e^{-t(\tau(\Omega)-c(\sigma))}\sum_{j}\left(e^{t\mu_j^2}-\frac{1}{2\pi}
\int_\R
e^{-t\lambda^2}\frac{\mu_j}{\lambda^2+\mu_j^2}
d\lambda\right)\\
&=\sum_{\substack{\sigma\in\hat
{M};\sigma=w_0\sigma\\ [\nu_p(\tau):\sigma]\neq
0}}\sum_{j}e^{-t(\tau(\Omega)-c(\sigma)-\mu_j^2)}h_{\mu_j}(t). 
\end{align*}
Now observe that $1\geq h_{\mu_j}(t)\geq \frac{1}{2}$. Moreover, by Corollary 
\ref{Korspec} it follows that for every $\mu_j$ we have
\begin{align}\label{eqsigma}
-\mu_j^2+\tau(\Omega)-c(\sigma)\geq\frac{1}{4}.
\end{align}
Thus for each $t\geq 10$ we get
\begin{align}\label{eqI2}
|I_2(t,\nu_p(\tau))|&\leq e^{-\frac{t}{8}}\sum_{\substack{\sigma\in\hat
{M};\sigma=w_0\sigma\\ [\nu_p(\tau):\sigma]\neq
0}}\sum_{j}e^{-\frac{t(\tau(\Omega)-c(\sigma)-\mu_j^2)}{2}}h_{\mu_j}
(t)\nonumber\\
&\leq e^{-\frac{t}{8}}\sum_{\substack{\sigma\in\hat
{M};\sigma=w_0\sigma\\ [\nu_p(\tau):\sigma]\neq
0}}\sum_{j}e^{-(\tau(\Omega)-c(\sigma)-\mu_j^2)} e^{-1}h_{\mu_j}(t)\nonumber\\
&\leq e^{-\frac{t}{8}}\sum_{\substack{\sigma\in\hat
{M};\sigma=w_0\sigma\\ [\nu_p(\tau):\sigma]\neq
0}}\sum_{j}e^{-(\tau(\Omega)-c(\sigma)-\mu_j^2)}
h_{\mu_j}(1)=e^{-\frac{t}{8}}I_2(1,\nu_p(\tau)).
\end{align}

Next we deal with $I_3(t,\nu_p(\tau))$. By \cite[Lemma 7.1]{MP2} we have
\begin{align}\label{LemMP}
\tau(\Omega)-c(\sigma)\geq\frac{1}{4}
\end{align}
for all $\sigma\in\hat{M}$ with $[\nu_p(\tau):\sigma]\neq 0$.  
Then since $a(\sigma,\nu)\leq 0$, $\Real(\eta)<0$, for each 
$t\geq 2$ we can estimate
\begin{align*}
|I_3(t,\nu_p(\tau))|& \leq
e^{-\frac{t}{8}}\sum_{\substack{\sigma\in\hat{M}/W(A)\\
[\nu_p(\tau):\sigma]\neq
0}}e^{-(\tau(\Omega)-c(\sigma))}\dim(\bar{\sigma}
)\kappa(\Gamma_1)(Y(\Gamma_0)+2)\frac{1}{\sqrt{4\pi }}
\\&-e^{-\frac{t}{8}}\sum_{\substack{\sigma\in\hat{M}/W(A)\\
[\nu_p(\tau):\sigma]\neq
0}}e^{-(\tau(\Omega)-c(\sigma))}\biggl(\frac{a(\sigma,\nu)}{4\sqrt{\pi}}\\
&\hskip30pt+\frac{1}{2\pi\dim(\nu_p(\tau))}\int_\R e^{-\lambda^2}
\sum_{\eta}\frac{\Real(\eta)}{\Real(\eta)^2+(\lambda-\Iim(\eta))^2}
d\lambda\biggr)\\
&=2e^{-\frac{t}{8}}\sum_{\substack{\sigma\in\hat{M}/W(A)\\
[\nu_p(\tau):\sigma]\neq
0}}e^{-(\tau(\Omega)-c(\sigma))}\dim(\bar{\sigma}
)(Y(\Gamma_0)+2)\kappa(\Gamma_1)\frac{1}{\sqrt{4\pi}}+e^{-\frac{t}{8}}I_3(1,
\nu_p(\tau)).
\end{align*}
By \cite[Proposition 3.3]{Ke} there exists $C(d)>0$ such that 
\begin{equation}\label{volume}
\kappa(X)\le C(d)\vol(X)
\end{equation}
for all complete hyperbolic manifolds of finite volume and dimension $d$.
Thus for each $t\geq 2$ we obtain
\begin{align}\label{eqI3}
|I_3(t,\nu_p(\tau))|\leq e^{-\frac{t}{8}}(I_3(1,\nu_p(\tau))+C_2\vol(X_1)),
\end{align}
where $C_2$ depends only on $\Gamma_0$ and $\mathfrak{P}_{\Gamma_0}$. 
To estimate $I_4(t,\nu_p(\tau))$ we recall that 
$\widetilde{\boldsymbol{C}}(\sigma,\nu,0)^2=\Id$. Hence there exist natural 
numbers $c_1(\Gamma,\sigma,\nu)$, $c_2(\Gamma,\sigma,\nu)$ such that 
\begin{align*}
c_1(\Gamma_1,\sigma,\nu)+c_2(\Gamma_1,\sigma,\nu)=\dim\left(\boldsymbol{\mathcal
{E}}
(\sigma,\nu)\otimes V_\nu\right)^K=\kappa(X_1)\dim(\sigma),
\end{align*}
and 
\begin{align*}
\Tr(\widetilde{\boldsymbol{C}}(\sigma,\nu
,0))=c_1(\Gamma_1,\sigma,\nu)-c_2(\Gamma_1,\sigma,\nu).
\end{align*}
Using \eqref{LemMP} and \eqref{volume} we obtain for $t\geq 2$:
\begin{align}\label{eqI4}
&|I_4(t,\nu_p(\tau))|\nonumber\\ &\leq
e^{-\frac{t}{8}}(I_4(1,\nu_p(\tau))+2c_2(\Gamma_1,\sigma,\nu))\leq
e^{-\frac{t}{8}}(I_4(1,\nu_p(\tau))+2C(d)\dim(\sigma)\vol(X_1)). 
\end{align}
Combining \eqref{eqI0}, \eqref{eqI1}, \eqref{eqI2}, \eqref{eqI3} and
\eqref{eqI4}, the 
proof of the proposition is complete.
\end{proof}

\section{Geometric side of the trace formula}\label{geoms}
\setcounter{equation}{0}

To study the behaviour of the analytic torsion under coverings we will apply
the trace formula to the regularized trace of the heat operator. In this 
section we recall the structure of the geometric side of 
the trace formula and study the parabolic contribution.

Let the assumptions be the same as at the beginning of the previous section. 
Let $\tau\in\Rep(G)$ and assume that $\tau\neq\tau_\theta$. Let 
$\widetilde E_\tau$ be the homogeneous vector bundle over $\widetilde X=G/K$,
associated to $\tau|_K$, equipped with an admissible Hermitian metric (see  
section \ref{secbol}). Let $\widetilde\Delta_p(\tau)$ be the Laplace operator
on $\widetilde E_\tau$-valued $p$-forms. The on $C^\infty(G,\nu_p(\tau))$ one
has 
\begin{align}\label{FlatH}
\widetilde\Delta_p(\tau)=-\Omega+\tau(\Omega), 
\end{align}
see \cite[(6.9)]{MtM}
Let
\begin{equation}\label{heatker}
H^{\tau,p}_t\colon G\to\End(\Lambda^p\pL^*\otimes V_\tau)
\end{equation}
be the kernel of the heat operator $e^{-t\widetilde \Delta_p(\tau)}$. Let 
\begin{equation}\label{tr-heatker}
h_t^{\tau,p}=\tr H_t^{\tau,p}.
\end{equation} 
We apply the
trace 
formulas in \cite[section 6]{MP2} to express the regularized trace 
as a sum of distributions evaluated at $h_t^{\tau,p}$. The terms appearing on
the geometric side of the trace formula are associated to the different types 
of $\Gamma$-conjugacy classes. We briefly recall their definition.
For further details, we refer the reader to \cite[section 6]{MP2} and 
the references therein. In order to indicate the dependence of the 
distributions on the manifold $X_1$, we shall use $X_1$ as a subscript. 
The contribution of the identity to the trace formula is given by
\begin{equation}\label{identcon}
I_{X_1}(h_t^{\tau,p}):=\vol(X_1) h_t^{\tau,p}(1).
\end{equation}
The hyperbolic contribution is given by
\begin{align}\label{hyperbcon}
H_{X_1}(h_t^{\tau,p}):=\int_{\Gamma_1\backslash
G}\sum_{\gamma\in\Gamma_{1,\s}-\{1\}}
h_t^{\tau,p}(x^{-1}\gamma x)\;dx,
\end{align}
where $\Gamma_{1,\s}$ are the semisimple elements of $\Gamma_1$. 
By \cite[Lemma 8.1]{Warner} the integral converges absolutely. Moreover,
arguing 
as in the cocompact case \cite{Wallach}, if
$G_{\gamma}$ resp. $(\Gamma_1)_\gamma$ 
denote the centralizers of $\gamma$ in $G$ resp. $\Gamma_1$, one has 
\[
H_{X_1}(h_t^{\tau,p})=\sum_{[\gamma]\in\CC(\Gamma_1)_{s}-[1]}
\vol((\Gamma_1)_\gamma
\backslash G_\gamma)\int_{G_\gamma\backslash G} h_t^{\tau,p}(x^{-1}\gamma x)dx,
\]
where $\CC(\Gamma_1)_{\s}$ are the $\Gamma_1$-conjugacy classes 
of semisimple elements of $\Gamma_1$. 
Now the latter sum can also 
be written as a sum over the set $\CC(\Gamma_0)_{s}$ of non elliptic semisimple
conjugacy classes of the group $\Gamma_0$ as 
follows. For each $\gamma\in\Gamma_0$ let $c_{\Gamma_1}(\gamma)$ be the number
of 
fixed points of $\gamma$ on $\Gamma_0 /\Gamma_1$. This number 
clearly depends only on the $\Gamma_0$-conjugacy class of $\gamma$. 
Then if $\Gamma_\gamma$ is the 
centralizer of $\gamma$ in $\Gamma_0$, one has 
\begin{align}\label{Hyperbolic}
H_{X_1}(h_t^{\tau,p}) =\sum_{[\gamma]\in\CC(\Gamma_0)_{\s}-[1]}c_{\Gamma_1}
(\gamma)\vol(\Gamma_\gamma\backslash G_\gamma)\int_{G_\gamma\backslash G}
h_t^{\tau,p}(x^{-1}\gamma x)dx,
\end{align}
see \cite[page 152-153]{Co}. This expression will be used when 
we treat the Hecke subgroups of the Bianchi groups. 

Next we describe the
distributions associated to the parabolic conjugacy 
classes. Firstly let 
\begin{equation}\label{weigpara}
T'_{X_1}(h_t^{\tau,p}):=\kappa(X_1)\int_K\int_N h_t^{\tau,p}(knk^{-1})
\log{\left\|\log n\right\|}\;dkdn.
\end{equation}
We note that $T$ is a non-invariant distribution which depends on $X_1$ only 
via the number of cusps of $X_1$. 
Now let $P'$ be any $\Gamma_0$-cuspidal parabolic subgroup of $G$, or
equivalently a $\Gamma_1$-cuspidal 
parabolic subgroup of $G$. Let $\nL_{P'}$ denote 
the Lie algebra of $N_{P'}$. Then $\exp:\nL_{P'}\to N_{P'}$ is 
an isomorphism and we 
denote its inverse by $\log$. We equip $\nL_{P'}$ with 
the inner product obtained by restriction of the inner 
product in \eqref{metr}. By 
$\left\|\cdot\right\|$ we denote the corresponding norm. 
Let
\begin{align*}
\Lambda_{P'}(\Gamma_1):=\log(\Gamma_1\cap N_{P'});\quad
\Lambda_{P'}^0(\Gamma_1):=\vol(\Lambda_{P'}(\Gamma_1))^{-\frac{1}{2n}}\Lambda_{
P'}(\Gamma_1).
\end{align*}
Then $\Lambda_{P'}(\Gamma_1)$ and $\Lambda_{P'}^0(\Gamma_1)$ are lattices in 
$\nL_{P'}$ and $\Lambda_{P'}^0(\Gamma_1)$ is unimodular. Then for
$\Real(s)>0$
the Epstein-type zeta function $\zeta_{P';\Gamma_1}$, defined by
\begin{align}\label{DefEp}
\zeta_{P';\Gamma_1}(s):=\sum_{\eta\in\Lambda_{P'}(\Gamma_1)-\{0\}}{\left\|\eta
\right\|}^{-2n(1+s)},
\end{align}
converges and $\zeta_{P';\Gamma_1}$ has a meromorphic continuation to $\C$ with
a simple pole at $0$. 
Let $C(\Lambda_{P'}(\Gamma_1))$ be the constant term of $\zeta_{P';\Gamma_1}$ at
$s=0$. Now as before let $\mathfrak{P}_{\Gamma_1}$ be a set 
of representatives of $\Gamma_1$-cuspidal parabolic subgroups and for each
$P_j\in\mathfrak{P}_{\Gamma_1}$ 
let $t_j$ be as in the previous sections. Then put
\begin{align*}
S_{X_1}(h_t^{\tau,p}):=\sum_{P_j\in\mathfrak{P}_{\Gamma_1}}\biggl(&C
(\Lambda_{P_j}(\Gamma_1))\frac{\vol(\Lambda_{P_j}(\Gamma_1))}{\vol(S^{2n-1})}
\sum_{\sigma\in\hat{M}}\frac{\dim(\sigma)}{2\pi}\int_{
\R}\Theta_{\sigma,\lambda}(h_t^{\tau,p})\;d\lambda\\-&\sum_{\substack{
\sigma\in\hat{M}\\
\left[{\nu_p(\tau)}:\sigma\right]\neq 0}}
\frac{e^{-t(\tau(\Omega)-c(\sigma))}\dim(\sigma)\log{(t_j)}}{\sqrt{4\pi
t}}\biggr).
\end{align*} 
Comparing the Definition \ref{Defregtr} and \cite[Definition 5.1]{MP2}, it
follows from \cite[Theorem 6.1]{MP2} that 
\begin{align}\label{Trfrml}
\Tr_{\reg;X_1}(e^{-t\Delta_{X_1,p}(\tau)})=I_{X_1}(h_t^{\tau,p})+H_{X_1}
(h_t^{\tau,p})+T'_{X_1}(h_t^{\tau,p})+S_{X_1}(h_t^{\tau,p}).
\end{align}

We now study the distribution $S_{X_1}(h_t^{\tau,p})$ in more detail. By 
\cite[Proposition 4.1]{MP2} we have
\begin{align*}
\Theta_{\sigma,\lambda}(h_t^{\tau,p})&=e^{-t(\lambda^{2}
+\tau(\Omega)-c(\sigma))}
\end{align*} for $\left[\nu_p(\tau):\sigma\right]\neq 0$, and 
$\Theta_{\sigma,\lambda}(h^{\nu}_{t})=0$ otherwise. Thus we can rewrite
\begin{align*}
S_{X_1}(h_t^{\tau,p}):=&\sum_{\substack{\sigma\in\hat{M}\\
\left[\nu_p(\tau):\sigma\right]\neq 0}}
\frac{
e^{-t(\tau(\Omega)-c(\sigma))}\dim(\sigma)}{\sqrt{4\pi
t}}\left(\sum_{P_j\in\mathfrak{P}_{\Gamma_1}}C(\Lambda_{P_j}(\Gamma_1))\frac{
\vol(\Lambda_{P_j}(\Gamma_1))}{\vol(S^{2n-1})}-\log{(t_j)}\right).
\end{align*}
Let $\Lambda$ be a lattice in $\R^{2n}$. The associated Epstein zeta function
\begin{align}\label{Epstein}
\zeta_{\Lambda}(s):=\sum_{\lambda\in
\Lambda-\{0\}}{\left\|\lambda\right\|}^{-2n(1+s)}
\end{align}
converges for $\Re(s)>0$ and admits a meromorphic extension to $\C$. Let 
$C(\Lambda)$ denote the constant term of the Laurent expansion of 
$\zeta_{\Lambda}(s)$ at $s=0$. The following lemma describes the behaviour of
$C(\Lambda)$ under scaling.
\begin{lem}\label{LemEpstein}
Let $\Lambda$ be a lattice in $\R^{2n}$. Let $\mu\in (0,\infty)$ and put
$\Lambda':=\mu\Lambda$. Then one has
\begin{align*}
C(\Lambda')=\mu^{-2n}\left(C(\Lambda)-\frac{\vol(S^{2n-1})\log\mu}{\vol(\Lambda)
}\right).
\end{align*}
\end{lem}
\begin{proof}
Let $R(\Lambda)$ be the residue of $\zeta_\Lambda$ at $0$. Then one has
\[
C(\Lambda')=\mu^{-2n}(C(\Lambda)-R(\Lambda)2n\log\mu).
\] 
Moreover, by \cite[Chapter 1.4, Theorem 1]{Terras} one has
\begin{align*}
R(\Lambda)=\frac{\vol(S^{2n-1})}{2n\vol(\Lambda)}
\end{align*}
and the lemma follows. 
\end{proof}

Now we let $P'$ be any $\Gamma_0$-cuspidal parabolic subgroup of $G$. Following
section \ref{sectr}, we let
$l'\in\{1,\dots,\kappa(X_0)\}$ such that there exists
$\gamma'\in\Gamma_0$ with $\gamma' P'\gamma'^{-1}=P_{0,l'}$. 
As in \eqref{gammprime} we write
$\gamma'=n_{0,l'}a_{0,l'}(\log{t_{P'}})k_{0,l'}$.
If $\Gamma_1$ is a finite index subgroup of $\Gamma_0$, we 
define a lattice $\tilde{\Lambda}_{P'}(\Gamma_1)$ in $\nL_{P_{0,l(P')}}$ as 
\begin{align*}
\tilde{\Lambda}_{P'}(\Gamma_1):=\log (\gamma'(\Gamma_1\cap N_{P'})\gamma'^{-1}).
\end{align*}
If $\Gamma_1$ is normal in $\Gamma_0$, one has 
$\tilde{\Lambda}_{P'}(\Gamma_1)=\Lambda_{P_{0,l'}}(\Gamma_1)$. 
Since $\gamma'$ is unique in $\Gamma_0/(\Gamma_0\cap P')$ and
$\Gamma_0\cap P'=\Gamma_0\cap (M_{P'}N_{P'})$, the isometry class of
$\tilde{\Lambda}_{P'}(\Gamma_1)$ 
is independent of the choice of  $\gamma'$ having the
required property.  
Let $\hat{\Lambda}_{P'}(\Gamma_1)$ be the unimodular lattice corresponding to
$\tilde{\Lambda}_{P'}(\Gamma_1)$,
i.e. 
\[
\hat{\Lambda}_{P'}(\Gamma_1):=(\vol(\tilde{\Lambda}_{P'}(\Gamma_1))^{-\frac{1}
{2n}}\cdot\tilde{\Lambda}_{P'}
(\Gamma_1).
\]

With respect to the norms induced by the Killing form, 
the lattice $\Lambda_{P'}(\Gamma_1)$ in $\mathfrak{n}_{P'}$ is
isometric to the lattice $t_{P'}^{-1}\tilde{\Lambda}_{P'}(\Gamma_1)$ in
$\mathfrak{n}_{P_{0,l(P')}}$. 
Thus the preceding Lemma implies that
\begin{align*}
\frac{C(\Lambda_{P_j}(\Gamma_1))\vol(\Lambda_{P_j}(\Gamma_1))}{\vol(S^{
2n-1})}=\frac{
C(\tilde{\Lambda}_{P_j}(\Gamma_1))\vol(
\tilde{\Lambda}_{P_{j}}(\Gamma_1))}{\vol(S^{2n-1})}+\log{t_j}.
\end{align*}
Now define 
\begin{align}\label{Defalpha}
\alpha(X_1):=\alpha(\Gamma_1):=\sum_{j=1}^{\kappa(X_1)}\frac{
C(\tilde{\Lambda}_{P_j}(\Gamma_1))\vol(
\tilde{\Lambda}_{P_{j}}(\Gamma_1))}{\vol(S^{2n-1})}.
\end{align}

Then, putting everything together, we can write
\begin{align}\label{Snew}
S_{X_1}(h_t^{\tau,p})
=\alpha(X_1)\sum_{\substack{\sigma\in\hat{M}\\
\left[\nu_p(\tau):\sigma\right]\neq 0}}\frac{
e^{-t(\tau(\Omega)-c(\sigma))}\dim(\sigma)}{\sqrt{4\pi
t}}.
\end{align}

Finally, for each $l=1,\dots,\kappa(\Gamma_0)$, we let
$\mathcal{P}(\nL_{P_{0,l}})$ be the set of isometry classes of unimodular 
lattices in $\nL_{P_{0,l}}$ equipped with the standard topology, i.e., with the
topology induced by identification of $\mathcal{P}(\nL_{P_{0,l}})$ with 
$\SO(2n)\backslash\Sl_{2n}(\R)/\Sl_{2n}(\Z)$.
Now in order to control the constant $\alpha(\Gamma_i)$ for sequences of
finite coverings, we make the following definition.

\begin{defn}\label{Defcuspun}
Let $\Gamma_i$ be a sequence of finite index subgroups of $\Gamma_0$. 
Let $\mathfrak{P}_{\Gamma_0}$ be a fixed set of representatives 
$\Gamma_0$-cuspidal parabolic subgroups of $\Gamma_0$. Then the sequence 
$\Gamma_i$ is called cusp uniform if for each $l=1,\dots,\kappa(\Gamma_0)$ 
there exists a compact set $\mathcal{K}_l$ in  $\mathcal{P}(\nL_{P_{0,l}})$ 
such that for each $\Gamma_0$-cuspidal parabolic $P'$ the lattices
$\hat{\Lambda}_{P'}(\Gamma_i)$, $i\in\mathbb{N}$, belong to $\mathcal{K}_{l}$.
\end{defn}

We can reformulate the condition of cusp-uniformity in a simpler way as
follows. 
We let $\mathcal{P}(\nL)$ be the space of isometry classes of unimodular
lattices in $\nL$, equipped with the topology as above. For each
parabolic subgroup
$P'$ of $G$ there exists a $g_{P'}\in G$ with $g_{P'}P'g_{P'}^{-1}=P$. 
Let $\Gamma$ be a discrete subgroup of $G$ of finite covolume. 
If $P'$ is $\Gamma$-cuspidal, we let 
\begin{align}\label{LambdaP}
\Lambda_{P|P'}(\Gamma):=\vol\left(\log(g_{P'}(\Gamma\cap
N_{P'})g_{P'}^{-1})\right)^{\frac{1}{2n}}\log(g_{P'}(\Gamma\cap
N_{P'})g_{P'}^{-1}).
\end{align}
This a unimodular lattice in $\nL$ and since the image of $g_{P'}$ in
$P\backslash G$ is unique, the isometry class of $\Lambda_{P'}(\Gamma)$ 
is independent of the choice of $g_{P'}$ with $g_{P'}P'g_{P'}^{-1}=P$. 

\begin{lem}\label{Lemcuspun}
The following conditions are equivalent:
\begin{enumerate}
\item The sequence $\Gamma_i$ is cusp-uniform.  
\item For each $\Gamma_0$-cuspidal parabolic subgroup $P'$ of $G$ there exists a
compact set $\mathcal{K}_{P'}$ in $\mathcal{P}(\nL_{P'})$ 
 such that $\Lambda^0_{P'}(\Gamma_i)\in\mathcal{K}_{P'}$ for every
$i$. 
\item There exists a compact set $\mathcal{K}_P$ in $\mathcal{P}(\nL_{P})$ such
that for each $\Gamma_0$-cuspidal parabolic subgroup $P'$ of $G$
one has $\Lambda_{P|P'}(\Gamma_i)\in\mathcal{K}_P$ for each $i\in\mathbb{N}$. 
\end{enumerate}
\end{lem}
\begin{proof}
By the preceding arguments all lattices are isometric. 
\end{proof}

\begin{lem}\label{LemEpsteinII}
Let $\mathcal{K}$ be a compact set of unimodular lattices in $\R^{2n}$. Then
the constant term  of the Laurent expansion of the Epstein zeta functions 
$\zeta_{\Lambda}(s)$ at $s=0$ is bounded on $\mathcal{K}$. 
\end{lem}
\begin{proof}
By \cite[Chapt.I, \S 1.4, Theorem 1]{Terras} the analytic
continuation of $\zeta_\Lambda(s)$ is given by

\begin{align}\label{Eqzeta}
\pi^{-s}\Gamma(s)\zeta_\Lambda(s)=\frac{2}{ns}-\frac{2}{n(1+s)}
+\left(\int_{1}^\infty
(t^{\frac{n}{2}(1+s)-1}+t^{-\frac{n}{2}s-1})\sum_{\lambda\in\Lambda-\{0\}}e^{
-t\pi
\left\|\lambda\right\|^2}\;dt
\right).
\end{align}

Now for a lattice $\Lambda$ in $\R^{2n}$, let $\lambda_1(\Lambda)$ denote the 
smallest norm of a non-zero vector in $\Lambda$. Let $\mathbb{B}(R)$ denote the 
ball in $\R^{2n}$ around the origin of radius $R$. Then it follows from
\cite[Theorem 2.1]{BHW} that for each $R>0$ we have
\[
\#\{ \mathbb{B}(R)\cap \Lambda \}\leq
\left(\frac{2R}{\lambda_1(\Lambda)}+1\right)^{2n}.
\]
If $\mathcal{K}$ is a compact set of unimodular lattices in $\R^{2n}$, then by 
Mahler's criterion there exists a constant $\mu$ such that 
$\lambda_1(\Lambda)\geq \mu$ for each $\Lambda\in\mathcal{K}$. Thus for each 
$\Lambda\in\mathcal{K}$ and for each $t\in[1,\infty)$ we have

\begin{align*}
&\sum_{\lambda\in\Lambda-\{0\}}e^{-t\pi
\left\|\lambda\right\|^2}\leq
e^{-\frac{t\pi\mu^2}{2}}\sum_{\lambda\in\Lambda-\{0\}}e^{-\frac{\pi
\left\|\lambda\right\|^2}{2}}\leq  e^{-\frac{t\pi\mu^2}{2}}\sum_{k=1}^\infty
e^{-\frac{\pi(\mu k)^2}{2}}\#\{\mathbb{B}(\mu (k+1))\cap\Lambda\}\\
&\leq e^{-\frac{t\pi\mu^2}{2}}\sum_{k=1}^\infty
e^{-\frac{\pi(\mu k)^2}{2}}(2k+3)^{2n}=:C_1e^{\frac{-t\pi\mu^2}{2}},
\end{align*}
where $C_1$ is a constant which is independent of $\Lambda$. Applying
\eqref{Eqzeta}, the Lemma follows.
\end{proof}

Now we can control the behaviour of the constants, appearing in the definitions 
of the terms $T'_{X_i}(h_t^{\tau,p})$ and
$\mathcal{S}_{X_i}(h_t^{\tau,p})$, under sequences of coverings
$X_i=\Gamma_i\backslash \widetilde{X}$ of $X_0$. 
As always we assume that a set $\mathfrak{P}_{\Gamma_0}$ of representatives of
$\Gamma_0$-cuspidal  
parabolic subgroups of $G$ is fixed. For  
each $i$ we let  $\mathfrak{P}_{\Gamma_i}=\{P_{i,j},\:
j=1,\dots,\kappa(\Gamma_i)\}$ be a set of representatives of
$\Gamma_i$-conjugacy 
classes of $\Gamma_i$-cuspidal parabolic subgroups. We can estimate
$\alpha(\Gamma_i)$ as follows. 
\begin{prop}\label{Propalpha}
Let $\Gamma_i$ be cusp-uniform  sequence of finite index subgroups of
$\Gamma_0$. Then there exists a constant $c_1(\Gamma_0)$ such
that 
\begin{align*}
|\alpha(\Gamma_i)|\leq
c_1(\Gamma_0)\kappa(\Gamma_i)+c_1(\Gamma_0)\sum_{j=1}^{\kappa(\Gamma_i)}\log[
\Gamma_0\cap N_{P_{i,j}}:\Gamma_i\cap N_{P_{i,j}}].
\end{align*}
In particular, there exists a constant $c_2(\Gamma_0)$ such that we have
\begin{align*}
|\alpha(\Gamma_i)|\leq
c_2(\Gamma_0)\kappa(\Gamma_i)\log{[\Gamma_0:\Gamma_i]}.
\end{align*}

\end{prop}
\begin{proof}
By Lemma \ref{LemEpstein}, for each $P_{i,j}\in\mathfrak{P}_{\Gamma_i}$ one has
\begin{align*}
C(\tilde{\Lambda}
_{P_{i,j}}(\Gamma_i))\vol(\tilde{\Lambda}_{P_{i,j}}(\Gamma_i))=C
(\hat{
\Lambda} _{P_{i,j}}(\Gamma_i))-\frac {\vol(S^{2n-1}
)\log{\vol(\tilde{\Lambda}_{P_{i,j}}(\Gamma_i)}}{2n}.
\end{align*}
By assumption the lattices $\hat{\Lambda}
_{P_{i,j}}(\Gamma_i)$, $i\in\mathbb{N}$, lie in a compact
subset of $\mathcal{P}(\mathfrak{n}_{P_{0,l(j)}})$. Thus by Lemma
\ref{LemEpsteinII} there exists 
a constant $c_1'(\Gamma_0)$ such that for each $i$ one has
$|C(\hat{\Lambda}
_{P_{i,j}}(\Gamma_i))|\leq c_1'(\Gamma_0)$. Since 
$\tilde \Lambda_{P_{i,j}}(\Gamma_0)=\Lambda_{P_{0,l(j)}}(\Gamma_0)$, the lattice
$\tilde{\Lambda}_{P_{i,j}}(\Gamma_i)$ is a sublattice of
$\Lambda_{P_{0,l(j)}}(\Gamma_0)$ of index 
$[\Gamma_0\cap N_{P_{i,j}}:\Gamma_i\cap N_{P_{i,j}}]$. Therefore one has
\begin{align*}
\vol(\tilde{\Lambda}
_{P_{i,j}}(\Gamma_i))=\vol(\Lambda_{P_{0,l(j)}}(\Gamma_0))[\Gamma_0\cap
N_{P_{i,j}}:\Gamma_i\cap N_{P_{i,j}}]\leq  c_1''(\Gamma_0)[\Gamma_0\cap
N_{P_{i,j}}:\Gamma_i\cap N_{P_{i,j}}],
\end{align*}
where $c_1''(\Gamma_0)$ is a constant which is independent of $i$. This proves
the first estimate. The second estimate follows immediately from the first one. 
\end{proof}

In the next proposition we estimate the number of cusps and the behaviour 
of the constant $\alpha(\Gamma_i)$ under sequences of 
normal coverings. 
\begin{prop}\label{Propalpha2}
Let $\Gamma_i$ be a sequence of normal subgroups of $\Gamma_0$ of finite index
$[\Gamma_0:\Gamma_i]$ such that $[\Gamma_0:\Gamma_i]\to\infty$ as 
$i\to\infty$ and such that each $\gamma_0\in\Gamma_0$, $\gamma_0\neq 1$,
belongs 
only to finitely many $\Gamma_i$. Assume that each $\Gamma_i$ satisfies
assumption \eqref{asGamma}. 
Then one has
\begin{align*}
\lim_{i\to\infty}\frac{\kappa(\Gamma_i)}{[\Gamma_0:\Gamma_i]}=0.
\end{align*}
If in addition the sequence $\Gamma_i$ is cusp-uniform, then one has
\begin{align*}
\lim_{i\to\infty}\frac{|\alpha(\Gamma_i)|}{[\Gamma_0:\Gamma_i]}=0.
\end{align*}
\end{prop}
\begin{proof}
Using that each $\Gamma_i$, $i\ge 1$, satisfies
\eqref{asGamma} and $\Gamma_i$ is normal in $\Gamma_0$, one has 
\begin{align*}
\#\{\Gamma_i\backslash\Gamma_0/\Gamma_0\cap
P_{0,l}\}=\frac{[\Gamma_0:\Gamma_i]}{[\Gamma_0\cap P_{0,l}:\Gamma_i\cap
P_{0,l}]}\leq \frac{[\Gamma_0:\Gamma_i]}{[\Gamma_0\cap
N_{P_{0,l}}:\Gamma_i\cap
N_{P_{0,l}}]}
\end{align*}
for each $l=1,\dots,\kappa(\Gamma_0)$.
Thus using \eqref{preim}, one can estimate
\begin{align*}
\frac{\kappa(\Gamma_i)}{[\Gamma_0:\Gamma_i]}=\frac{\sum_{P_{0,l}\in\mathfrak{P}_
{
\Gamma_0}}\#\{\Gamma_i\backslash\Gamma_0/\Gamma_0\cap
P_{0,l}\}}{[\Gamma_0:\Gamma_i]}\leq
\sum_{P_0\in\mathfrak{P}_{\Gamma_0}}\frac{1}{[\Gamma_0\cap
N_{P_{0,l}}:\Gamma_i\cap
N_{P_{0,l}}]}.
\end{align*}

Moreover, for each $l=1,\dots,\kappa(\Gamma_0)$ and each 
$j=1,\dots,\kappa(\Gamma_i)$ one has 
\begin{align*}
\Gamma_0\cap N_{P_{i,j}}=\gamma_{j}(\Gamma_0\cap
N_{P_{0,l(j)}})\gamma_j^{-1},\quad \Gamma_i\cap
N_{P_{i,j}}=\gamma_{j}(\Gamma_i\cap N_{P_{0,l(j)}})\gamma_j^{-1},
\end{align*}
where the second equality is due to the assumption that $\Gamma_i$ is normal in
$\Gamma_0$. 
Thus applying \eqref{preim}, one can estimate
\begin{align*}
&\frac{1}{[\Gamma_0:\Gamma_i]}\sum_{j=1}^{\kappa(\Gamma_i)}\log[
\Gamma_0\cap N_{P_{i,j}}:\Gamma_i\cap N_{P_{i,j}}]\\
&=\frac{\sum_{P_{0,l}\in\mathfrak{P}_{
\Gamma_0}}\#\{\Gamma_i\backslash\Gamma_0/\Gamma_0\cap
P_{0,l}\}\log{[\Gamma_{0}\cap
N_{P_{0,l}}:\Gamma_{i}\cap N_{P_{0,l}}]}}{[\Gamma_0:\Gamma_i]}\\ &\leq
\sum_{P_0\in\mathfrak{P}_{\Gamma_0}}\frac{\log{[\Gamma_{0}\cap
N_{P_{0,l}}:\Gamma_{i}\cap N_{P_{0,l}}]}}{[\Gamma_0\cap
N_{P_{0,l}}:\Gamma_i\cap
N_{P_{0,l}}]}.
\end{align*}

The condition that each $\gamma_0\in\Gamma_0-\{1\}$, $\gamma_0\neq 1$, belongs
only to finitely many $\Gamma_i$ implies that
$[\Gamma_0\cap N_{P_{0,l}}:\Gamma_i\cap
N_{P_{0,l}}]$ goes to $\infty$ as $i\to\infty$. Thus the first statement 
and together with the previous proposition also the second one are proved.
\end{proof}

\section{Proof of the main results}\label{secmainres}
\setcounter{equation}{0}
We keep the assumptions of the previous sections. So $\Gamma_0$ is a lattice in
$G$ and $\Gamma_1$ is a torsion-free subgroup of finite index of $\Gamma_0$,
which satisfies \eqref{asGamma}. We let 
$X_0:=\Gamma_0\backslash\widetilde{X}$ and
$X_i:=\Gamma_i\backslash\widetilde{X}$. We assume that a set 
$\mathfrak{P}_{\Gamma_0}$ of representatives of $\Gamma_0$-conjugacy classes of
$\Gamma_0$-cuspidal parabolic subgroups of $G$ is fixed. Then 
for each $\tau\in\Rep(G)$, $\tau\neq\tau_{\theta}$,  let
$\Tr_{\reg;X_1}(e^{-t\Delta_{X_1,p}(\tau)})$ be the the regularized trace of
$e^{-t\Delta_{X_1,p}(\tau)}$, as defined by \ref{Defregtr}. 
It follows from Proposition \eqref{estimregtr} that there exist constants 
$C,c>0$ such that 
\begin{align}\label{expdec}
\big|\Tr_{\reg; X_1}\left(e^{-t\Delta_{X_1,p}(\tau)}\right)\big|\leq Ce^{-ct},
\end{align}
for $t\geq 1$. Applying [Proposition 6.9]\cite{MP2}, it follows immediately from
the definition of $\Tr_{\reg;X_1}(e^{-t\Delta_{X_1,p}(\tau)})$ 
that there is an asymptotic expansion
\begin{align*}
\Tr_{\reg;X_1}(e^{-t\Delta_{X_1,p}(\tau)})\sim\sum_{j=0}^\infty
a_{j}t^{j-\frac{d}{2}}+\sum_{j=0}^\infty b_{j}t^{
j-\frac{1}{2}}\log{t}+\sum_{j=0}^\infty c_j t^j
\end{align*}
as $t\to +0$. Put
\begin{align*}
K_{X_1}(t,\tau):=\frac{1}{2}\sum_{p=1}^d (-1)^p p \Tr_{\reg;X_1}
\left(e^{-t\Delta_{X_1,p}(\tau)}\right).
\end{align*}
Then it follows that we can define the analytic torsion $T_{X_1}(\tau)$ by 
\[
\log T_{X_1}(\tau)=\frac{d}{ds}\left(\frac{1}{\Gamma(s)}\int_0^\infty
K_{X_1}(t,\tau) t^{s-1}\,dt\right)\Bigg|_{s=0},
\]
where the integral converges in the half-plane $\Re(s)>d/2$ and is defined 
near $s=0$ by analytic continuation. 
Let $T>0$. Then it follows from \eqref{expdec} that $\int_T^\infty
K_{X_1}(t,\tau)
t^{s-1}\;dt$ is an entire function of $s$. Therefore we have
\begin{equation}\label{anator3}
\log T_{X_1}(\tau)=\frac{d}{ds}\left(\frac{1}{\Gamma(s)}\int_0^T
K_{X_1}(t,\tau) t^{s-1}\,dt\right)\bigg|_{s=0}
+\int_T^\infty K_{X_1}(t,\tau) t^{-1}\,dt.
\end{equation}

To proceed further, we first need to estimate the integrand 
of the hyperbolic term \eqref{hyperbcon}. Recall that for a lattice $\Gamma$
in $G$ we denote by $\ell(\Gamma)$ the length of the shortest closed 
geodesic of $\Gamma\backslash\widetilde{X}$.

\begin{lem}\label{Lemhyp}
Let $h_t^{\tau.p}\in C^\infty(G)$ be defined by \eqref{tr-heatker}.
For each $T\in(0,\infty)$ there exists a constant $C>0$, depending on $T$ and
$X_0$ only, such that for all hyperbolic manifolds 
$X_1=\Gamma_1\backslash\widetilde{X}$, 
which are finite coverings of $X_0$, and all $g\in G$ one has
\begin{align*}
\left|\sum_{\gamma\in\Gamma_{1,\s}-\{1\}}
h_t^{\tau,p}(g^{-1}\gamma g)\right|\leq C
e^{-\frac{\ell(\Gamma_0)^2}{32t}}e^{-\frac{\ell(\Gamma_1)^2}{8t}}
\end{align*}
for all $t\in (0,T]$. 
\end{lem}
\begin{proof}
Let $\nu_p(\tau)$ be the representation of $K$ defined by \eqref{nutau}. Let
$\widetilde E_{\nu_p(\tau)}$ be the associated homogeneous vector
bundle over $\widetilde X$ equipped with the canonical 
metric connection \cite[section 4]{MP2}. Let $\widetilde\Delta_{\nu_p(\tau)}$ be
the  
Bochner-Laplace operator acting on 
$C^\infty(\widetilde X,\widetilde E_{\nu_p(\tau)})$. Then 
on $C^\infty(G,\nu_p(\tau))$, the action of this operator is given by
\[
\widetilde\Delta_{\nu_p(\tau)}=-R(\Omega)+\nu_p(\tau)(\Omega_K),
\]
where $\Omega_K$ is the Casimir eigenvalue of $\kL$ with 
respect to the restriction of the normalized Killing form $\gL$ to $\kL$, see
\cite[Proposition 1.1]{Mi1}. Thus 
by \eqref{FlatH} there exists  
an endomorphism $E_p(\tau)$ of $\Lambda^p\pL^*\otimes V_\tau$ such that
\[
\widetilde\Delta_p(\tau)=\widetilde\Delta_{\nu_p(\tau)}+E_p(\tau).
\]
Moreover $E_p(\tau)$ commutes with $\widetilde\Delta_{\nu_p(\tau)}$. Let
\[
H^{\nu_p(\tau)}_t\colon G\to \End(\Lambda^p\pL^*\otimes V_\tau)
\] 
be the kernel of the heat operator 
$e^{-t\widetilde\Delta_{\nu_p(\tau)}}$. Then it follows that
\begin{equation}\label{hkestim}
H_t^{\tau,p}=e^{-t E_p(\tau)}\circ H_t^{\nu_p(\tau)}.
\end{equation}
Let $H^0_t(g)$ be the heat kernel for the Laplacian on functions on 
$\widetilde X$. Using \eqref{hkestim} and \cite[Proposition 3.1]{MP1} it 
follows that there exist constants $C>0$ and $c\in\R$ such that
\[
\| H_t^{\tau,p}(g)\|\le C e^{ct} H^0_t(g),\quad g\in G,\;t>0.
\]
Hence we get
\[
|h_t^{\tau,p}(g)|\le C\dim(\tau)e^{ct} H^0_t(g),\quad g\in G,\;t>0.
\]
By \cite{Do1} there exists $C_1>0$ which depends only on $T$ such 
that for each $t\in (0,T]$ one has 
\[
H_t^0(g)\le C_1' t^{-d/2}\exp\left(-\frac{d^2(gK,K1)}{4t}\right)
\]
for $0<t\le T$. The constant $C_1'$ depends only on $T$.
Thus we get
\begin{equation}\label{sum}
\begin{split}
\sum_{\gamma\in\Gamma_{1,\s}-\{1\}}|h_t^{\tau,p}(g^{-1}\gamma g)|&\le C_2
t^{-d/2}
e^{cT} 
\sum_{\gamma\in\Gamma_{1,\s}-\{1\}} e^{-d^2(\gamma gK,gK)/(4t)}\\
&\le C_3 e^{-\ell(\Gamma_i)^2/(8t)}
e^{-\ell(\Gamma_0)^2/(32t)}\sum_{\gamma\in\Gamma_{0,\s}-\{1\}} e^{-d^2(\gamma
gK,gK)/(16T)},
\end{split}
\end{equation}
where $C_2$, $C_3$ are constants which depend only on $T$. It remains to show
that
the last sum converges and can be estimated independently of $g$ . 
For $r\in(0,\infty)$ and $x\in\widetilde{X}$ we let $B_r(x)$ be the metric ball 
of radius $r$ around $x$. There exists a constant $C>0$ such that 
\begin{align}\label{volgr}
\vol (B_r(x))\leq C e^{2nr}
\end{align}
for all $r\in (0,\infty)$. 
It easily follows from \eqref{FBI} and \eqref{FBII} that there
exists 
an $\epsilon>0$ such that for all $x\in\widetilde{X}$ and all
$\gamma\in\Gamma_{0,\s}$, $\gamma\neq 1$ 
one has $B_\epsilon(x)\cap \gamma B_\epsilon(x)=\emptyset$. Thus for each
$x\in\widetilde{X}$ the union
\begin{align*}
\bigsqcup_{\gamma\in\Gamma_{0,\s}\colon d(x,\gamma x)\leq R}\gamma B_\epsilon(x)
\end{align*}
is disjoint and contained in $B_{\epsilon+R}(x)$. Using \eqref{volgr} it follows
that 
there exists a constant $C_{X_0}>0$, depending on $X_0$, such that  
for all $R\in (0,\infty)$ and all $x\in\widetilde{X}$ 
one has
\begin{align*}
\#\{\gamma\in\Gamma_{0,\s}\colon d(x,\gamma x)\leq R\}\leq C_{X_0}e^{2nR}.
\end{align*}
Applying \eqref{sum} the Lemma follows.
\end{proof}

Applying the preceding lemma we obtain the following estimate for the
regularized
trace which is uniform with respect to coverings.

\begin{prop}\label{estimregtr2}
There exists a constant $C>0$  such that for each hyperbolic manifold 
$X_1=\Gamma_1\backslash\widetilde{X}$, which is a finite
covering of $X_0$, and for which $\Gamma_1$ satisfies \eqref{asGamma}, one has
\begin{align*}
|\Tr_{\reg;X_1}\left(e^{-\Delta_{X_1,p}(\tau)}\right)|\leq
C(\vol(X_1)+\kappa(X_1)+
\alpha(X_1)), 
\end{align*}
where $\kappa(X_1)$ is the number of cusps of $X_1$ and $\alpha(X_1)$ is as in
\eqref{Defalpha}.
\end{prop}
\begin{proof}
We put $t=1$ in \eqref{Trfrml} and estimate the terms on the right hand 
side. The identity contribution \eqref{identcon} can be estimated by $C_1
\vol(X_1)$. By \eqref{weigpara}, the
third term can be estimated by $C_2\kappa(X_1)$. Using \eqref{Snew}, it
follows that the forth term 
is bounded by $C_3\alpha(X_1)$. Finally, \eqref{hyperbcon} and Lemma 
\ref{Lemhyp} imply that the hyperbolic term is bounded by $C_4\vol(X_1)$. The 
constants $C_i>0$, $i=1,\cdots,4$, are all independent of $X_1$. This finishes 
the proof.
\end{proof}

Now we can deal with the second integral in \eqref{anator3}. 
Using Proposition \ref{estimregtr}, Proposition
\ref{estimregtr2}, assumption \eqref{condseq} and Proposition \ref{Propalpha}, 
it follows that there exists a $C,c>0$ such that for all finite coverings
$\pi\colon X_1\to X_0$ as above we have
\begin{equation}\label{largetime1}
\frac{1}{\vol(X_1)}\left|\int_T^\infty K_{X_1}(t,\tau) t^{-1}\,dt\right|\le 
C e^{-cT}
\end{equation}
for all $T\geq 10$. 

It remains to treat the first term on the right hand side of \eqref{anator3}. 
For this purpose we use the geometric side of the trace formula 
as it is given in \eqref{Trfrml}. Therefore, put
\begin{equation}
k_t^\tau:=\frac{1}{2}\sum_{p=1}^d (-1)^p p h_t^{\tau,p}.
\end{equation}
It follows from \cite[section 9]{MP2} that the Mellin transform
$\int_0^\infty k_t^\tau(1) t^{s-1}\;dt$ converges absolutely and uniformly on
compact subsets of $\Re(s)>d/2$, and admits a meromorphic extension to $\C$,
which is holomorphic at $s=0$. Let
\begin{equation}\label{l2-tor2}
t^{(2)}_{\widetilde X}(\tau):=\frac{d}{ds}\left(\frac{1}{\Gamma(s)}\int_0^\infty
k_t^\tau(1) t^{s-1}\,dt\right)\bigg|_{s=0}.
\end{equation}
Then in analogy to the compact case \eqref{l2-tor}, the $L^2$-torsion 
$T_{X_1}^{(2)}(\tau)\in\R^+$ is given by 
\[
\log T_{X_1}^{(2)}(\tau)=\vol(X_1)t^{(2)}_{\widetilde X}(\tau).
\]
For details we refer to \cite[section 9]{MP2}.
Furthermore, it follows from \cite[equation 9.4]{MP2} that there exist $C,c>0$
such that
\[
\left|\int_T^\infty k_t^\tau(1) t^{-1}\,dt\right|\le C e^{-cT}
\]
for $T>0$. Hence we get
\begin{equation}\label{limident}
\frac{d}{ds}\left(\frac{1}{\Gamma(s)}\int_0^T 
I_{X_1}(k_t^\tau)
t^{s-1}\,dt\right)\bigg|_{s=0}=\vol(X_1)\cdot 
(t^{(2)}_{\widetilde X}(\tau)+O\left(e^{-cT}\right)).
\end{equation}
Now let $\Gamma_i$, $i\in\N$, be a sequence of torsion-free subgroups of finite
index of $\Gamma_0$, which satisfy the assumptions of Theorem \ref{Mainthrm}.
Firstly, by \eqref{limident} we have
\begin{equation}\label{limident1}
\lim_{i\to\infty}\frac{1}{[\Gamma_0:\Gamma_i]}
\frac{d}{ds}\left(\frac{1}{\Gamma(s)}\int_0^T I_{X_i}(k_t^\tau)
t^{s-1}\,dt\right)\bigg|_{s=0}=\vol(X_0)\cdot 
(t^{(2)}_{\widetilde X}(\tau)+O\left(e^{-cT}\right)).
\end{equation}
Let $(\Gamma_i)_s$ be the set of semi-simple elements in $\Gamma_i$. By
\eqref{hyperbcon} the hyperbolic contribution is given by
\[
H_{X_i}(k_t^\tau)=\int_{\Gamma_i\backslash G}\sum_{\gamma\in(\Gamma_i)_s-\{1\}}
k_t^\tau(g^{-1}\gamma g)\,d\dot g.
\]
It follows from Lemma \ref{Lemhyp} that 
\[
\frac{d}{ds}\left(\frac{1}{\Gamma(s)}\int_0^T
H_{X_i}(k_t^\tau)t^{s-1}\;dt\right)
\bigg|_{s=0}=\int_0^T H_{X_i}(k_t^\tau) t^{-1}\;dt
\]
and that there exists a constant $C_2$, depending on $T$, such that
\[
\left|\int_0^T H_{X_i}(k_t^\tau) t^{-1}\;dt\right|\le
C_2\vol(X_i)e^{-\frac{\ell(\Gamma_i)^2}{8T}}.
\]
Hence if $\ell(\Gamma_i)\to\infty$ as $i\to\infty$, one has
\begin{equation}\label{limhyperb}
\lim_{i\to\infty}\frac{1}{[\Gamma_0\colon\Gamma_i]}\frac{d}{ds}
\left(\frac{1}{\Gamma(s)}\int_0^T
H_{X_i}(k_t^\tau)t^{s-1}\;dt\right)\bigg|_{s=0}=0.
\end{equation}

Next we study the term associated to $T'_{X_i}(k_t^\tau)$, defined 
in \eqref{weigpara}. We let
$J_{X_i}(k_t^\tau)$ and $\mathcal{I}_{X_i}(k_t^\tau)$ be 
defined according to \cite[(6.13), (6.15)]{MP2}, where the
subindex $X_i$ indicates that these distributions depend
on the manifold $X_i$. Then by definition we have
\begin{align*}
T'_{X_i}(k_t^\tau)=\kappa(X_i)\mathcal{I}_{X_i}(k_t^\tau)+J_{X_i}(k_t^\tau).
\end{align*}
Using the results of \cite[section 6]{MP2}, it follows that there is an
asymptotic expansion 
\begin{align*}
T'_{X_i}(k_t^\tau)\sim
\sum_{k=0}^\infty a_k t^{k-(d-2)/2}+\sum_{k=0}^\infty b_k t^{k-1/2}\log t +c_0
\end{align*}
as $t\to 0$. Thus for $\Real(s)>(d-2)/2$, the integral 
\[
\int_0^T T'_{X_i}(k_t^\tau) t^{s-1}\,dt
\]
converges and has a meromorphic extension to $\C$, which at $s=0$ has at most 
a simple pole. Applying the definition of $T'_{X_i}$ it follows that there 
exists a function $\phi(T,\tau)$ such that 
\[
\frac{d}{ds}\left(\frac{1}{\Gamma(s)}\int_0^T T'_{X_i}(k_t^\tau) t^{s-1}\,
dt\right)\bigg|_{s=0}=\phi(T,\tau)\cdot\kappa(X_i).
\]
Thus if $\lim_{i\to\infty}\kappa(X_i)/[\Gamma_0:\Gamma_i]=0$, we obtain
\begin{equation}\label{limweightint}
\lim_{i\to\infty}\frac{1}{[\Gamma_0\colon\Gamma_i]}
\frac{d}{ds}\left(\frac{1}{\Gamma(s)}\int_0^T T'_{X_i}(k_t^\tau) t^{s-1}\,
dt\right)\bigg|_{s=0}=0.
\end{equation}
Finally, by \eqref{Snew}
the integral
\begin{align*}
\int_{0}^T t^{s-1}\mathcal{S}_{X_i}(k_t^\tau)dt
\end{align*}
converges absolutely for $s\in \C$ with $\Real(s)>\frac{1}{2}$ and has a
meromorphic extension 
to $\C$ with an at most a simple pole at $s=0$. Moreover, it follows from
\eqref{Snew} that there exists a function $\psi(T,\tau)$ such that 
\[
\frac{d}{ds}\left(\frac{1}{\Gamma(s)}\int_0^T
\mathcal{S}_{X_i}(k_t^\tau)t^{s-1}\,
dt\right)\bigg|_{s=0}=\psi(T,\tau)\cdot\alpha(\Gamma_i),
\]
where $\alpha(\Gamma_i)$ is as in \eqref{Defalpha}. By assumption
\eqref{condseq} and Proposition \ref{Propalpha} it follows that 
\begin{align*}
\lim_{i\to\infty}\frac{1}{[\Gamma_0:\Gamma_i]}\frac{d}{ds}\left(\frac{1}{
\Gamma(s)}\int_0^T
\mathcal{S}_{X_i}(k_t^\tau)t^{s-1}\,
dt\right)\bigg|_{s=0}=0.
\end{align*}
Combined with \eqref{limident1}, \eqref{limhyperb}, and \eqref{limweightint}
we get
\begin{equation}\label{limitfirst}
\lim_{i\to\infty}\frac{1}{[\Gamma_0:\Gamma_i]}
\frac{d}{ds}\left(\frac{1}{\Gamma(s)}\int_0^T K_{X_1}(t,\tau) t^{s-1}\,dt\right)
\bigg|_{s=0}=\vol(X_0)\cdot(t^{(2)}_{\widetilde X}(\tau)+O(e^{-cT})).
\end{equation}
Finally, combining  \eqref{limitfirst}, \eqref{anator3} and \eqref{largetime1}, 
and using that $T>0$ is arbitrary, Theorem \ref{Mainthrm} follows.
\hfill $\square$

\bigskip
Now assume that $\Gamma_i$ is normal in $\Gamma_0$ and each 
$\gamma\in\Gamma_0$ belongs only to finitely many $\Gamma_i$. Note that
$\ell(\gamma)$ depends only on the $\Gamma_0$-conjugacy class. Since 
by \eqref{Estgeod}, for each $R>0$  there are only finitely many 
conjugacy classes $[\gamma]\in\CC(\Gamma_{0,\s})$ with $\ell(\gamma)\leq R$, one
has  $\lim_{i\to\infty}\ell(\Gamma_i)=\infty$. Thus, if one applies 
Proposition \ref{Propalpha2} and the preceding arguments, Theorem
\ref{Mainthrm2} follows.

\section{Principal congruence subgroups of $\SO^0(d,1)$}\label{secSO}
\setcounter{equation}{0}
In this section we apply Theorem \ref{Mainthrm2} to 
the case of principal congruence subgroups of $\SO^0(d,1)$ and prove 
Corollary \ref{Korcong1}. Therefore,
throughout this section we let $G:=\SO^0(d,1)$, $d$ odd, $d=2n+1$. Let
$K=\SO(d)$, regarded 
as a subgroup of $G$. Then $K$ is a maximal compact subgroup of $G$. 

We realize the standard parabolic subalgebra $\mathfrak{p}$ of $\mathfrak{g}$ 
as follows. 
Denote by $E_{i,j}$ the matrix in $\gL$ whose entry at the
i-th row 
and j-th column is equal to 1 and all of whose other entries are equal to 0 and
let  
$H_1:=E_{1,2}+E_{2,1}$. Let $\aL:=\R H_1$
and let 
\begin{align}\label{eqn}
\nL=\left\{X(v):=\begin{pmatrix}0&0&v^{t}\\0&0&v^{t}\\
v&-v&0\end{pmatrix},\quad
v\in\R^{d-1}\right\}.
\end{align}
Then for the standard ordering of the restricted roots of $\aL$ in $\gL$, $\nL$
is the direct sum of the 
positive restricted root spaces. 
We let
\[
\gL=\nL\oplus\aL\oplus\kL
\]
be the associated Iwasawa decomposition. Let $N$ be the connected Lie group with
Lie algebra $\nL$ and let $A:=\exp(\aL)$. Let $M$ be
the
centralizer of $A$ in $K$. Then
\begin{align*}
P=MAN
\end{align*}
is a parabolic subgroup of $G$.

For $v\in\R^{d-1}$ one has
\begin{align}\label{EqN}
\exp(X(v))=1+X(v)+\frac{X^2(v)}{2}=\begin{pmatrix}1+\|v\|^2/2&-\|v\|^2/2&v^{t}
\\ \|v\|^2/2&1-\|v\|^2/2&v^{t}\\ v&-v&I_{d-1}\end{pmatrix} ,
\end{align}
where $I_{d-1}$ denotes the unit-matrix and where $\|\cdot\|$ denotes the
Euclidean norm on $\R^{d-1}$. 
We have $N=\exp(\nL)$. 

The group $G$ is an algebraic group defined over
$\Q$ and we let $\Gamma_0:=G(\Z)$ be its integral points. By \cite{BH},
$\Gamma_0$ is a lattice in $G$.
It follows from
\eqref{EqN} that
\begin{align}\label{Gitter1}
\log{(\Gamma_0\cap N)}=\left\{\begin{pmatrix}0&0&v^{t}\\0&0&v^{t}\\
v&-v&0\end{pmatrix},\quad
v\in \Z^{d-1},\quad \|v\|^2\in 2\Z\right\}.
\end{align}
In particular, $P$ is a $\Gamma_0$-cuspidal parabolic subgroup of $G$.

Now for $q\in\mathbb{N}$ we let $\Gamma(q)$ be the principal congruence subgroup
of level $q$, i.e. 
\begin{align*}
\Gamma(q)=\left\{A\in \Gamma_0\colon A\equiv I \:\modo(q)\right\}.
\end{align*}
Then $\Gamma(q)$ coincides with the kernel of the canonical map
$\Gamma_0 \to G(\Z/q\Z)$. In particular, 
$\Gamma(q)$ is a normal subgroup of $\Gamma_0$. 
If $q\geq 3$, then the group $\Gamma(q)$ is neat in the sense of Borel, see
\cite[17.4]{Borel}. In particular, $\Gamma(q)$ is 
torsion free and satisfies \eqref{asGamma}.

In the following Lemma we verify the cusp-uniformity of the
groups $\Gamma(q)$. The Lemma is just a special case of Lemma 4 of the paper
\cite{DH} of Deitmar and Hoffmann who treated the 
more general case of families of strictly bounded depth in algebraic 
$\Q$-groups of arbitrary real rank. However, for the convenience of 
the reader we shall now recall the proof of Deitmar and Hoffmann in our
situation. 
\begin{lem}\label{Cuspun2}
Let $P'$ be a $\Gamma_0$-cuspidal parabolic subgroup defined over $\Q$ with
nilpotent radical $N_{P'}$. Let $\nL_{P'}$ be the Lie-algebra of
$N_{P'}$. Then 
there exists a lattice $\Lambda^{+}_{\nL_{P'}}$ in $\nL_{P'}$ such
that
\begin{align*}
q\Lambda^{+}_{\nL_{P'}}
\subseteq\log{(\Gamma(q)\cap
N_{P'})}\subseteq \frac{q}{4}\Lambda^{+}_{\nL_{P'}}
\end{align*}
for each $q\in\mathbb{N}$. In particular, the sequence $\Gamma(q)$, $q\in\N$, is
cusp-uniform. 
\end{lem}
\begin{proof}
Let $\Mat_{(d+1)\times (d+1)}(\Z)$ be the integral
$(d+1)\times(d+1)$-matrices. Then by \eqref{eqn} $\nL\cap\Mat_{(d+1)\times
(d+1)}(\Z)$
is a lattice in $\nL$.
We choose $g\in G(\Q)$ such that $P'=gPg^{-1}$. Then
$\nL_{P'}=g\nL g^{-1}$ 
 and thus
\[
\Lambda^+_{\nL_{P'}}:=2(\nL_{P'}\cap\Mat_{(d+1)\times
(d+1)}(\Z)) 
\]
is a lattice in $\nL_{P'}$. 
By \eqref{EqN}, one has $\exp(Y)=1+Y+\frac{Y^2}{2}$ for each $Y\in\nL_{P'}$
and thus 
the first inclusion is clear. Moreover, by \eqref{eqn}, if $k\geq 3$ one has
$Y^k=0$ for each $Y\in\nL_{P'}$  and thus for 
each $n_{P'}\in N_{P'}$ one has 
\begin{align*}
\log{n_{P'}}=(n_{P'}-1)-\frac{1}{2}(n_{P'}-1)^2
\end{align*}
and this gives the second inclusion. 
The second statement follows from Mahler's criterion and Lemma \ref{Lemcuspun}. 
\end{proof}
It is obvious that every $\gamma_0\in\Gamma_0$ belongs only to finitely many
$\Gamma(q)$. If we use equation \eqref{Gitter1}, we easily see that
$[\Gamma_0\cap N:\Gamma(q)\cap N]$ 
goes to infinity if $q$ does and so 
$[\Gamma_0:\Gamma(q)]$ goes to infinity if $q\to\infty$.
Thus applying Lemma \ref{Cuspun2}, 
Corollary \ref{Korcong1} follows from Theorem \ref{Mainthrm2}.

\section{Principal congruence subgroups and Hecke subgroups of Bianchi
groups}\label{secSL2}
\setcounter{equation}{0}
We finally turn to the proofs of Corollary \ref{KorCongr2} and Theorem
\ref{ThrmCongr3}. We let $F:=\Q(\sqrt{-D})$,
$D\in\mathbb{N}$ square-free, be an imaginary
quadratic 
number field. Let $\mathcal{O}_D$ be the ring of
integers of $F$, i.e.
$\mathcal{O}_D=\Z+\sqrt{-D}\Z$ if 
$D\equiv 1,2$ modulo $4$, $\mathcal{O}_D=\Z+\frac{1+\sqrt{-D}}{2}\Z$ if
$D\equiv 3$ modulo $4$. We let
$\Gamma(D):=\Sl_2(\mathcal{O}_D)$ be the associated Bianchi-group. Then
$X_D:=\Gamma(D)\backslash\mathbb{H}^3$ 
is of finite volume. More precisely, one has
\begin{align*}
\vol(X_D)=\frac{|\delta_F|^{\frac{3}{2}}\zeta_F(2)}{4\pi^2},
\end{align*}
where $\zeta_F$ is the Dedekind zeta function of $F$ and $\delta_F$ is 
is the discriminant of $F$, see \cite{Hu}, \cite[Proposition 2.1]{Sa}. 
Let $\aL$ be any nonzero ideal in $\mathcal{O}_D$ and let
$N(\aL)$ denote its norm. 
Then the associated principal congruence subgroup $\Gamma(\aL)$ is defined as
\begin{align*}
\Gamma(\aL):=\left\{\begin{pmatrix}a&b\\
c&d\end{pmatrix}\in\Gamma(D)\colon a-1\in\aL; d-1\in\aL;
b,c\in\aL\right\}.
\end{align*}
Moreover, the associated Hecke subgroup $\Gamma_0(\aL)$ is defined as 
\begin{align*}
\Gamma_0(\aL):=\left\{\begin{pmatrix}a&b\\
c&d\end{pmatrix}\in\Gamma(D)\colon c\in\aL\right\}.
\end{align*}
Let $P$ be the parabolic subgroup given by the upper triangular matrices 
in $\Sl_2(\C)$. Then the Langlands decomposition $P=MAN$ is given by
\begin{align*}
M=\left\{\begin{pmatrix}e^{i\theta}&0\\
0&e^{-i\theta}\end{pmatrix},\theta\in[0,2\pi)\right\}
\end{align*}
and 
\begin{align*}
A=\left\{\begin{pmatrix}\lambda&0\\0&\lambda^{-1}\end{pmatrix},\lambda\in\R,
\lambda>0\right\};\quad
N=\left\{\begin{pmatrix}1&b\\0&1\end{pmatrix},b\in\C\right\}.
\end{align*}

We recall that by \cite[Corollary 5.2]{Bass} the canonical map from
$\Sl_2(\mathcal{O}_D)$ to $\Sl_2(\mathcal{O}_D/\aL)$ 
is surjective. Thus the sequence 
\begin{align*}
1\to\Gamma(\aL)\to \Gamma(D)\to\Sl_2(\mathcal{O}_D/\aL)\to
1
\end{align*}
is exact and taking the prime-decomposing of $\aL$ it follows  as in
\cite[Chapter 1.6]{Shi} for the
$\Sl_2(\R)$-case that
\begin{align}\label{Indcong}
[\Gamma(D):\Gamma(\aL)]=N(\aL)^3\prod_{\mathfrak{p}
|\aL}\left(1-\frac
{1}{N(\mathfrak{p})^2}\right).
\end{align}
It also follows that the sequence 
\[
1\to\Gamma(\aL)\to \Gamma_0(\aL)\to P(\mathcal{O}_D/\aL)\to 1 
\]
is exact. Moreover
the order of $P(\mathcal{O}_D/\aL)$ is $N(\aL)\phi(\aL)$, where
\begin{align}\label{phi}
\phi(\aL):=\#\{(\mathcal{O}_D/\aL)^*\}=N(\aL)\prod_{\pL|\aL}\left(1-
N(\pL)^{-1}\right).
\end{align}
Thus one obtains 
\begin{align}\label{Indhecke}
[\Gamma(D):\Gamma_0(\aL)]=N(\aL)\prod_{\pL|\aL}(1+N(\pL)^{-1}).
\end{align}
Here the products in \eqref{Indcong}, \eqref{phi} and \eqref{Indhecke} are taken
over all
prime ideals $\mathfrak{p}$ in
$\mathcal{O}_D$ dividing $\aL$.

Let $\mathbb{P}^1(F)$ be the one-dimensional projective space over $F$. As 
usual, we write $\infty$ for the element $[1,0]\in \mathbb{P}^1(F)$.
Then $\Sl_2(F)$ acts naturally on $\mathbb{P}^1(F)$ and by \cite[Chapter 7.2,
Proposition 2.2]{EGM} one has
\[
\kappa(\Gamma(D))=\#\left(\Gamma(D)\backslash\mathbb{P}^1(F)\right).
\]
Using \cite[Chapter 7.2, Theorem 2.4]{EGM}, it follows that
$\kappa(\Gamma(D))=d_F$, where $d_F$ is the class number of $F$.
The group $P$ is the stabilizer of $\infty$ in $\Sl_2(\C)$. 
For each $\eta\in\mathbb{P}^1(F)$ we fix a $B_\eta\in\Sl_2(F)$ with
$B_\eta\eta=\infty$. We let $B_\infty=\Id$. 
Then $P_\eta:=B_{\eta}^{-1}PB_{\eta}$ is the stabilizer of $\eta$ in
$\Sl_2(\C)$ and the $\Gamma(D)$-cuspidal parabolic subgroups of $G$ are 
given as $P_\eta$. We let 
$N_\eta:=B_{\eta}^{-1}NB_{\eta}$. If $\eta\in\mathbb{P}^1(F)$, we let 
$\Gamma(D)_\eta$, $\Gamma(\aL)_\eta$, $\Gamma_0(\aL)_\eta$ be the stabilizers of
$\eta$ 
in $\Gamma(D)$ resp. $\Gamma(\aL)$ resp. $\Gamma_0(\aL)$. 

The following Proposition is an immediate consequence of the finiteness of 
the class number. 

\begin{prop}\label{PropCuspun}
The set of all principal congruence subgroups $\Gamma(\aL)$ and all Hecke
subgroups $\Gamma_0(\aL)$, $\aL$ a non-zero
ideal in $\mathcal{O}_D$, is cusp-uniform. \end{prop}
\begin{proof}
Let $\mathcal{J}_F$ be the ideal group of $F$, i.e. the group 
of all finitely generated non-zero $\mathcal{O}_D$-modules in $F$. We regard
$F^*$ as a subgroup of $\mathcal{J}_F$ by identifying $F^*$ with the 
group of fractional principal ideals. Let
$\mathcal{I}_F:=\mathcal{J}_F/F^*$ 
be the ideal class group. Then $\#\mathcal{I}_F=d_F<\infty$, see \cite[chapter
I.6]{Neu}. 
Now for $\eta\in\mathbb{P}^1(F)$, $B_\eta$ as above,
write $B_\eta=\begin{pmatrix}\alpha &\beta \\ \gamma &\delta\end{pmatrix}\in
\Sl_2(F)$ and let $\mathfrak{u}$ be the $\mathcal{O}_D$-module 
generated by $\gamma$ and $\delta$ and let
$\bL:=\mathfrak{u}^{-2}\cap\gamma^{-2}\aL$. It is easy to see that $\bL\neq 0$. 
Then proceeding as in \cite[Chapter 8.2, Lemma 2.2]{EGM}, one obtains
\begin{align*}
B_\eta\Gamma(\aL)_{\eta}B_\eta^{-1}\cap N=\left\{\begin{pmatrix}1&\omega'\\
0&1\end{pmatrix};\: \omega'\in
\aL\mathfrak{u}^{-2}\right\};\:B_\eta\Gamma_0(\aL)_\eta B_\eta^{-1}\cap N
=\left\{\begin{pmatrix}1&\omega''\\
0&1\end{pmatrix};\: \omega''\in \mathfrak{b}\right\}.
\end{align*}
Let $P'$ be a $\Gamma(D)$-cuspidal parabolic subgroup of $G$ and let
$\Lambda_{P|P'}(\Gamma(\aL))$ and $\Lambda_{P|P'}(\Gamma_0(\aL))$ denote the set
of lattices defined as in \eqref{LambdaP}.
Since $\aL\mathfrak{u}^{-2}$ and $\mathfrak{b}$ belong to $\mathcal{J}_F$, 
and $\mathcal{I}_F$ is finite, it follows that $\Lambda_{P|P'}(\Gamma(\aL))$,
and $\Lambda_{P|P'}(\Gamma_0(\aL))$ are finite  sets. Applying the third 
criterion of Lemma \ref{Lemcuspun}, the proposition follows. 
\end{proof}

The groups $\Gamma(\aL)$ are torsion-free and satisfy \eqref{asGamma} for
$N(\aL)$ sufficiently 
large. This was shown for example in the proof of Lemma 4.1 in \cite{Pf2}. 
Since
$[\Gamma(D):\Gamma(\aL)]$ 
tends to $\infty$ if $N(\aL)$ tends to $\infty$ and since each
$\gamma_0\in\Gamma(D)$, $\gamma_0\neq 1$, 
is contained in only finitely many $\Gamma(\aL)$, Corollary \ref{KorCongr2} 
follows from Proposition \ref{PropCuspun} and Theorem \ref{Mainthrm2}.  
\newline 

We finally turn to Theorem \ref{ThrmCongr3}. The Hecke groups $\Gamma_0(\aL)$
are never torsion-free and 
never satisfy \eqref{assumGamma}. However, we may take a finite index subgroup 
$\Gamma'$ of $\Gamma_D$, for example a fixed principal congruence subgroup 
of sufficiently high level, which is torsion free and satisfies 
assumption \eqref{assumGamma}. 
Then for each non zero ideal $\aL$ of $\mathcal{O}_D$ we let 
\begin{align*}
\Gamma_0'(\aL):=\Gamma_0(\aL)\cap\Gamma'.
\end{align*}
This group satisfies now the required assumptions and if 
$n_0:=[\Gamma(D):\Gamma']$, then
\begin{align}\label{EqInd}
[\Gamma_0(\aL):\Gamma_0'(\aL)]\leq n_0
\end{align}
for each non-zero ideal $\aL$. 
Thus since the set of all $\Gamma_0(\aL)$ is cusp uniform by the preceding
lemma, also the set of all $\Gamma_0'(\aL)$, $\aL$ a non-zero 
ideal in $\mathcal{O}_D$, is cusp uniform. 
Now, as in \cite[page 15]{AC}, for an ideal $\bL$ of $\mathcal{O}_D$ we let
\[
\phi_u(\bL):=\#((\mathcal{O}_D/\bL)^*/\mathcal{O}_D^*).
\]
Then by \cite[Theorem 7]{AC} one has
\begin{align}\label{kappaH}
\kappa(\Gamma_0(\aL))=d_F\sum_{\bL|\aL}\phi_u(\bL+\bL^{-1}\aL).
\end{align}
Now as in
\cite[Lemma 5.7]{FGT}, on the set of ideals in
$\mathcal{O}_D$, we introduce the 
multiplicative function $\kappa$ given by 
\begin{align*}
\kappa(\pL^k):=\begin{cases}N(\pL)^{\frac{k}{2}}+N(\pL)^{\frac{k}{2}-1}& k\equiv
0 (2),\\ 2N(\pL)^{\frac{k-1}{2}}&  k \equiv 1 (2), \\ \end{cases}
\end{align*}
where $\pL$ is a prime ideal of $\mathcal{O}_D$. Using \eqref{kappaH},
it easily follows that 
\[
\kappa(\Gamma_0(\aL))\leq d_F\kappa(\aL),
\]
where one has equality if one replaces $\phi_u$ by $\phi$ in \eqref{kappaH}. 
Now observe that
\[
\kappa(\aL)\le 2 N(\aL)^{1/2}\prod_{\pL|\aL}\left(1+N(\pL)^{-1}\right).
\]
Using \eqref{Indhecke}, we obtain
\begin{align*}
\frac{\kappa(\Gamma_0(\aL))}{[
\Gamma(D):\Gamma_0(\aL)]}\leq \frac{2 d_F}{\sqrt{N(\aL)}}.
\end{align*}
Now by \eqref{Indhecke} we have the trivial bound
$[\Gamma(D):\Gamma_0(\aL)]\leq N(\aL)^2$. It follows that
\[
\lim_{N(\aL)\to\infty}\frac{\kappa(\Gamma_0(\aL))\log [\Gamma(D):\Gamma_0(\aL)]}
{[\Gamma(D):\Gamma_0(\aL)]}=0.
\]
Thus every sequence
$\Gamma_0(\aL)$ 
satisfies assumption \eqref{condnew} for $N(\aL)\to\infty$. As above, if
$P_{0,1},\dots,P_{0,d_F}$ are 
fixed representatives of $\Gamma(D)$-cuspidal parabolic subgroups of
$\Sl_2(\C)$,
then 
\[
\kappa(\Gamma_0'(\aL))=\sum_{j=1}^{d_F}\#\{
\Gamma_0(\aL)'\backslash\Gamma(D)/\Gamma(D)\cap P_{0,j}\}
\]
and there is a similar formula for $\kappa(\Gamma_0(\aL))$. Thus one has 
$\kappa(\Gamma_0'(\aL))\leq n_0\kappa(\Gamma_0(\aL))$ 
and putting everything together, it follows that the sequence $\Gamma_0'(\aL)$
satisfies condition \eqref{condnew}.
\newline

It remains to prove that the contribution of the semisimple conjugacy classes
to the analytic torsion goes to zero for towers of Hecke subgroups. In order 
to prove this, we consider  the formula \eqref{Hyperbolic}. According to 
section \ref{geoms}, for $\gamma\in\Gamma(D)$ we 
let $c_{\Gamma_0(\aL)}(\gamma)$ be the number of fixed points of $\gamma$ on
$\Gamma(D)/\Gamma_0(\aL)$. To begin with, as in \cite{FGT} we let 
\[
\tilde{\Gamma}(\aL):=\left\{\begin{pmatrix}a&b\\ c&d\end{pmatrix}\colon
a-d\in\aL\colon b,c\in\aL\right\}.
\]

Now we define a multiplicative function $c(\cdot,\cdot)$ on the 
ideals of $\mathcal{O}_D$ by putting  
\[
c(\pL^{k},\pL^{r}):=\begin{cases}N (\pL)^{(k+r)/2}, & \text{$k-r$ odd, $k-r>0$}
\\ 
2 N (\pL)^{(k+r-1)/2}, & \text{$k-r$ even, $k-r>0$}\\
N (\pL)^k + N (\pL)^{k-1},& k\le r,
\end{cases} 
\]
if $\pL$ is a prime ideal and $k,r\in\mathbb{N}^0$. Then the following
proposition and its proof were kindly provided by Tobias Finis. 

\begin{prop}\label{Props}
Let $\gamma\in\Gamma(D)$ and let $\bL$ be the largest divisor of $\aL$ such 
that $\gamma\in\tilde{\Gamma}(\bL)$. Then one has 
\[
c_{\Gamma_0(\aL)}(\gamma)\leq c(\aL,\bL).
\]
In particular, if $\nu(\aL)$ denotes the number of prime divisors of $\aL$, one
can estimate
\[
c_{\Gamma_0(\aL)}(\gamma)\leq 2^{\nu(\aL)}\sqrt{N(\aL)N(\bL)}. 
\]

\end{prop}
\begin{proof}
We can identify the quotient $\Gamma(D)/\Gamma_0(\aL)$ with the projective line
$\mathbb{P}^1(\mathcal{O}_D/\aL)$ and for a given $\gamma\in\Gamma(D)$ we have 
to estimate the number of its fixed points $N(\gamma,\aL)$ on
$\mathbb{P}^1(\mathcal{O}_D/\aL)$. By the 
strong approximation theorem we have
\[
N(\gamma,\aL)=\prod_{\pL}N(\gamma,\pL^{\nu_{\pL}(\aL)}),\quad \aL=\prod_{\pL}
\pL^{\nu_{\pL}(\aL)}.
\]
So it suffices to study $N(\gamma,\pL^k)$ for a prime ideal $\pL$
of 
$\mathcal{O}_D$. First assume that $\gamma$ is scalar modulo $\pL^k$. Then
every point of $\mathbb{P}^1(\mathcal{O}_D/\pL^k)$ is a fixed point of $\gamma$.
The number of elements of the projective line 
$\mathbb{P}^1(\mathcal{O}_D/\pL^k)$ equals $N(\pL)^k+N(\pL)^{k-1}$. Thus in this
case the lemma is proved. Next assume that $\gamma$ is not scalar
modulo $\pL^k$. Let $r<k$ be the maximal integer such that $\gamma$ is scalar 
modulo $\pL^r$. We work over the completion $\mathcal{O}_\pL$ of $\mathcal{O}$ 
at $\pL$. Let $\pi$ be the corresponding 
prime element. Then we have $\mathcal{O}_{\pL}/\pi^l\cong\mathcal{O}/\pL^l$ for 
every $l$. Over $\mathcal{O}_{\pL}$ we have the decomposition
\[
\gamma=a+\pi^r\eta,
\]
where $a$ is a scalar matrix and $\eta$ is not scalar modulo $\pi$. 
A vector $v\in\mathcal{O}_{\pL}^2$ which is not divisible by $\pi$ is an
eigenvector of $\gamma$ modulo $\pi^k$ if and only if it is an eigenvector of
$\eta$
modulo 
$\pi^{k-r}$. If we consider the canonical map
$\mathbb{P}^1(\mathcal{O}/\pL^k)\to\mathbb{P}^1(\mathcal{O}/\pL^{k-r})$, 
then the preimage of each element in $\mathbb{P}^1(\mathcal{O}/\pL^{k-r})$ 
has $N(\pL)^r$ elements. Thus if $n$ denotes the number of eigenvalues 
of $\eta$ in $\mathbb{P}^1(\mathcal{O}/\pL^{k-r})$, we have 
$N(\gamma,\pL^k)=N(\pL)^r n$. It remains to
estimate $n$.

To this end, we may assume that $\eta$ has an eigenvalue. Otherwise there is
nothing to prove. Then adding a scalar 
matrix and performing a base change over $\mathcal{O}_\pL$, which 
does not change the number $n$, we may assume that $\eta$ has the 
eigenvalue $0$ with eigenvector $(1,0)^t$.
Since we assumed that $\eta$ is not scalar modulo $\pi$, after a base
change we may assume that $\eta$ is of the form
\[
\eta=\begin{pmatrix}0&1\\ 0&d\end{pmatrix},
\]
where $d\in\mathcal{O}_{\pL}$. Now a set of representatives of eigenvectors 
in $\mathbb{P}^1(\mathcal{O}/\pL^{k-r})$ of this matrix is given by all classes
of vectors 
represented by $(1,y)$, where $y$ is chosen modulo $\pL^{k-r}$ and satisfies
$y^2-dy\equiv 0$ modulo $\pL^{k-r}$. Thus $n$ is the number of solutions of
the quadratic congruence for $y\in\mathcal{O}/\pL^{k-r}$. 
Let $\nu_{\pL}$ be the valuation corresponding  to $\pL$. Then this 
congruence is equivalent to $\nu_{\pL}(y)+\nu_{\pL}(y-d)\geq k-r$.
This implies that at least one summand is $\geq (k-r)/2$. We distinguish two
cases. 
First, we assume that $\nu_{\pL}(d)<(k-r)/2$. Then exactly 
one summand is $\geq (k-r)/2$ and the other has the valuation $\nu_{\pL}(d)$. 
Thus in this case $n$ is 2 times 
the number of all representatives whose valuation is $\geq k-r-\nu_{\pL}(d)$,
i.e. 
$n=2N(\pL)^{\nu_p(d)}$. 
Secondly, we assume that $\nu_{\pL}(d)\geq(k-r)/2$. Then the
congruence 
is equivalent to 
$\nu_{\pL}(y)\geq (k-r)/2$. Thus in this case one has $n=N (\pL)^{\lfloor
\frac{k-r}{2} \rfloor}$.
In all cases we obtain $n \le N (\pL)^{(k-r)/2}$ if $k-r$ is even and $n \le 2 N
(\pL)^{(k-r-1)/2}$ if $k-r$ is odd. Putting everything together, the 
first estimate follows. This estimate immediately implies the second one. 
\end{proof}

\begin{bmrk}
Proposition \ref{Props} also follows from more general estimates 
which are the content of a paper of Tobias Finis and Erez Lapid 
that is in preparation. Related results are also obtained in \cite{A++}. 
\end{bmrk}

The following Lemma is due to Finis, Grunewald and Tirao.

\begin{lem}\label{LemFGT}
For every $\delta>0$ there is a constant $C>0$ such that for all non zero 
ideals $\bL$ of $\mathcal{O}_D$ and all $R>0$ the number of elements in
$[\gamma]\in \CC(\Gamma(D))_s$ which satisfy 
$\ell(\gamma)\leq R$ and which belong to $\tilde{\Gamma}(\bL)$ is bounded by
$N(\bL)^{-2}e^{(2+\delta)R}$. 
\end{lem}
\begin{proof}
This follows directly from \cite[Lemma 5.10]{FGT}. 
\end{proof}

Now we take a sequence $\aL_i$ of ideals such that $N(\aL_i)$ tends to infinity
with 
$i$ and we let $\Gamma_i:=\Gamma_0'(\aL_i)$,
$X_i:=\Gamma_i\backslash\mathbb{H}^3$. 
We need to estimate the hyperbolic contribution $H_{X_i}(h_t^\tau)$. We use
formula \eqref{Hyperbolic}, and 
apply the Fourier inversion formulas of Harish-Chandra to the invariant orbital
integrals using that the Fourier transform of $h_t^\tau$ can be computed
explicitly. This was carried out in \cite{MP2}. If we combine \cite[(10.4)]{MP2}
for the special case of dimension 3 with equation \eqref{Hyperbolic},
we obtain:
\begin{align}
H_{X_i}(h_t^{\tau})=\sum_{k=0}^1 (-1)^{k+1}
e^{-t\lambda_{\tau,k}^2}\sum_{[\gamma]\in
\CC(\Gamma(D))_s-[1]}c_{\Gamma_i}(\gamma)\frac{\ell(\gamma)}{n_\Gamma(\gamma)
} 
L_{\sym}(\gamma;\sigma_{\tau,k})\frac{e^{-\ell(\gamma)^2/4t}}{(4\pi
t)^{\frac{1}{2}}}.
\end{align}
Here the $\lambda_{\tau,k}\in (0,\infty)$ are as in \cite[(8.4)]{MP2} and
the $\sigma_{\tau,k}\in\hat{M}$ are determined by their 
highest weight $\Lambda_{\sigma_{\tau,k}}$ given as in \cite[(8.5)]{MP2}.
Moreover, $n_\Gamma(\gamma)$ is the period of the closed geodesic corresponding 
to $\gamma$ and 
$L_{\sym}(\gamma;\sigma_{\tau,k})$ is as in \cite[(6.2), (10.3)]{MP2}. By
\cite[(10.11)]{MP2} and the 
definition of $L_{\sym}(\gamma;\sigma_{\tau,k})$, there exists a constant $C_0$
such 
that for all $\gamma\in\Gamma(D)_{s}-\{1\}$ one has
\[
\frac{\ell(\gamma)}{n_\Gamma(\gamma)} |L_{\sym}(\gamma;\sigma_{\tau,k})|\leq
C_0.
\]
Thus together with equation \eqref{EqInd}, Proposition \ref{Props} and Lemma \ref{LemFGT}, 
it follows that there exist constants $C_1, C_2$ such that for each $i$ we can
estimate 
\begin{align*}
&H_{X_i}(h_t^\tau)\leq C_1
2^{\nu(\aL)}\sum_{\bL|\aL}\sqrt{N(\bL)N(\aL)}\sum_{\substack{[\gamma]
\in\CC(\Gamma(D))_{s}-[1] \\ \gamma\in\tilde{\Gamma}(\bL) }}\frac{
e^{-\frac{
\ell(\gamma)^2}{4t}}}{(4\pi t)^{\frac{1}{2}}}\\
&\leq C_1 2^{\nu(\aL)}\sum_{\bL|\aL}\sqrt{N(\aL\bL)}
\sum_{k=1}^\infty\biggl(\frac{e^{-\frac{
(k\ell(\Gamma(D)))^2}{4t}}}{(4\pi t)^{\frac{1}{2}}}\\ &\times
\#\{[\gamma]\in\CC(\Gamma(D))_s\colon\gamma\in\tilde{\Gamma}(\bL)\colon
k\ell(\Gamma(D))\leq\ell(\gamma)\leq (k+1)\ell(\Gamma(D))\}\biggr)\\ 
&\leq C_2
2^{\nu(\aL)}\sqrt{N(\aL)}\sum_{\bL|\aL}N(\bL)^{-\frac{3}{2}}
\sum_{k=1}^\infty\frac{k e^{-\frac{
(k\ell(\Gamma(D)))^2}{4t}}}{(4\pi t)^{\frac{1}{2}}}
e^{(2+\delta)k\ell(\Gamma(D))}.
\end{align*}
Let $\aL=\pL_1^{k_1}\cdot\dots\cdot\pL_{\nu(\aL)}^{k_{\nu(\aL)}}$ be 
the prime ideal decomposition of $\aL$. Then we have 
\[
2^{\nu(\aL)}\sum_{\bL|\aL}N(\bL)^{-\frac{3}{2}}
\leq 2^{\nu(\aL)}\prod_{j=1}^{\nu(\aL)}\frac{1}{1-N(\pL_j)^{-\frac{3}{2}}}\leq
4^{\nu(\aL)}.
\]
Now note that there are only finitely many prime ideals with a given norm.
This implies that for every $\epsilon>0$ there exists $C(\epsilon)>0$ 
such that for all $\aL$ we have $2^{\nu(\aL)}\leq C(\epsilon)N(\aL)^{\epsilon}$.
Hence the right hand side is $O(N(\aL)^\epsilon)$ as $N(\aL)\to\infty$ for any
$\epsilon>0$, where 
the implied constant depends on $\epsilon$.
Thus there exist constants $c, C_3, C_4>0$ such that we have
\begin{align}\label{Esthyp}
H_{X_i}(h_t^\tau)\leq C_3 2^{\nu(\aL)}
\sqrt{N(\aL)}\sum_{\bL|\aL}N(\bL)^{-\frac{3}{2}} e^{-\frac{c}{t}}\leq
C_4N(\aL)^{\frac{3}{4}}e^{-\frac{c}{t}}.
\end{align}
Applying \eqref{Indhecke}, it follows that for ever $T\in(0,\infty)$ one has 
\begin{align}\label{lasteq}
\lim_{i\to\infty}\frac{1}{[\Gamma(D):\Gamma_i]}\int_0^T t^{-1}H_{X_i}(h_t^\tau)
dt=0.
\end{align}

Thus the analog of equation \eqref{limhyperb} is also verified for the present
sequence $\Gamma_i$ of subgroups derived from Hecke subgroups. Since it was
shown above that this 
sequence is cusp uniform and satisfies condition \ref{condnew}, the proof of
Theorem \ref{Mainthrm} given in section \ref{secmainres} can be carried over to
the present case. Thus 
also Theorem \ref{ThrmCongr3} is proved.

\end{document}